\documentclass[11pt]{amsart}  

\usepackage[latin1]{inputenc}
\usepackage{tcolorbox}
\usepackage{longtable} 
\usepackage{multirow} 

\usepackage{amsmath}
\usepackage{amsfonts}
\usepackage{amssymb}
\usepackage{stmaryrd}
\usepackage{latexsym}
\usepackage{graphicx}
\usepackage{subfigure}
\usepackage{color}
\usepackage{mathtools}
\usepackage{verbatim}
\usepackage[all]{xy}
\usepackage{graphics}
\usepackage{pdfsync}
\usepackage{hyperref}
\usepackage{enumitem}

\usepackage{ytableau}
\usepackage{xcolor}
\usepackage{tikz}
\usetikzlibrary{arrows,petri,topaths}
\usepackage{tikz-qtree}
\usetikzlibrary{positioning,arrows,calc,intersections,shapes,decorations,cd}
\usepackage[normalem]{ulem}
\tikzset{
    /tikz/commutative diagrams/crossing over clearance=0.2em,
}

\theoremstyle{plain}
\newtheorem{theorem}{Theorem}[section]
\newtheorem{proposition}[theorem]{Proposition}

\newtheorem{lemma}[theorem]{Lemma}
\newtheorem{corollary}[theorem]{Corollary}

\theoremstyle{definition}
\newtheorem{definition}[theorem]{Definition}
\newtheorem{example}[theorem]{Example}

\theoremstyle{remark}
\newtheorem{remark}[theorem]{Remark}

\numberwithin{equation}{section}
\newcommand{\bC}{\mathbb{C}}
\newcommand{\bN}{\mathbb{N}}

\newcommand{\bP}{\mathbb{Z}_+}
\newcommand{\bQ}{\mathbb{Q}}
\newcommand{\bR}{\mathbb{R}}
\newcommand{\bZ}{\mathbb{Z}}

\newcommand{\Qsym}{\ensuremath{\operatorname{QSym}}}

\newcommand{\sgrp}{\mathbb{S}}

\newcommand{\set}{\mathrm{Set}} 
\newcommand{\des}{\mathrm{Des}} 
\newcommand{\cont}{\mathrm{cont}} 

\newcommand{\suchthat}{\;|\;}

\newcommand{\n}{N} 

\newcommand{\spec}{\ensuremath{\operatorname{sp}}}

\newcommand{\Dsc}{\mathrm{Des}} 

\newcommand{\dsc}{\mathrm{des}} 

\newcommand{\alpx}{\mathbf{x}}

\newcommand{\alpz}{\mathsf{z}}

\newcommand{\ds}[1]{\big\langle {#1}\big\rangle}

\newcommand{\wc}[2]{\mathcal{W}_{#1}^{#2}}

\newcommand{\qbin}[2]{\genfrac(){0pt}{}{#1}{#2}_q}
\newcommand{\qbinom}[2]{\left(\begin{array}{c}#1 \\#2\end{array}\right)_q} 

\newcommand{\sgrpp}{\mathbb{S}'}
\newcommand{\schub}[1]{\mathfrak{S}_{#1}}

\newcommand{\comaj}{\mathrm{comaj}} 
\newcommand{\inv}{\mathrm{inv}} 
\newcommand{\maj}{\mathrm{maj}} 

\newcommand{\rev}[1]{\mathrm{rev}(#1)} 
\newcommand{\mc}[1]{\mathcal{#1}} 
\newcommand{\qsym}{\mathrm{QSym}} 
\newcommand{\sym}{\mathrm{Sym}} 

 \newcommand{\qint}[1]{(#1)_q} 
\newcommand{\qfact}[1]{(#1)_q!} 
\newcommand{\sht}[1]{h(#1)}
\newcommand{\kly}{\mc{K}}
\newcommand{\proj}{\pi} 
\newcommand{\morph}{\Delta} 
\newcommand{\projmorph}{\rho} 
\newcommand{\supp}{\mathrm{Supp}}

\newcommand{\flag}[1]{\mathrm{Flag}({#1})}
\newcommand{\Perm}{\mathrm{Perm}}
\newcommand{\Top}{\mathrm{Top}}
\newcommand{\Hilb}{\mathrm{Hilb}}
\newcommand{\ideal}{{\sf IK}} 
\newcommand{\reduced}{\mathrm{Red}} 
\newcommand{\mca}[1]{\mathcal{#1}} 
\newcommand{\mcbk}{\mca{B}} 
\newcommand{\mcbkp}{\mca{B}_+} 
\newcommand{\mcbkn}{\mca{B}_n} 
\newcommand{\prob}{\mathbb{P}} 
\newcommand{\abb}{\mathcal{A}} 

\begin{document}

\title[The Klyachko algebra]{A $q$-deformation of an algebra of Klyachko\\ and Macdonald's reduced word formula }

\author{Philippe Nadeau}
\address{Univ Lyon, Universit\'e Claude Bernard Lyon 1, CNRS UMR
5208, Institut Camille Jordan, 43 blvd. du 11 novembre 1918, F-69622 Villeurbanne cedex, France}
\email{\href{mailto:nadeau@math.univ-lyon1.fr}{nadeau@math.univ-lyon1.fr}}

\author{Vasu Tewari}
\address{Department of Mathematics, University of Hawaii at Manoa, Honolulu, HI 96822, USA}
\email{\href{mailto:vvtewari@math.hawaii.edu}{vvtewari@math.hawaii.edu}}

\begin{abstract}
There is a striking similarity between Macdonald's reduced word formula and the image of the Schubert class in the cohomology ring of the permutahedral variety $\mathrm{Perm}_n$ as computed by Klyachko. Toward understanding this better, we undertake an in-depth study of a $q$-deformation of the $\sgrp_n$-invariant part of the rational cohomology ring of $\mathrm{Perm}_n$, which we call the $q$-Klyachko algebra.
We uncover intimate links between expansions in the basis of squarefree monomials in this algebra and various notions in algebraic combinatorics, thereby connecting seemingly unrelated results by finding a common ground to study them. Our main results are as follows.
\begin{itemize}
\item A $q$-analog of divided symmetrization ($q$-DS) using Yang-Baxter elements in the Hecke algebra. It is a linear form that picks up coefficients in the squarefree basis.
\item A relation between $q$-DS and the ideal of quasisymmetric polynomials involving work of Aval--Bergeron--Bergeron.
\item A family of polynomials in $q$ with nonnegative integral coefficients that specialize to Postnikov's mixed Eulerian numbers when $q=1$. We refer to these new polynomials as remixed Eulerian numbers. For $q>0$, their normalized versions occur as probabilities in the internal diffusion limited aggregation (IDLA) stochastic process.
\item A lift of Macdonald's reduced word identity in the $q$-Klyachko algebra.
\item The Schubert expansion of the Chow class of the standard split Deligne--Lusztig variety in type $A$, when $q$ is a prime power.
\end{itemize}
\end{abstract}

\maketitle

\tableofcontents

\section{Introduction}
\label{sec:intro}

The permutahedral variety in type $A$ has garnered a lot of attention recently, especially since the study of its cohomology ring plays a key role in the dramatic recent resolutions of various questions concerning log-concavity of various polynomials arising in combinatorics; see for instance \cite{Ber20,HK12}.

Let $\flag{n}$ denote the (type $A$) complete flag variety.
Denote by $\Perm_n\subset \flag{n}$ the closure of a torus orbit of a general point in $\flag{n}$. Equivalently, $\Perm_n$ can be described as the smooth projective toric variety described by the braid fan. It is the \emph{permutahedral variety of type $A$}.
The inclusion $\Perm_n\subset \flag{n}$ induces a pullback $i^*:H^*(\flag{n}, \bQ) \to H^*(\Perm_n,\bQ)$ between cohomology rings. In fact the image of $i^*$ coincides with the algebra of invariants $H^* (\Perm_n,\bQ)^{\sgrp_n}$.

It is this $\sgrp_n$-invariant subring of the cohomology ring, which is also isomorphic to the cohomology ring of the Peterson variety \cite{AbeHar19}, that was first studied by Klyachko \cite{Kly85}.
In particular he obtained a presentation for this ring of invariants, and this presentation has subsequently been rederived in \cite{Fuk15,HHM15}. This ring of invariants has dimension $2^{r}$ where $r=n-1$ as shown by Klyachko \cite{Kly85} and later by Stembridge \cite[\S 3]{Ste94}.
The dimension hints at the existence of a nice combinatorial linear basis for this space, and indeed, Klyachko constructs a basis indexed by subsets of $\{1,\dots,r\}$. We note that Klyachko's results work in all Lie types and we have emphasized the type $A$ picture here.

Recall that the cohomology ring $H^*(\flag{n},\bQ)$ has a basis given by Schubert classes $\sigma_w$ as $w$ ranges over elements in $\sgrp_n$, which brings us to our point of departure.
It is natural to inquire about the image of $\sigma_w$ under $i^*$.
Klyachko expresses $i^*(\sigma_w)$ as a sum of monomials indexed by reduced words of $w$; see \cite[\S8]{NT20} where we reproduce his proof\footnote{The original is in Russian.} with mild modifications.
Particularly remarkable is the fact that Klyachko's expression bears an unmistakable resemblance to Macdonald's reduced word formula \cite{Macdonald}, which we now recall. Let $\reduced(w)$ denote the set of reduced words for  a permutation $w\in \sgrp_n$. Macdonald \cite{Macdonald} established that for $w$ with length $\ell$ we have
\begin{align}
\label{eq:macd}
\schub{w}(1,\dots,1)=\frac{1}{\ell!}\sum_{a_1\dots a_{\ell}\in\reduced(w)}a_1\cdots a_{\ell},
\end{align}
where $\schub{w}$ denotes the Schubert polynomial indexed by  $w$. These polynomials are certain distinguished representatives of Schubert classes and were introduced by Lascoux and Sch\"utzenberger.
Note that the left-hand side counts the number of reduced pipe dreams for $w$.
We note here that Macdonald's reduced word identity has already been given several proofs \cite{Bil19,GT20,Ham20}, each shedding new light.
More importantly for us, a $q$-analogue (conjectured by Macdonald) was first proved by Fomin and Stanley \cite{FS94} by working in the NilCoxeter algebra; see also \cite{MP21}.

To make matters more intriguing, we note that Klyachko did not have access to the combinatorial formula for Schubert polynomials established by Billey-Jockusch-Stanley \cite{Bil93} and Fomin-Stanley \cite{FS94} as it came later. Furthermore, from the perspective of Klyachko's work, there is \emph{a priori} nothing that suggests that a particular representative of a Schubert class may be more important than others, whereas Macdonald's identity clearly deals with a specific representative.

This article stems from our desire to explain this resemblance between Klyachko's result and Macdonald's reduced word identity, as this appears to have been overlooked in literature as far as we can tell.
With the benefit of hindsight, we see that Klyachko's result foreshadows Macdonald's result.
In fact, given that Macdonald's result permits a $q$-analogue, we are able to generalize Klyachko's result by working a parameter $q$ into the setup.

\smallskip

Let $\mathbf{k}$ be a field of characteristic zero, and $q\in \mathbf{k}$ not equal to a nontrivial root of unity.
We introduce an algebra $\kly$, called the \emph{Klyachko algebra} henceforth, by considering a natural $q$-deformation of Klyachko's presentation \cite{Kly85}.
More precisely $\kly$ is the commutative algebra over $\mathbf{k}$ on the generators $\{u_i\suchthat i\in \mathbb{Z}\}$ subject to the following relations for $i\in\bZ$:
\begin{align}\label{eq:klyachko_presentation}
(q+1)u_i^2=qu_iu_{i-1}+u_iu_{i+1}.
\end{align}
The algebra that results on setting $u_i=0$ for $i\geq n$ and $i<0$, will be called $\kly_n$. In this case, when we additionally set $q=1$, the relations in ~\eqref{eq:klyachko_presentation} descend to Klyachko's presentation of $H^* (\Perm_n,\bQ)^{\sgrp_n}$~\cite{Kly85}.
We will be particularly  interested in the `positive half' $\kly_+$ of $\kly$ obtained by setting $u_i=0$ for all $i\leq 0$.
\medskip

\subsection{Discussion of main results}\label{sub:main_res}

The first half of the article focuses on the interplay between an algebraic operation that we call $q$-divided symmetrization, and the coefficients when considering expansions in the distinguished basis $\{u_I\}$ of $\kly$ comprising squarefree monomials indexed by finite subsets $I$ of $\bZ$.
The algebras $\kly_+$ and $\kly_n$ inherit bases of squarefree monomials by omitting terms involving $u_i$ where $i\notin \bZ_+$ or $i\notin \{1,\dots,n-1\}$ respectively.
The  $q$-divided symmetrization operator is a linear form that takes a polynomial $f\in \mathbf{k}[x_1,\dots,x_n]$ of degree $n-1$ as input and output $\ds{f}_n^q \in \mathbf{k}$; see \S \ref{sec:qDS} for the precise definition.
When $q=1$, $q$-divided symmetrization specializes to Postnikov's divided symmetrization \cite[\S 3]{Pos09}, and our lift naturally involves substituting Hecke operators $T_i$ acting on rational functions in $x_1,\dots,x_n$ as opposed to ordinary transpositions $s_i$ in Postnikov's symmetrization.

Fix $r\geq 1$. Our first main result is the following:
\begin{theorem}
\label{th:main_2}
Let $f$ be a polynomial in $\mathbf{k}[\alpx_{r+1}]$  with $\deg(f)=r$.
Consider the expansion of $f(u_1,u_2-u_1,\dots, u_{r}-u_{r-1},-u_{r})$ in the squarefree basis of $\kly_{r+1}$, and let $\Top_{r+1}(f)$ be the top coefficient in this expansion, namely that of $u_1\cdots u_r$. Then
\[\ds{f}^q_{r+1}= \qfact{r} \times \Top_{r+1}(f). \]
\end{theorem}
\noindent We are in fact able to extend Theorem~\ref{th:main_2} to a statement involving coefficients of arbitrary squarefree monomials occuring in expansions in the full Klyachko algebra $\kly$; see Theorem~\ref{th:expansion_f}.
We proceed to describe our results borne out of this interplay between the Klyachko algebra and $q$-divided symmetrization, with a particular focus on combinatorially pertinent families of polynomials.
\smallskip

Our second main result concerns a family of polynomials that we call remixed Eulerian numbers.
While our original definition describes them as $q$-divided symmetrizations of special polynomials $y_c$, the link to the Klyachko algebra in Theorem~\ref{th:main_2} permits an alternative characterization that we present next.
To state it we need some notation.
For a fixed positive integer $r$, let $\wc{r}{r}$ be the set of sequences of nonnegative integers $(c_1,\dots,c_r)$ whose parts sum to $r$.
We let $\qfact{r}$ denote the usual $q$-factorial of $r$.
\begin{theorem}
\label{th:main_1}
There exists a \emph{unique} family of polynomials $\{A_c(q)\}_{c\in\wc{r}{r}}$ with $A_{(1^{r})}(q)=\qfact{r}$ that satisfies the following relations: for any $i=1,\ldots,r$ such that $c_i\geq 2$,
		\begin{equation*}
		(q+1)A_{(c_1,\dots,c_{r})}(q)= qA_{(c_1,\dots,c_{i-1}+1,{c_i}-1,\dots,c_{r})}(q)+A_{(c_1,\dots,c_i-1,c_{i+1}+1,\dots,c_{r})}(q).
		\end{equation*}
If $i=1$ \textup{(}resp. $i=r$\textup{)}, then the first \textup{(}resp. second\textup{)} term on the right hand side is to be omitted.
\end{theorem}

It follows easily from \cite{NT20,Pet18} that $A_c(1)=A_c$ where the latter are the mixed Eulerian numbers introduced by Postnikov \cite{Pos09}. Furthermore, for $c\in \wc{r}{r}$, the following properties hold:
\begin{enumerate}
\item $A_c(q)\in \bN[q]$.
\item $A_c(q)$ is the probability to end in the configuration $(1^{r})$ starting from the configuration $c$ in the process described in \S\ref{sec:remixed}.
\item  Given $c\in \wc{r}{r}$, define $\widetilde{c}\coloneqq(c,0)$. Let $\mathsf{Cyc}(c)$ be the set of all $c'$ such that $(c',0)$ is a cyclic rotation of $\widetilde{c}$. We have the following cyclic sum rule:
\[
\sum_{c'\in \mathsf{Cyc}(c)}A_{c'}(q)=\qfact{r}.
\]
\end{enumerate}

We postpone further study of these polynomials to a separate article \cite{NT21_remixed}, and proceed to make several remarks.
It is worth noting that the nonnegativity of the coefficients of $A_c(q)$ does not immediately follow from the unique characterization above. One needs to derive a better recursion, which in turn implies a nontrivial divisibility property; see \S\S\ref{sub:positivity rmen}.
If one considers $A_c(q)$ where the $c$ are drawn from $\mathsf{Cyc}((r,0^{r-1}))$, then one obtains a familiar family of polynomials:  they are the MacMahon-Carlitz $q$-Eulerians, i.e. the $q$ tracks the major index of  permutations on $r$ letters with a fixed number of descents.
Note furthermore that there are $\mathrm{Cat}_r\coloneqq \frac{1}{r+1}\binom{2r}{r}$ cyclic classes determined by compositions in $\wc{r}{r}$, and we get a Mahonian distribution for each class.
\smallskip

Let $\Qsym$ denote the ring of quasisymmetric functions
in $x_1,x_2,\dots$. We may truncate quasisymmetric functions to obtain quasisymmetric polynomials by declaring $x_i=0$ for $i>m$ for some choice of positive integer $m$.
We denote by $\qsym_m^+$ the ideal in $\bQ[x_1,\dots,x_m]$ generated by positive degree quasisymmetric polynomials.
Of particular interest to is the quotient $\bQ[x_1,\dots,x_m]/\qsym_m^{+}$, studied in the work of Aval--Bergeron--Bergeron \cite{ABB04}.
These authors established that this quotient has dimension $\mathrm{Cat}_m$ and possesses a monomial basis (henceforth the ABB basis) naturally indexed by lattice paths subject to certain constraints.
Since we care particularly about degree $m-1$ polynomials in the context of $q$-divided symmetrization, we note here that in the top degree of $\bQ[x_1,\dots,x_m]/\qsym_m^{+}$, the ABB monomial basis is identified with Dyck paths from $(0,0)$ to $(m-1,m-1)$, as the exponent vectors of these monomials give `prototypical' Catalan objects.

Principal specializations of truncations of quasisymmetric functions are important in combinatorics especially in the context of the study of $P$-partitions.
Our third main result demonstrates that $q$-divided symmetrization, computations in $\kly$, and the ABB basis fit  nicely in this context.

\begin{theorem}
\label{th:main_3}
Let $f$ be a quasisymmetric function of degree $r>0$.
\begin{enumerate}
\item We have
\begin{align}
\label{eq:main_1.3}
\sum_{j\geq 1}f(1,q,\dots,q^{j-1})t^j=\frac{\sum_{m}\ds{f(x_1,\dots,x_m)}_{r+1}^qt^m}{(t;q)_{r+1}}.
\end{align}
\item The expression $f(u_1,u_2-u_1,\dots,u_{r}-u_{r-1},-u_r)$ vanishes in $\kly_{r+1}$.
\item $\ds{\cdot}_{r+1}^q$ vanishes on the space of degree $r$ polynomials in $\qsym_{r+1}^+$.
\end{enumerate}
\end{theorem}

If $f=g+h$ where $h\in \qsym_{r+1}^+$ and $g$ is expressed in the basis of ABB monomials, then
\[
\ds{f}_{r+1}^q=g(1,q,q^2,\dots,q^r).
\]
In particular for $\alpx^c$ an ABB monomial of degree $r$ we have
\[
\ds{\alpx^c}_{r+1}^q=q^{\mathrm{coarea}(\mathsf{Dyck}(c))}.
\]
Here $\mathsf{Dyck}(c)$ is the Dyck path corresponding to $c$, and coarea refers to the well-known \emph{complementary area} statistic.

Let $\sym_n^+$ denote the ideal in $\bQ[x_1,\dots,x_n]$ generated by positive degree symmetric polynomials in $x_1,\dots,x_n$. The quotient $\bQ[x_1,\dots,x_n]/\sym_n^+$ is the well-known coinvariant algebra, which is known to be isomorphic to $H^{*}(\flag{n},\bQ)$ by  Borel's isomorphism under which the Schubert classes can be regarded as Schubert polynomials.
Clearly $\sym_n^+ \subset \qsym_n^+$.
By Theorem~\ref{th:main_3} we have that $\ds{\cdot}_n^q$ vanishes on degree $n-1$ polynomials in $\sym_n^+$.

It is thus natural to inquire if anything substantial can be said about the $q$-divided symmetrization of Schubert polynomials of the appropriate degree. It turns out that the  $\ds{\schub{w}}_{n}^q$ for $w\in \sgrp_{n}$ of length $n-1$ possess a geometric meaning in the setting of Deligne--Lusztig varieties, thanks to the work of D.~Kim \cite{Kim20}.
Given the link between $q$-divided symmetrization and $\kly$ in Theorem~\ref{th:main_2}, we are thus led to consider expansions for Schubert polynomials in $\kly$.
Fortuitously, our initial motivation to unify Klyachko's identity and Macdonald's identity aligns with this new endeavour.

This brings us to our fourth main result.
We will need some notions attached to reduced words.
Suppose $w\in \sgrp_n$ has length $\ell$.
Let $\mathbf{a}=a_1\cdots a_{\ell}\in \reduced(w)$.
We define $\comaj(\mathbf{a})$ to be the sum of all $i\in \{1,\dots,\ell-1\}$ such that $a_{i}<a_{i+1}$.
\begin{theorem}
\label{th:main_4}
\textup{(}q-Klyachko-Macdonald identity \textup{)}
Let $w\in \sgrp_n$ of length $\ell$.
The following equality holds in $\kly_+$.
\begin{align*}
\schub{w}(u_1,u_2-u_1,\dots)=\frac{1}{\qfact{\ell}}\sum_{\mathbf{a}\in \reduced(w)} q^{\comaj(\mathbf{a})}u_{a_1}u_{a_2}\cdots u_{a_\ell}.
\end{align*}
\end{theorem}

 Theorem~\ref{th:main_4} is a generalization of Klyachko's result for $q=1$ in the quotient $\kly_n$. The reader is invited to compare our proof in \S\ref{sec:qKM} with Klyachko's proof presented in \cite[\S8]{NT20}.  It should become apparent that the latter proof does not modify so as to incorporate the $q$. Instead we adapt the argument in \cite{FS94} to $\kly_+\otimes NC_n$ and employ a `Yang-Baxter' relation that works particularly well with the relations in $\kly_+$.
Here $NC_n$ refers to the  (type $A$) NilCoxeter algebra.

We define the \emph{content} of $\mathbf{a}$, denoted by $\cont(\mathbf{a})$ to be the sequence of nonnegative integers whose $i$th entry counts the number of $i$s in $\mathbf{a}$ for all $i\geq 1$. The following two equalities  are immediate consequences of Theorem~\ref{th:main_4}. \begin{align*}
\schub{w}(1,q,q^2,\dots)&=\frac{1}{\qfact{\ell}}\sum_{\mathbf{a}\in \reduced(w)} q^{\comaj(\mathbf{a})}\qint{a_1}\qint{a_2}\cdots \qint{a_\ell},\\
\ds{\schub{w}}_{\ell+1}^q&=\frac{1}{\qfact{\ell}}\sum_{\mathbf{a}\in \reduced(w)} q^{\comaj(\mathbf{a})}A_{\cont({\bf a})}(q),
\end{align*}
The first one is the $q$-Macdonald reduced word identity \cite{FS94} and is obtained by specialization, while the second formula is a consequence of Theorem~\ref{th:main_2}.

 We turn our attention to $a_w(q)\coloneqq \ds{\schub{w}}_{\ell+1}^q$ in this latter formula, from which it is not obvious that $a_w(q)\in \bN[q]$. That this is indeed true follows from work of Kim \cite{Kim20}, which has the feature that whilst it implies the nonnegativity and integrality of the coefficients of $a_w(q)$, the resulting explicit expression unfortunately involves generalized Littlewood-Richardson coefficients. 
 These are in general hard to compute, and finding a combinatorial rule to describe them is a wide open problem.
We approach understanding $a_w(q)$ from our explicit formula above and uncover some curious properties of these polynomials, extending previous work of the authors~\cite{NT20} in the case $q=1$. We briefly describe some developments next, referring the reader to \S\S\ref{sub:Kim}, \S\S\ref{sub:properties of a_w}, and \S\S\ref{sub:special aw}\smallskip

 We obtain a \emph{cyclic sum rule} which generalizes \cite[Theorem 1.5]{NT20} nicely. Informally our rule says that $\sgrpp_n$ can be decomposed into $\mathrm{Cat}_{n-1}$ cyclic classes, in such a way that the sum of the $a_w(q)$ as $w$ ranges over a fixed class tracks the distribution of $\comaj$ over the set of reduced words for any choice of representative from that class.
 
This summatory rule suggests that an approach to describing the $a_w(q)$ combinatorially might involve decomposing the set of reduced words of ${\bf a}$ in some hitherto undetermined manner.
 Nevertheless, in \S\S\ref{sub:special aw}, we are able to deal with special cases, namely, $w$ being a \L ukasiewicz permutation or a Grassmannian permutation.
 
Additionally, for the class of permutations we call quasiindecomposable permutations, we show that the corresponding $a_w(q)$ `encode' the generating function tracking principal specializations of certain Schubert polynomials. Specifically, for $v$ an indecomposable permutation in $\sgrp_{p+1}$ of length $\ell$, we show in Theorem~\ref{thm:quasiindecomposable} that
 \begin{align}
 \label{eq:main_6}
 \sum_{j\geq 0}t^j\;\mathfrak{S}_{1^j\times v}(1,q,q^2,\dots)=\frac{{\displaystyle{\sum_{m=0}^{\ell-p}a_{1^m\times v\times 1^{\ell-p-m}}(q)t^m}}}{(t;q)_{\ell+1}}.
 \end{align}
See \S\S\ref{sub:properties of a_w} for details on undefined notation.
A remark of independent interest is that we do not know an alternative proof for the fact that the rational function expressing the left-hand side of \eqref{eq:main_6} has a numerator which is a polynomial in $\bN[q,t]$.

\smallskip

This brings us to our final stop.
The reader may have observed that two of our highlighted results \textemdash{} equation~\ref{eq:main_1.3} and equation \eqref{eq:main_6}\textemdash{}  are equalities that involve the specialization $x_i\mapsto q^{i-1}$ on one side and $q$-divided symmetrization of either Schubert polynomials or quasisymmetric polynomials on the other side. \emph{What do these two families have common?} Motivated by this we introduce a property of polynomials that we refer to as the interval property. In fact this property is shared by numerous expansions in this work.
Broadly speaking, we say that a polynomial possesses the said property if its image in $\kly$ (or $\kly_+$) expands in the squarefree basis of $\kly$ in terms of monomials $u_I$ where $I$ is a finite interval in $\bZ$.
As we demonstrate in \S\ref{sec:interval}, this has interesting consequences for principal specializations of the sort considered here.

\medskip

\noindent\textbf{Outline of the article:}
We briefly introduce relevant notation in \S\ref{sec:notation}.
The core of this work consists of \S\ref{sec:algebra} and \S\ref{sec:qDS}. Our central object \textemdash{} the Klyachko algebra $\kly$ \textemdash{} is introduced in \S\ref{sec:algebra}. We show that squarefree monomials form a basis of $\kly$, as well as the important quotients $\kly_+$ and $\kly_n$, and  introduce the principal specialization $\spec$. The operation of $q$-divided symmetrization (qDS) is introduced in \S\ref{sec:qDS}. After some preliminary results, we introduce remixed Eulerian numbers $A_c^q$ in \S\S\ref{sub:remixed_Eulerian}, and Theorem~\ref{th:main_2} is established in \S\S\ref{sub:qDS and K}.

The next three sections give different illustrations of these concepts. In \S\ref{sec:remixed} we shed further light on the remixed Eulerian numbers by giving a probabilistic interpretation for $q>0$ as well as establishing the non-negativity of their coefficients as polynomials in $q$. \S\ref{sec:quasisymmetric} contains the interplay with quasisymmetric functions and the proof of Theorem~\ref{th:main_3}. Our motivating result Theorem~\ref{th:main_4} is established at the beginning of \S\ref{sec:qKM}. We then explain the geometric significance of the $a_w(q)$ when $q$ is a prime power, and proceed to establish several of their properties.

\S\ref{sec:interval} introduces the so-called interval property, which is a satisfying unifying explanation to several identities of this work. Finally \S\ref{sec:volumes} ties several qDS evaluations from this work, by introducing a $q$-deformation of the volume polynomial of the permutahedron.

\section{Notation}
\label{sec:notation}
For any undefined terminology relating to various combinatorial objects arising in this work, we refer the reader to \cite{Macdonald,St99}.

We denote the set of nonnegative integers by $\bN$, and the set of positive integers by $\bZ_+$.
Given integers $a\leq b$, on occasion we denote the interval $\{a,a+1,\dots,b\}$ by $[a,b]$. If $a=1$, instead of writing $[1,b]$ we occasionally write $[b]$. We interpret $[1,b]+k$ to mean the interval $[1+k,b+k]$.

We will need $q$-integers. For $n\in \bN$ and $0\leq k\leq n$, set
\begin{align}\label{eq:cq-defs}
\qint{n}\coloneqq \frac{q^n-1}{q-1}=1+q+\cdots+q^{n-1};\quad
\qfact{n}\coloneqq \prod_{1\leq i\leq n} \qint{i};\quad \qbin{n}{k}\coloneqq\frac{\qfact{n}}{\qfact{k}\qfact{n-k}}.
\end{align}
These are respectively called the $q$-integers, $q$-factorials and $q$-binomial coefficients.
These are all polynomials in $q$ with nonnegative integer coefficients. As a consequence, if $q$ is an element of any $\bZ$-algebra, then $\qint{n}, \qfact{n}$ and $\qbin{n}{k}$ are well-defined elements therein.

For an indeterminate $t$ and $n$ a nonnegative integer, define the {\em $q$-Pochhammer symbol}
\begin{align}
  (t;q)_n\coloneqq\prod_{1\leq i\leq n}(1-tq^{i-1}).
\end{align}
Finally, we record Cauchy's $q$-binomial theorems that state
\begin{align}\label{eq:cauchy_binomial_theorem}
  (t;q)_n=\sum_{0\leq n\leq n}(-1)q^{\binom{j}{2}}t^j\qbin{n}{j},\\
  \label{eq:cauchy_2}
  \frac{1}{(t;q)_n}=\sum_{j\geq 0}\qbin{n+j-1}{j}t^j.
\end{align}

\subsection{$\bN$-vectors and compositions}
We call a sequence $c=(c_i)_{i\in\bZ}\in \bN^{\bZ}$ an \emph{$\bN$-vector} if its \emph{support} is finite. Recall that the support of $c$, denoted by $\supp(c)$, is the set of indices $i\in \bZ$ such that $c_i>0$. Let $\wc{}{}$ be the set of $\bN$-vectors.
For $c\in \wc{}{}$, define its \emph{weight} $|c|$ to equal $\sum_ic_i$, and let $\wc{}{k}$ be the set of $\bN$-vectors with $|c|=k$.
We also define $e(c)=|c|-|\supp(c)|$. Note that $e(c)$ is nonnegative, and is equal to $0$ precisely when $c_i\in\{0,1\}$ for all $i$. Such vectors are naturally identified with finite subsets of $\bZ$ via their support.

We will also refer to finite sequences $(c_1,\dots,c_N)$ of nonnegative integers as $\bN$-vectors, or simply vectors. They are naturally embedded in $\wc{}{}$ by setting $c_i=0$ for all $i\in \bZ\setminus [N]$. We write $\wc{N}{}$ for these finite sequences, and $\wc{N}{k}$ for those with $|c|=k$. Let $\rev{c}$ denote the reverse of the vector $c$: if $c=(c_1,\dots,c_N)$, then $\rev{c}$ is by definition $(c_N,\dots,c_1)$.
Given an $\bN$-vector $c=(c_1,\dots, c_n)$, we define $N(c)=\sum_{1\leq i\leq n}(i-1)c_i$.

A \emph{composition} $\alpha$ of $r$ is an $\bN$-vector $(\alpha_1,\dots,\alpha_N)$ with positive entries such that $\alpha_1+\alpha_2+\dots+\alpha_N=r$. This is denoted $\alpha\vDash r$. Let $\set(\alpha)$ denote the subset of $[r-1]$ associated to $\alpha$ under the folklore one-to-one correspondence $\alpha\mapsto \{\alpha_1,\alpha_1+\alpha_2,\dots,\alpha_1+\alpha_2+\dots+\alpha_{N-1}\}$.

\subsection{Symmetric and quasisymmetric polynomials/functions}
Throughout, let $\alpx=\{x_i\}_{i\in \bZ}$, $\alpx_+=\{x_i\}_{i\in \bZ_{>0}}$ and $\alpx_n=\{x_i\suchthat i\in[n]\}$ for $n$ a positive integer.

A polynomial $f\in \bQ[\alpx_n]$ is called \emph{quasisymmetric} if for any strong composition $(\alpha_1,\dots,\alpha_k)$ the coefficient of $x_{i_1}^{\alpha_1}\cdots x_{i_k}^{\alpha_k}$ equals that of $x_{1}^{\alpha_1}\cdots x_{k}^{\alpha_k}$ whenever $1\leq i_1<\cdots<i_k\leq n$.
We denote the ring of quasisymmetric polynomials in $x_1,\dots,x_n$ by $\Qsym_n$. 
A  linear basis for the degree $d$ homogeneous component of $\Qsym_n$ is given by the \emph{fundamental quasisymmetric polynomials} $F_{\alpha}(x_1,\dots,x_n)$ indexed by compositions $\alpha\vDash d$. To define this recall the  correspondence sending $\alpha\vDash d$ to $\set(\alpha)\subseteq [d-1]$ from the last section.
We have
\begin{equation}
\label{eq:fund_quasi_def}
  F_{\alpha}(x_1,\dots,x_n)=\sum_{\substack{1\leq i_1\leq \cdots \leq i_d\leq  n\\ i_j<i_{j+1} \text{ if } j\in \set(\alpha)}}x_{i_1}\cdots x_{i_d}.
\end{equation}
We note here that $F_{\alpha}(x_1,\dots,x_n)=0$ if $\ell(\alpha)>n$.

Let $\Qsym$ denote the ring of quasisymmetric functions.
Elements of $\Qsym$ can be defined as bounded-degree formal power series $f\in \bQ[\![{\alpx_+}]\!]$ such that for any strong composition $(\alpha_1,\dots,\alpha_k)$ the coefficient of $x_{i_1}^{\alpha_1}\cdots x_{i_k}^{\alpha_k}$ equals that of $x_{1}^{\alpha_1}\cdots x_{k}^{\alpha_k}$ for positive integers $i_1<\cdots<i_k$. The fundamental quasisymmetric functions $F_\alpha$ for $\alpha$ a strong composition are then naturally defined by removing the condition $i_d\leq n$ in the sum~\eqref{eq:fund_quasi_def}.

\subsection{Reduced words of permutations, and Schubert polynomials}
We denote the set of permutation on $n$ letters by $\sgrp_n$. An \emph{inversion} in $w\in \sgrp_n$ is an ordered pair $(i,j)$ such that $i<j$ and $w_i>w_j$. The number of inversions will be denoted by $\inv(w)$. This quantity is also otherwise known as the \emph{length} of $w$ and denoted by $\ell(w)$.
Of particular importance is the subset $\sgrpp_n$ of $\sgrp_n$ comprising permutations of length $n-1$.
 A \emph{descent} in $w\in \sgrp_n$ is an index $i\in [n-1]$ such that $w_{i}>w_{i+1}$. The set of descents in $w$ will be denoted by $\Dsc(w)$. The sum of entries in $\Dsc(w)$ is referred to as the \emph{major index} of $w$ and denoted by $\maj(w)$.

A \emph{reduced expression} of a permutation $w$ is any minimal length factorization $w=s_{i_1}\dots s_{i_{\ell(w)}}$ where each $s_{i_j}$ is a simple transposition. 
We refer to the word $i_1\dots i_{\ell(w)}$ as a \emph{reduced word}. 
The set of reduced words for $w$ is denoted by $\reduced(w)$.
Given ${\bf a} \in \reduced(w)$, its \emph{content} $\cont({\bf a})$ is the $\bN$-vector whose $j$th entry counts occurrences of $j$ in ${\bf a}$ for all $j$. We define the statistic $
\comaj({\bf a})$ to be the sum of all $j$ such that $a_j<a_{j+1}$. For instance given the reduced word ${\bf a}=7534232$ for $w=16423587\in \sgrp_8$, we have $\cont({\bf a})=(0,2,2,1,1,0,1)$ and $\comaj({\bf a})=3+5=8$.

\subsection{Schubert polynomials}
\label{sub:schubert}
Let $\sym_n\subseteq \qsym_n\subseteq \bQ[\alpx_n]$ be the subring of symmetric polynomials in $x_1,\ldots,x_n$, and $\sym_n^+$ be the ideal of $\bQ[\alpx_n]$ generated by  elements $f\in \sym_n$ such that $f(0)=0$. 
The quotient ring $R_n=\bQ[\alpx_n]/\sym_n^+$ is the \emph{coinvariant ring}.

 Let $\partial_i$ be the divided difference operator on $\bQ[\alpx_n]$, given by
\begin{equation}
\partial_i(f)=\frac{f-s_i\cdot f}{x_i-x_{i+1}}.
\end{equation}

Define the \emph{Schubert polynomials} for $w\in \sgrp_n$ as follows: $\schub{w_o}=x_1^{n-1}x_2^{n-2}\cdots x_{n-1}$, while if $i$ is a descent of $w$, let $\schub{ws_i}=\partial_i\schub{w}$. Here $w_o$ is the longest word in $\sgrp_n$.
For $w\in\sgrp_n$, the Schubert polynomial $\schub{w}$ is a homogeneous polynomial of degree $\ell(w)$ in $\bQ[\alpx_n]$.
In fact Schubert polynomials are well defined for $w\in\sgrp_\infty$. Moreover, when $w\in\sgrp_\infty$ runs through all permutations whose largest descent is at most $n$, the  Schubert polynomials $\schub{w}$  form a basis for  $\bQ[\alpx_n]$. \smallskip

\section{The Klyachko algebra $\kly$}
\label{sec:algebra}

Let $\mathbf{k}$ be a field of characteristic zero, and  $q\in \mathbf{k}$. {\em We assume that $q$ is not a nontrivial root of unity in $\mathbf{k}$.} In particular $q$ can be $1$ but necessarily $q\neq -1$. This condition ensures that $\qint{n}\neq 0$ for $n\geq 1$, so that we can safely divide by $q$-integers and $q$-factorials.

We will mostly be interested in the `generic' case  where $q$ is an indeterminate and $\mathbf{k}=\bQ(q)$ or $\bC(q)$ say. However important applications occur when:\begin{itemize}
\item $\mathbf{k}=\bR$ and $q>0$, cf. \S\ref{sec:remixed}.
\item $q$ is a prime power, cf. \S\ref{sec:qKM};
\end{itemize}
The case $q=1$ is the starting point of this work, since certain results presented here will give new proofs of certain results of the authors in \cite{DS, NT20}.

\subsection{ The algebra $\kly$}
The central object of our study is the following $\mathbf{k}$-algebra depending on the parameter $q$.

\begin{definition}[Klyachko algebra $\kly$]
Let $\kly=\kly^q$ be the commutative algebra over $\mathbf{k}$ on the generators $\{u_i\suchthat i\in \mathbb{Z}\}$ subject to the following relations for all $i\in\bZ$:
\begin{align}
\label{eq:defining_q_klyachko}
(q+1)u_i^2=qu_iu_{i-1}+u_iu_{i+1}.
\end{align}
\end{definition}

It is a \emph{quadratic algebra}, i.e. its defining relations have degree two. Let $\alpz=(z_i)_{i\in\bZ}$ be indeterminates; we will also need $\alpz_{+}=(z_i)_{i\in\bZ_+}$ and $\alpz_{r}=(z_i)_{i=1,\ldots,r}$ for $r>0$. We let
\[
\proj:\mathbf{k}[\alpz]\to \kly, z_i\mapsto u_i
\]
be the defining projection.
If we let $\ideal$ be the ideal of $\mathbf{k}[\alpz]$ generated by the elements $(q+1)z_i^2-qz_iz_{i-1}-z_iz_{i+1}$ for $i\in\bZ$, then $\kly$ is by definition isomorphic to $\mathbf{k}[\alpz]/\ideal$. 

\medskip

{\noindent \em Shift and homothecies.} Consider the \emph{shift} $\tau:i\mapsto i+1$ on $\bZ$. This induces an automorphism on $\mathbf{k}[\alpz]$ via $z_i\mapsto z_{i+1}$, which we continue to denote by $\tau$. The relations of $\kly$ are invariant under $\tau$, so that $u_i\mapsto  u_{i+1}$ defines an automorphism of $\kly$ still denoted by $\tau$ and named {shift}.

We also note that for any $\lambda\in\mathbf{k}^*$, $u_i\mapsto \lambda u_i$ extend to an automorphism of $\kly$ since the relations \eqref{eq:defining_q_klyachko} are homogeneous.

\begin{remark}
\label{rem:kly_symmetry}
There is a more homogeneous version of the relations of $\kly$, namely:
\[(q_1+q_2)u_i^2=q_1u_iu_{i-1}+q_2u_iu_{i+1}.
\]
for some $q_1,q_2\in \mathbf{k}$. If $q_2\neq 0$ this becomes the relation for $\kly^q$ with $q=q_1/q_2$ so there is no loss of generality in keeping just one parameter, as we opted to do in our definition.

 The homogeneous version has the advantage of exhibiting more symmetry. In particular the relation is invariant under $u_j\mapsto u_{-j}$ and $q_1\leftrightarrow q_2$, which shows that $\kly^q\simeq \kly^{1/q}$ for nonzero $q$.
\end{remark}

 {\noindent \em Quotients.} Given any subset $I$ of $\mathbb{Z}$, we let $\kly_{I}$ be the algebra generated by $u_i\in I$ subject to the relations \eqref{eq:defining_q_klyachko} with $u_i=0$ for $i\notin I$. That is, $\kly_{I}$ is the quotient of $\kly$ by the ideal generated by the $u_i$ for $i\notin I$. The most important quotients are for $I$ semi-infinite and for $I$ an interval:

\begin{definition}
Let $\kly_+\coloneqq \kly_{\bP}$ and $\kly_n\coloneqq \kly_{\{1,\ldots,n-1\}}$.
\end{definition}

We define $\proj_+:\mathbf{k}[\alpz_+]\to \kly_+$ to be the natural projection sending $z_i$ to $u_i$, whose kernel $\ideal_+$ is generated by the elements $(q+1)z_i^2-qz_iz_{i-1}-z_iz_{i+1}$ for $i>1$ together with  $(q+1)z_1^2-z_1z_2$.
 Analogously, define $\proj_n:\mathbf{k}[\alpz_{n-1}]\to \kly_n$ with kernel $\ideal_n$ generated by the elements $(q+1)z_1^2-z_1z_2$, $(q+1)z_{n-1}^2-qz_{n-1}z_{n-2}$ and $(q+1)z_i^2-qz_iz_{i-1}-z_iz_{i+1}$ for $1<i<n-1$.\smallskip

{\em Note that we have kept the notation $u_i$ for the generators $u_i$ in different algebras $\kly_I$, in order to avoid too much notation. When a computation holds in a quotient but not in the full algebra $\kly$, we will always make sure to specify it clearly.}
\smallskip

\begin{remark} The presentation of $\kly_n$ for $q=1$ coincides with $H^* (\Perm_n,\bQ)^{\sgrp_n}$ as proved in~\cite{Kly85, Kly95} and used by the authors in \cite{NT20}.
\end{remark}

\subsection{Basic identities.}
\label{sub:basic_identities}

\begin{definition}
For any finite subset $I\subset \bZ$, define $u_I\coloneqq \prod_{i\in I} u_i\in \kly$.
\end{definition}

Given the defining relations of $\kly$, it is natural to expect that these squarefree monomials form a linear basis of $\kly$, which is the content of Proposition~\ref{prop:uI_basis}. Toward establishing this we prove some basic identities that hold in $\kly$. The next lemma lays the groundwork for many of the computations in this article.

\begin{lemma}\label{lem:mult_by_ui}
Consider integers $a\leq b$ and  $i\in[a,b]$. In $\kly$, we have
\begin{align}
\label{eq:mult_by_ui}
\qint{b-a+2}u_iu_{[a,b]}&= q^{i-a+1}\qint{b-i+1}u_{[a-1,b]}+\qint{i-a+1}u_{[a,b+1]}
\end{align}
Moreover, \eqref{eq:mult_by_ui} is a consequence of the relations~\eqref{eq:defining_q_klyachko} for $i=a,\ldots,b$.
\end{lemma}

The second claim shows a \emph{local} property of the rewriting, which is useful in order to compare computation in various quotients $\kly_I$.

\begin{proof}
Let $L=u_{[a-1,b]}$ and $R=u_{[a,b+1]}$, whih are the two elements of $\kly$ occurring on the right hand side of~\eqref{eq:mult_by_ui}.
Then the relations~\eqref{eq:defining_q_klyachko} gives that the elements $f_i\coloneqq u_iu_{\{a,\dots,b\}}$ satisfy 
\begin{align*}
(1+q)f_i&=qf_{i-1}+f_{i+1}\text{ for }a<i<b;\\
 (1+q)f_a&=qL+f_{a+1};\\
  (1+q)f_b&=qf_{b-1}+R.
 \end{align*}
 Let $n=b-a+1$. Let $F$ be the $n\times 1$ vector with entries $f_i$; $M$ the $n\times n$ tridiagonal Toeplitz matrix with diagonal $1+q$, superdiagonal $-1$ and subdiagonal $-q$; and $C$ the $n\times 1$ vector with $C_1=qL, C_n=R$ and all other entries zero. Then the above system of equations can be succinctly  written $MF=C$. For $a=1,b=4$ we get

 \[\underbrace{\begin{pmatrix}
 1+q  & -1 & 0 & 0
 \\ -q & 1+q & -1  & 0 \\ 0 & -q & 1+q & -1  \\ 0 & 0 & -q & 1+q
 \end{pmatrix}
 }_{M}
 \underbrace{\begin{pmatrix}
 u_1u_{[1,4]} \\ u_2u_{[1,4]} \\ u_3u_{[1,4]} \\ u_4u_{[1,4]}
 \end{pmatrix}}_{F}
 =
 \underbrace{\begin{pmatrix}
 qu_{[0,4]} \\ 0 \\ 0 \\ u_{[1,5]}
 \end{pmatrix}}_{C}
 \]

 The determinant of $M$ can be computed to be $\qint{b-a+2}$, which is nonzero by our assumption on $q$. The inverse $M^{-1}$ is easily checked to be given by
 \begin{align*}
(M^{-1})_{ij}=\frac{1}{\qint{b-a+2}}\times\begin{cases} q^{j-i}\qint{n+1-i}\qint{j}\text{ for }i\geq j;\\
\qint{n+1-j}\qint{i}\text{ for }i<j.
\end{cases}
 \end{align*}
Then the rows of the identity $F=M^{-1}C$ give the relations \eqref{eq:mult_by_ui}. The second claim in Lemma~\ref{lem:mult_by_ui} follows immediately from the previous proof.
\end{proof}

We note a corollary which follows from elementary computations:

\begin{corollary}
\label{cor:mult_by_xi}
Let $i\in\bZ$ and assume the interval $[a,b]$ contains $i-1$ or $i$; that is $[a,b]\cap\{i-1,i\}\neq \emptyset$. Then we have
\begin{equation}
\qint{b-a+2}(u_i-u_{i-1})u_{[a,b]}= q^{i-a}\left(u_{[a,b+1]}-u_{[a-1,b]}\right).
\end{equation}
\end{corollary}
Now we give an important relation allowing us to expand any element as a linear combination of the elements $u_I$.
\smallskip

Let $c$ be an $\bN$-vector and let $J=\supp(c)$.
Assume there exists an index $i$ such that $c_i\geq 2$, and fix such an $i$. Let $I=[a,b]$ be the \emph{maximal} interval in $J$ containing $i$, so that $a-1,b+1\notin J$ and thus
$c_{a-1}=c_{b+1}=0$.
Let $L_i(c)$ (resp. $R_i(c)$) be the $\bN$-vector $c'$ defined by $c'_k=c_k$ if $k\neq a-1,i$ (resp. $k\neq i,b+1$), $c'_i=c_i-1$, and $c'_{a-1}=1$ (resp. $c'_{b+1}=1$). In words, these are obtained from $c$ by decrementing $c_i$ by $1$ and incrementing $c_{a-1}$ (resp. $c_{b+1}$) by $1$.

For an $\bN$-vector $c$, let \[
u^c=\prod_{i\in \bZ}u_i^{c_i}.
\]

 \begin{corollary} With $c,i,a,b$ as above and assuming $c_i\geq 2$, we have in $\kly$
\begin{equation}
\label{eq:recurrence_uc}
\qint{b-a+2}u^c=q^{i-a+1}\qint{b-i+1}u^{L_i(c)}+\qint{i-a+1}u^{R_i(c)}.
\end{equation}
\end{corollary}

\begin{proof}
By definition the generators $u_j$ for $j=a,\ldots,b$ occur in $u^c$, and since $c_i\geq 2$ we can apply \eqref{eq:mult_by_ui}. Since $a-1,b+1\notin J$ we have the desired result.
\end{proof}

\subsection{The squarefree basis}
\label{sub:squarefree_basis}

\begin{proposition}
\label{prop:uI_basis}
Let $J\subset \bZ$. The $u_I$ for $I$ finite subset of $J$ form a $\mathbf{k}$-basis of the algebra $\kly_J$.
\end{proposition}

We denote the basis of squarefree monomials for $\kly_I$ by $\mca{B}_{I}$. In the important cases $\kly$, $\kly_+$, and $\kly_n$, we denote the corresponding bases by $\mcbk$, $\mcbkp$, and $\mcbkn$ respectively.

\begin{proof}
From~\eqref{eq:recurrence_uc}, by induction on $e(c)=|c|-|\supp(c)|$, one can reduce any monomial in $u_i$ to a linear combination of squarefree monomials. Indeed $e(L_i(c))=e(R_i(c))=e(c)-1$ when $c_i\geq 2$, and recall that $e(c)=0$ precisely when $u^c$ is squarefree. It follows immediately that the $u_I$ span $\kly$; by passing to the quotient, the same holds for any $\kly_J$.\smallskip

We now want to show linear independence of these elements. First, it is enough to prove this for the algebras $\kly_n$ for $n\geq 1$. Indeed any relation between squarefree monomials in an algebra $\kly_J$ must hold in an algebra $\kly_{J'}$ for $J'$ a finite interval: this is a consequence of the second part of Lemma~\ref{lem:mult_by_ui}. By the shift invariance of the relations, we can assume $J'=[1,n-1]$ for a certain $n$.

 We will use standard notions of commutative algebra in the proof, for which we refer to \cite[\S I.5]{St96} for instance.
Recall that $\kly_n$ is defined as the quotient of $\mathbf{k}[\alpz_{n-1}]$ by  $\ideal_n$, an ideal generated by $n-1$ quadratic polynomials. Now we have just shown that squarefree monomials span $\kly_n$, which is therefore finite-dimensional over $\mathbf{k}$. It follows that the generators of $\ideal_n$ form a \emph{homogeneous system of parameters} for $\mathbf{k}[\alpz_{n-1}]$, as a module over itself.

 Since $\mathbf{k}[\alpz_{n-1}]$ is Cohen-Macaulay, these generators form necessarily a \emph{regular sequence}. Hence we get the following identity between Hilbert series\footnote{Recall that the Hilbert series $\Hilb(\oplus_{i\geq 0}V_i,t)$ of a graded vector space with finite-dimensional  homogeneous components is the generating series $\sum_{i\geq 0}(\dim) V_i t^i$.}:
\[\Hilb(\mathbf{k}[\alpz_{n-1}],t)=\frac{1}{(1-t^2)^{n-1}}\Hilb(\kly_n,t).\]

Since $\Hilb(\mathbf{k}[\alpz_{n-1}],t)=\frac{1}{(1-t)^{n-1}}$, we obtain $\Hilb(\kly_n,t)=(1+t)^{n-1}$. So $\kly_n$ has $\mathbf{k}$-dimension $2^{n-1}$, which is exactly the number of squarefree monomials it contains. Since we have shown that these span $\kly_n$, they necessarily form a basis.
\end{proof}

\begin{corollary}
\label{cor:stability}
 Let $P\in\mathbf{k}[\alpz_+]$, and suppose  that there exists $N>0$ such that $P(u_1,u_2,\dots)=0$ in $\kly_n$ for all $n\geq N$. Then $P(u_1,u_2,\dots)=0$ in $\kly_+$.
\end{corollary}
\begin{proof}
$P(u_1,u_2,\dots)$ vanishes in $\kly_+$ if all its coefficients on the basis $u_I, I\subset \bZ_+$ are zero. Now let $d$ be the degree of $P$ and $a$ be the maximal index of an indeterminate $z_i$ occurring in $P$. It follows from the second part of Lemma~\ref{lem:mult_by_ui} that, in order to reduce $P$ in the basis $u_I, I\subset \bZ_+$, one only needs relations involving generators $u_i$ with $i<a+d$. These relations occur in any $\kly_n$ for $n$ large enough, and so by hypothesis we have $P(u_1,u_2,\dots)=0$ in $\kly_+$.  \end{proof}

As we will see in \S\ref{sec:remixed}, coefficients of general monomials expressed in the basis $\mcbk$ have a probabilistic interpretation due to Diaconis and Fulton~\cite{DiaFul91}. In this seminal work, the algebras $\kly_I$ actually occur as special cases of a wide class of algebras, and Proposition~\ref{prop:uI_basis} is proved in greater generality. We included our own proof as to be self contained; our arguments are also of a different nature.

\subsection{Specializations}
\label{sub:specialization}
The following specialization will be instrumental in this work to go from computations in $\kly$ to numerical formulas.

\begin{proposition}
\label{prop:specialization}
The assignment
\[
 u_i\mapsto
  \begin{cases}
  \qint{i}&\text{ for }i> 0;\\
              0&\text{ for }i\leq 0,
\end{cases}
\]
 extends to a unique algebra morphism $\spec:\kly\to \mathbf{k}$.
\end{proposition}

\begin{proof}
One has $(q+1)\qint{i}^2=q\qint{i}\qint{i-1}+\qint{i}\qint{i+1}$ for $i>1$ and $(q+1)\qint{1}^2=0\times\qint{1}+\qint{1}\qint{2}$, and the relations are trivially verified for $i\leq 0$ so the map extends to an algebra morphism. It is unique since the $u_i$  generate $ \kly$ as an algebra.
\end{proof}

The specialization obviously factors through $\kly_+$, and we keep the same notation $\spec$:
\begin{align*}
\spec:\kly_+&\to \mathbf{k}\\
 u_i&\mapsto \qint{i}
\end{align*}

Now consider an interval $I=[a+1,b]$ with $a\geq 0$ so that $u_I\in\kly_+$. Then we get
\begin{equation}
\label{eq:sp_uI}
\spec\left(\frac{u_{[a+1,b]}}{\qfact{b-a}}\right)=\qbin{b}{a}.
\end{equation}
This evaluation will allow us to go from identities in $\kly_+$ to numerical formulas.

\subsubsection*{Other specializations} It is natural to consider specializations, i.e. characters of $\kly$. Can we characterize the algebra morphisms $\kly\to \mathbf{k}$? Given the presentation~\eqref{eq:defining_q_klyachko}, the problem is equivalent to characterizing all sequences $(a_i)_{i\in\bZ}$ in $\mathbf{k}$ that satisfy $(1+q)a_i^2=qa_ia_{i-1}+a_ia_{i+1}$ for all $i$. The corresponding specialization is then given by $u_i \mapsto a_i$.
\begin{itemize}
\item  A first family of specializations is given by
 \[a_i=\alpha+\beta q^i\text{ for all }i\] with $\alpha,\beta\in\mathbf{k}$ two parameters. This family contains in particular all solutions $(a_i)_{i\in\bZ}$ that do not vanish.
\item  A second family of specializations, depending on two integer parameters $m\leq M$ and two parameters $\alpha,\beta\in\mathbf{k}$, can given by
\[
\begin{cases}
a_i=0\text{ for }i\in [m,M];\\
a_{M+k}=\alpha\qint{k}\text{ for }k>0;\\
a_{m-k}=\beta (k)_{1/q}\text{ for }k>0.
\end{cases}
\]
\end{itemize}
 It is then an easy exercise to show that every specialization belongs to one of the two families.\smallskip

It follows immediately that the specializations of $\kly_+$ are precisely the maps $\alpha\cdot\spec\circ \tau^{-m}$ for $\alpha\in\mathbf{k}, m\geq 0$. The only specialization from $\kly_n$ is the augmentation $u_i\mapsto 0$ sending elements of positive degree to zero.

\section{$q$-divided symmetrization}
\label{sec:qDS}

In this section we introduce $q$-divided symmetrization, which is a natural $q$-deformation of the divided symmetrization operator introduced by Postnikov~\cite{Pos09} and further studied in \cite{Amd16,DS,Pet18}. We state a few properties of this operator before showing its direct connection with the algebra $\kly$ in Theorems~\ref{th:qDS_and_K} and \ref{th:expansion_f}. The connection goes through an an interesting family of numbers, the {\em refined mixed Eulerian numbers $A_c^q$}, further studied in \S\ref{sec:remixed}.

\subsection{Definition and first properties}
Consider the natural action of $\sgrp_n$ on $\mathbf{k}[x_1,\dots,x_n]$ obtained by letting the generator $s_i$ for $1\leq i\leq n-1$ act by swapping $x_i$ and $x_{i+1}$.
The divided difference operator $\partial_i$ is then given by $\frac{1-s_i}{x_i-x_{i+1}}$.
Following \cite[\S2]{DKLLST95}, define operators $\square_i$ for $1\leq i\leq n-1$ acting on functions in $x_1,\dots,x_n$ as\footnote{We note here that the authors of \cite{DKLLST95} act on the left, whereas we act on the right.}
\[
\square_i=\partial_i(qx_i-x_{i+1})
\]
These operators can be described via Demazure-Lusztig operators $T_i$ that act on polynomials
as
\[
T_i=(q-1)\partial_i x_i +s_i.
\]
Indeed, one can check that $\square_i=T_i+1$.
Since the $T_i$ satisfy Hecke relations,\footnote{In fact $s_i$, $\partial_i$, and $T_i$ are all special cases of a five parameter family \cite{LS87} that satisfy Hecke relations.} one obtains a well-defined operator $T_{w}$ for $w\in \sgrp_n$ by picking a reduced word. For $w_o$ the longest word in $\sgrp_n$,  consider the operator $\square_{w_o}$ \cite[Theorem 3.1]{DKLLST95} expressed as follows:
\begin{equation}\label{eq:square}
\square_{w_o}=\partial_{w_o} \prod_{1\leq i<j\leq n}(qx_i-x_j).
\end{equation}

Recall that $\partial_{w_o}$ simply performs antisymmetrization followed by division by the Vandermonde determinant.
In fact, $\square_{w_o}$ is the \emph{maximal Yang-Baxter element} as defined in \cite{DeLa06}, up to a power of $q$; see also \cite[\S 7]{LLT97}.
By \cite[Theorem 3.1]{DKLLST95}, the operator $\square_{w_o}$ affords the expansion
\begin{equation}\label{eq:square2}
\square_{w_o}=\sum_{\sigma\in \sgrp_n}T_{\sigma}.
\end{equation}

\begin{definition}\label{def:qds}
Define the \emph{$q$-divided symmetrization} (abbreviated to \emph{qDS} on occasion) of $f\in \mathbf{k}[x_1,\dots,x_n]$ to be
\begin{align*}
\ds{f}_n^q&\coloneqq\square_{w_o} \frac{f}{\prod_{1\leq i\leq n-1}(qx_i-x_{i+1})}=\sum_{w\in \sgrp_n}T_{w}\left( \frac{f}{\prod_{1\leq i\leq n-1}(qx_i-x_{i+1})}\right)
\end{align*}
\end{definition}

When $q=1$, the Hecke operator $T_w$ reduces to the permutation $w$ acting on polynomials. Definition \ref{def:qds} thus specializes to Postnikov's divided symmetrization \cite[\S 3]{Pos09}, which was the starting point for our definition. We will describe several properties of qDS that extend known properties of the case $q=1$.
\smallskip
Thanks to~\eqref{eq:square} we have $\ds{f}_n^q=\partial_{w_o}\left(f\prod_{1\leq i<j-1\leq n-1}(qx_i-x_j)\right)$ and so $\ds{f}_n^q$ is a symmetric polynomial in $x_1,\dots,x_n$.

It follows that if $\deg(f)<n-1$, then $\ds{f}_n^{q}$ is $0$. So the `first' instance where qDS becomes interesting is when $\deg(f)=n-1$, in which case $f\mapsto \ds{f}_n^{q}$ is a linear form on $\mathbf{k}[x_1,\dots,x_n]$. For the purposes of this article we restrict ourselves to this setting, as it turns out that this is the case of interest in conjunction with the algebra $\kly$.

It is clear that the definition just mentioned make qDS appear rather arduous to compute. The next result, extending a formula of Petrov \cite[Eq. (1)]{Pet18}, goes a long way in simplifying that task.

\begin{lemma}\label{lem:petrov_recursion}
If $f\in \mathbf{k}[x_1,\dots,x_n]$ of degree $n-1$ factors as $g(x_1,\dots,x_i)(qx_i-x_{i+1})h(x_{i+1},\dots,x_n)$, then
\begin{align*}
\ds{f}_n^q= \qbin{n}{i}\ds{g}_{i}^q \ds{h}_{n-i}^{q}.
\end{align*}
\end{lemma}
\begin{proof}
Let $C(i,n-i)$ be the set of minimal length coset representatives for $\sgrp_n/ (\sgrp_{i}\times \sgrp_{n-i})$.
We may factor $\square_{w_o}$ as
\begin{align}\label{eq:petrov_lem_1}
	\square_{w_o}=\sum_{w\in \sgrp_n}T_w
	=\sum_{v\in C(i,n-i)}T_{v}\left(\sum_{w'\in S_i}T_{w'}\right)\left(\sum_{w''\in S_{n-i}}T_{w''}\right)
\end{align}
Given the special form that $f$ possesses, from the above factorization we infer that
\begin{align}\label{eq:petrov_lem_2}
\ds{f}_n^q=\ds{g}_i^q\ds{h}_{n-i}^q\sum_{v\in C(i,n-i)}T_{v}(1).
\end{align}
Since $T_i(1)=q$ for any $i$, from \eqref{eq:petrov_lem_2} we obtain
\begin{align}
\ds{f}_n^q=\ds{g}_i^q\ds{h}_{n-i}^q\sum_{v\in C(i,n-i)}q^{\ell(v)}.
\end{align}
The claim follows because the length generating function over minimal length coset representatives is given by $q$-binomial coefficients.
\end{proof}

\subsection{qDS of monomials}

\begin{figure}[ht]
  \begin{tikzpicture}[scale=.6]
    \draw[gray,very thin] (0,0) grid (10,4);
    \draw[line width=0.25mm, black, <->] (0,2)--(10,2);
    \draw[line width=0.25mm, black, <->] (1,0)--(1,4);
    \node[draw, circle,minimum size=5pt,inner sep=0pt, outer sep=0pt, fill=blue] at (1, 2)   (b) {};
    \node[draw, circle,minimum size=5pt,inner sep=0pt, outer sep=0pt, fill=red] at (2, 1)   (c) {};
    \node[draw, circle,minimum size=5pt,inner sep=0pt, outer sep=0pt, fill=blue] at (3, 3)   (d) {};
    \node[draw, circle,minimum size=5pt,inner sep=0pt, outer sep=0pt, fill=blue] at (4, 2)   (e) {};
    \node[draw, circle,minimum size=5pt,inner sep=0pt, outer sep=0pt, fill=red] at (5, 1)   (f) {};
    \node[draw, circle,minimum size=5pt,inner sep=0pt, outer sep=0pt, fill=red] at (6, 0)   (g) {};
    \node[draw, circle,minimum size=5pt,inner sep=0pt, outer sep=0pt, fill=red] at (7, 0)   (h) {};
    \node[draw, circle,minimum size=5pt,inner sep=0pt, outer sep=0pt, fill=blue] at (8, 2)   (i) {};
    \node[draw, circle,minimum size=5pt,inner sep=0pt, outer sep=0pt, fill=blue] at (9, 1)   (j) {};
    \draw[blue, line width=0.7mm] (b)--(c);
    \draw[blue, line width=0.7mm] (c)--(d);
    \draw[blue, line width=0.7mm] (d)--(e);
    \draw[blue, line width=0.7mm] (e)--(f);
    \draw[blue, line width=0.7mm] (f)--(g);
    \draw[blue, line width=0.7mm] (g)--(h);
    \draw[blue, line width=0.7mm] (h)--(i);
    \draw[blue, line width=0.7mm] (i)--(j);
  \end{tikzpicture}
  \caption{$\mc{P}_{c}$ when ${c}=(0,3,0,0,0,1,3,0)$ with $S_{c}=\{1,4,5,6\}$.}
  \label{fig:path_to_compute_Sc}
\end{figure}
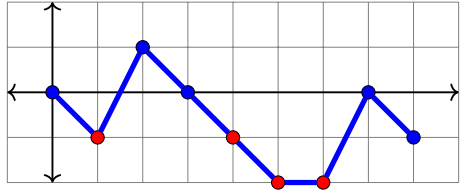

We represent $\bN$-vectors $(c_1,\dots,c_{n})$ as \emph{{\L}ukasiewicz paths} starting at $(0,0)$ by interpreting each $c_i$ as as a vector step $(1,c_i-1)$.
We refer to the path corresponding to $c$ as $\mc{P}_c$.
For instance, Figure~\ref{fig:path_to_compute_Sc} depicts $\mc{P}_{c}$ for  ${c}=(0,3,0,0,0,1,3,0)$
The lattice point on $\mc{P}_c$ with $x$-coordinate $i$ has $y$-coordinate given by $h_i=h_i(c)\coloneqq \left(\sum_{k=1}^ic_k\right) -i$.
In particular $h_0=0$ and $h_n=\sum_ic_i-n$.
We define a statistic by
\begin{align*}
  \sht{c}=\sum_{1\leq i\leq n-1}h_i.
\end{align*}

We are now ready to state the $q$-analogue of \cite[Proposition 3.5]{Pos09}.
The authors do not know of a proof in the spirit of Postnikov's argument which relies on constant term extraction. Instead, our proof, presented in the appendix, is based on the repeated use of Lemma \ref{lem:petrov_recursion}.
\begin{proposition}\label{prop:qDS for monomials}
Let $c=(c_1,\dots,c_{n})$ be an $\bN$-vector with $\sum_ic_i=n-1$.
Let $S$ be the set of $i\in\{1,\ldots,n-1\}$ such that $h_i(c)<0$. Let $X_S$ be the set of permutations in $\sgrp_n$ with descent set $S$.
Then
\begin{equation}
\label{eq:qDS for monomials}
\ds{x_1^{c_1}\dots x_n^{c_n}}_n^q=(-1)^{|S|} \sum_{\sigma\in X_S}q^{\binom{n-1}{2}-\mathrm{inv}(\sigma)-\sht{c}}
\end{equation}
\end{proposition}

\begin{example}
\label{ex:demo_qds_monomials}
Consider the monomial corresponding to $c=(1,0,1)$. Then $S=\{2\}$, and therefore $X_{S}=\{132,231\}$. Also, we have $h(c)=(1-1)+(1-2)=-1$.
Proposition~\ref{prop:qDS for monomials} tells us that
\[
\ds{x_1x_3}_3^q=-(q^{\binom{3-1}{2}-\mathrm{inv}(132)+1}+q^{\binom{3-1}{2}-\mathrm{inv}(231)+1})=-(q+1).
\]
\end{example}

\subsection{Remixed Eulerian numbers}
\label{sub:remixed_Eulerian}
For an $\bN$-vector $c=(c_1,\dots,c_{r})$, let
\[y_c=x_1^{c_1}(x_1+x_2)^{c_2}\dots (x_1+\dots+x_r)^{c_r},\]
a polynomial of degree $|c|$.

\begin{definition}[Remixed Eulerian numbers $A_c^q$]
\label{defi:remixed_Eulerian}
Let $c\in\wc{r}{r}$, so that $c=(c_1,\dots,c_{r})$ with $|c|=r$. We define the \emph{remixed Eulerian number} $A_{c}^q\in \mathbf{k}$ by
\begin{equation}
A_{c}^q=\ds{y_c}_{r+1}^q.
\end{equation}
\end{definition}

Note that $A_{c}^q$ indeed belongs to the base field $\mathbf{k}$ since $y_c$ has degree $r$ and we apply qDS with respect to $r+1$ variables, see remarks following Definition~\ref{def:qds}. For $q=1$, we get back the \emph{mixed Eulerian numbers} $A_c=A_c^{1}$ introduced by Postnikov~\cite[\S 16]{Pos09}; see also \cite[\S 4]{NT20}.

\begin{remark}
It follows from the definition of qDS that $A_c^{q}\in \bZ[q]\subset \mathbf{k}$. If $q$ is transcendental over $\bQ\subset\mathbf{k}$, then $A_c^{q}$ is a polynomial in $q$ that does not depend on $\mathbf{k}$. This explains why we can, and will, safely talk about remixed Eulerian numbers independently of $\mathbf{k}$, as polynomials is $q$.
\end{remark}

\begin{remark}
  \label{rem:why y_c 0}
 Notice that we divide-symmetrize with respect  to $x_1,\ldots,x_{r+1}$, while $y_c$ for $c=(c_1,\dots,c_{r})$ only depends on variables among $x_1,\ldots,x_{r}$. As the $T_i$ admit symmetric polynomials as scalars, we have $\ds{y_{(c_1,\ldots,c_{r+1})}}_{r+1}^q=0$ whenever $c_{r+1}>0$, so there is no loss of generality in restricting to sequences with $c_{r+1}=0$.
\end{remark}

\begin{theorem}
\label{thm:uniqueness_via_recurrence_petrov}
	For a fixed positive integer $r$, the family $(A^q_{c})_{c\in\wc{r}{r}}$ is uniquely determined by the following conditions:
	\begin{enumerate}[label=(\alph*)]
		\item \label{it:uniq1} $A^q_{(1^{r})}=\qfact{r}$.
		\item \label{it:uniq2} for any $i=1,\ldots,r$ such that $c_i\geq 2$,
		\begin{equation}
		\label{eq:Ac_relation}
		(q+1)A^q_{(c_1,\dots,c_{r})}= qA^q_{(c_1,\dots,c_{i-1}+1,{c_i}-1,\dots,c_{r})}+A^q_{(c_1,\dots,c_i-1,c_{i+1}+1,\dots,c_{r})}.
		\end{equation}
	\end{enumerate}
If $i=1$ (resp. $i=r$), then the first (resp. second) term on the right hand side is to be omitted.
\end{theorem}

\begin{example}
\label{ex:remixed_Eulerians}
Here are all remixed Eulerian numbers $A_c^q$ indexed by $c=(c_1,c_2,c_3)$ such that $|c|=3$. We suppress commas and parentheses in writing our compositions for clarity.
\begin{align*}
\begin{array}{lllll}
A_{111}^q=(1+q)(1+q+q^2) & A_{120}^q=(1+q)^2  & A_{201}^q=(1+q)^2 & A_{210}^q=(1+q)^2 & A_{300}^q=1\\
A_{102}^q=q(1+q+q^2) & A_{012}^q=q^2(1+q) & A_{021}^q=q(1+q)^2 & A_{030}^q=2q(1+q) & A_{003}^q=q^3
\end{array}
\end{align*}

The careful reader will have observed that all $A_c^q$ listed here are polynomials in $q$ with \emph{nonnegative} integral coefficients. This is not immediate from Theorem~\ref{thm:uniqueness_via_recurrence_petrov}, nor from Definition~\ref{defi:remixed_Eulerian}.
We establish this in \S\ref{sec:remixed} by working with a variant of equation~\eqref{eq:recurrence_A}.
\end{example}

\begin{proof}[Proof of Theorem~\ref{thm:uniqueness_via_recurrence_petrov}]
Let us first show that the numbers $A^q_{c}$ satisfy \ref{it:uniq1} and \ref{it:uniq2}. First, expanding $y_{(1^{r})}=x_1(x_1+x_2)\dots (x_1+\dots+x_r)$ gives the sum of all $r!$ monomials $x_{i_1}\dots x_{i_r}$ with $1\leq i_j\leq j$. 
Write such monomials as $\alpx_{\bf i}$ with ${\bf i}=(i_1,\dots,i_r)$. We apply Proposition~\ref{prop:qDS for monomials} with $n=r+1$ to each of those monomials: let $m({\bf i})=(m_1,\dots,m_r,m_{r+1})$ where $m_j$ is the number of occurrences of $j$ in $m({\bf i})$. Then the descent set $S$ attached to any vector $m({\bf i})$ is empty, and thus \eqref{eq:qDS for monomials} simplifies to $\ds{\alpx_{\bf i}}_{r+1}^q =q^{\binom{r}{2}- h(m({\bf i}))}$. Using $\sum_j jm_j=\sum_ji_j$, one computes easily $h(m({\bf i}))=\binom{r}{2}-\sum_{j=1}^ri_j$. It follows that 
\[A^q_{(1^{r})}=\sum_{\bf i}\ds{\alpx_{\bf i}}_{r+1}^q=\sum_{\bf i}q^{\sum_{j=1}^ri_j}\]
 which is readily seen to be the desired $q$-factorial, proving \ref{it:uniq1}.

\smallskip

We now focus on \ref{it:uniq2}.
To establish it we need to check that
\begin{align}
\ds{(q+1)y_{(c_1,\dots,c_{r})}-qy_{(c_1,\dots,c_{i-1}+1,{c_i}-1,\dots,c_{r})}-y_{(c_1,\dots,c_i-1,c_{i+1}+1,\dots,c_{r})}}_{r+1}^q=0,
\end{align}
which by the definition of $y_c$ is equivalent to
\begin{align}\label{eq:goal_for_property_2}
\ds{y_{(c_1,\dots,c_{i-1},{c_i}-1,c_{i+1},\dots,c_r)}(qx_i-x_{i+1})}_{r+1}^q=0.
\end{align}
By expressing $x_1+\cdots+x_m$ as $(x_1+\cdots+x_i)+(x_{i+1}+\cdots+x_m)$  for $m>i$, we may expand $y_{(c_1,\dots,c_{i-1},{c_i}-1,c_{i+1},\dots,c_r)}$ as
\begin{align}
y_{(c_1,\dots,c_{i-1},{c_i}-1,c_{i+1},\dots,c_r)}=\sum_{k\geq 0}y_{(c_1,\dots,c_{i-1},c_i-1+k)}f_k(x_{i+1},\dots,x_r)
\end{align}
where only finitely many polynomials $f_k$ are nonzero.
We can then apply apply Lemma~\ref{lem:petrov_recursion} to obtain
\begin{align}
\ds{y_{(c_1,\dots,c_{i-1},{c_i}-1,\dots,c_r)}(qx_i-x_{i+1})}_{r+1}^q=\qbin{{r+1}}{i}\sum_{k\geq 0}\ds{y_{(c_1,\dots,c_{i-1},c_i-1+k)}}_{i}^q \ds{f_k}_{r+1-i}^{q}.
\end{align}

We claim that the right hand side vanishes. Indeed if $\deg(f_k)<r-i$, this is already true.
So let us suppose that $\deg(f_k)\geq r-i$.
If so, then we must have $\deg(y_{(c_1,\dots,c_{i-1},c_i-1+k)})\leq i-1$.
Since $c_i\geq 2$, we know that $c_i-1+k\geq 1$, and thus $\ds{y_{(c_1,\dots,c_{i-1},c_i-1+k)}}_{i}^q=0$ because of the presence of the symmetric factor $x_1+\cdots+x_i$.
Thus, the equality in \eqref{eq:goal_for_property_2} follows.
The preceding argument is technically only valid for $i\neq 1,r$, and we leave it to the reader to check those two cases separately.\smallskip

Now let us show that the $A_c^q$ are the only numbers satisfying these relations. The issue is that property ~\ref{it:uniq2} in Theorem~\ref{thm:uniqueness_via_recurrence_petrov} does not immediately determine these numbers by induction.
However, one can easily deduce the new relations \eqref{eq:recurrence_A} below.
These relations together with the initial condition \ref{it:uniq1} show that such a family of numbers is unique if it exists, by induction on $e(c)=|c|-|\supp(c)|$.

 Let $i$ such that $c_i\geq 2$. Let $I=[a,b]\subset [1,r]$ be the maximal interval of the support of $c$ containing $i$. Define $L_i(c),R_i(c)$ as in Lemma~\ref{lem:mult_by_ui}. A slight modification is in order: that if $a=1$ (resp. $b=r$) then $L_i(c)$ (resp. $R_i(c)$) has support outside of $[1,r]$. In this case the corresponding mixed Eulerians are not defined, we set them to $0$. Then we have
\begin{equation}
\label{eq:recurrence_A}
\qint{b-a+2}A_{c}^q=q^{i-a+1}\qint{b-i+1}A_{L_i(c)}^q+\qint{i-a+1}A_{R_i(c)}^q.
\end{equation}

The proof is entirely identical to Lemma \ref{lem:mult_by_ui}, since the system of equations to solve is the same. This completes the proof of Theorem~\ref{thm:uniqueness_via_recurrence_petrov}.
\end{proof}
We reiterate that the relation in property ~\ref{it:uniq2} in Theorem~\ref{thm:uniqueness_via_recurrence_petrov} is of a different character than that in equation~\eqref{eq:recurrence_A}.
Indeed the former implies, for instance, that
\begin{align}
A_{012}^q=(1+q)A_{021}^q,
\end{align}
whereas the latter says that
\begin{align}
\qint{3}A_{012}^q=q^2\qint{1}A_{111}^q.
\end{align}
The second immediately implies given the initial condition $A_{111}^q=\qfact{3}$ that $A_{012}=q^2(1+q)$. On the other hand, the first relation does not yield the same return right away.

\subsection{qDS as top coefficient in $\kly$}
\label{sub:qDS and K}
Given the similarity in the second property in Theorem~\ref{thm:uniqueness_via_recurrence_petrov} and the defining relations of the Klyachko algebra $\kly$, one senses that there is an intimate relation between qDS and computations in $\kly$. This relation in the case $q=1$ is implicit in \cite{NT20}, in which case this connection has a geometric explanation; cf. Remark~\ref{rem:q1_proof} below.

 We make the relation precise here, allowing for general $q$. We need to introduce certain elementary morphims first.
\begin{definition}[$\morph,\morph_+,\morph_n$ and $\projmorph,\projmorph_+,\projmorph_n$]
We define the morphism of algebras from $\mathbf{k}[\alpx]$ to $\mathbf{k}[\alpz]$ by 
\[\morph:x_i\mapsto z_i-z_{i-1}.\]
 We define also $\morph_+$ from $\mathbf{k}[\alpx_+]$ to $\mathbf{k}[\alpz_+]$, and $\morph_n$ from $\mathbf{k}[\alpx_n]$ to $\mathbf{k}[\alpz_{n-1}]$ by restriction where $z_i$ is set to $0$ if it is not part of the image. Explicitly,
$\morph_+(x_i)=z_i-z_{i-1}$ for $i>1$ and $\morph_+(x_1)=z_1$; and $\morph_n(x_i)=z_i-z_{i-1}$ for $1<i<n$, $\morph_n(x_1)=z_1$, $\morph_n(x_n)=-z_{n-1}$. 

Finally, we define $\projmorph\coloneqq \proj\circ \morph$, $\projmorph_+\coloneqq \proj\circ \morph_+$, and $\projmorph_n\coloneqq\proj\circ \morph_n$.
\end{definition}

Note in particular that $\morph_+(x_1+\dots+x_i)=z_i$ so that $\morph_+(y_c)=z^c$.
These morphisms are illustrated in Figure~\ref{fig:all_morphisms}.
\begin{figure}
\begin{tikzcd}[column sep=5.6em,row sep=1.9em]
& \mathbf{k} &\\ 
\kly\simeq \mathbf{k}[\alpz]/\ideal \arrow[ur,"\mathrm{sp}"]\arrow[r,"\text{$u_i=0$ for $i\leq 0$}",swap] &  \kly_+  \arrow[u,"u_i\mapsto\qint{i}",swap]\arrow[r,"\text{$u_i=0$ for $i\geq  n$}",swap] & \kly_n\\
{}& {}& {} & \mathbf{k}[\alpx_n]/\qsym_n^+ \arrow[ul]\\
\mathbf{k}[\alpz] \arrow[uu,"\pi:z_i\mapsto u_i",swap] \arrow[r,leftarrow] &  \mathbf{k}[\alpz_+] \arrow[uu,"\pi_+",swap]\arrow[r,leftarrow]&   \mathbf{k}[\alpz_{n-1}]\arrow[uu,"\pi_n",swap]\\
{}& {}& {} & \mathbf{k}[\alpx_n]/\sym_n^+ \arrow[uu] \\
\mathbf{k}[\alpx] \arrow[uu,"\Delta:x_i\to z_i-z_{i-1}",swap] \arrow[r,leftarrow] \arrow[uuuu,bend left=40, crossing over,"\rho",near end]
 &  \mathbf{k}[\alpx_+] \arrow[uu,"\Delta_+",swap]\arrow[r,leftarrow] \arrow[uuuu,bend left=40, crossing over,"\rho_+",near end]
&   \mathbf{k}[\alpx_{n}]\arrow[uu,"\Delta_n",swap]\arrow[uuuu,bend left=40, crossing over,"\rho_n",near end]\arrow[ur]
\end{tikzcd}
\caption{\label{fig:all_morphisms}
Commutative diagram grouping various algebra morphisms relevant to us. The morphisms from the bottom row are introduced in~\S\ref{sub:qDS and K}. The right hand side is dealt with in \S\ref{sub:ABB}, cf. Corollary~\ref{cor:rho vanishes on qsym_+}.}
\end{figure}

\begin{theorem}
\label{th:qDS_and_K} Let $f$ be a polynomial in $\mathbf{k}[\alpx_n]$  with $\deg(f)=n-1$.
Consider the expansion of $\rho_n(f)=f(u_1,u_2-u_1,\dots, u_{n-1}-u_{n-2},-u_{n-1})$ in the  basis $\mcbkn$, and let $\Top_n(f)$ be the top coefficient in this expansion, namely that of $u_1\cdots u_{n-1}$. Then
\begin{equation}
\label{eq:qDS_and_K}
\ds{f}^q_n= \qfact{n-1} \times \Top_n(f).
\end{equation}

\end{theorem}
\begin{proof}
Observe that polynomials
$y_{(c_1,\dots,c_{n})}=x_1^{c_1}(x_1+x_2)^{c_2}\dots (x_1+\dots+x_n)^{c_n}$ with $c=(c_1,\dots,c_n)$ satisfying $|c|=n-1$ comprise a basis of the degree $n-1$ homogeneous component of $\mathbf{k}[\alpx_n]$.
By linearity, it suffices to check \eqref{eq:qDS_and_K} on these basis elements.

When $c_n>0$, then $\projmorph_n(y_c)$ vanishes and therefore $\Top_n(y_c)=0$. On the other hand, as mentioned in Remark~\ref{rem:why y_c 0}, we have $\ds{y_c}_n^q=0$ as well. Hence we now assume $c_n=0$, and define $r=n-1$. We also redefine  $c=(c_1,\dots,c_r)$ to keep notations simple. Now $\projmorph_n(y_c)=u^c$, so the right hand side of\eqref{eq:qDS_and_K} for $f=y_c$ is $\qfact{r}$ times the coefficient of $u_{[r]}$ in $u^c$; let $A'_c$ be this value. Since $\ds{y_c}_n^q=A_c^q$, we now need to show that the numbers $A'_c$ satisfy the two conditions of Theorem~ \ref{thm:uniqueness_via_recurrence_petrov}.

If $c=(1^{r})$ then $u^c=u_{[r]}$ so $A'_c=\qfact{r}$ as Property \ref{it:uniq1} in Theorem~\ref{thm:uniqueness_via_recurrence_petrov}. Now if we have any $c_i\geq 2$, consider the relation~\eqref{eq:defining_q_klyachko}, multiply both sides by $\qfact{r}u^d$ for $d=(c_1,\dots,c_{i-1},c_i-2,c_{i+1},\dots,c_r)$ and extract the coefficient of $u_{[r]}$. The resulting relation is Property \ref{it:uniq2} in Theorem~\ref{thm:uniqueness_via_recurrence_petrov} for $A'_c$, which finally shows that $A'_c=A_c$ and completes the proof.
\end{proof}

\begin{example}
\label{ex:qDS_and_K}
We revisit the monomials in Example~\ref{ex:demo_qds_monomials}. We set $n=3$. If $f=x_1^2$, then $\projmorph_n(f)=u_1^2$, which in view of the relation $(1+q)u_1^2=u_1u_2$ implies that the coefficient of $u_1u_2$ in $\projmorph_n(f)$ is $1/(1+q)$. Theorem~\ref{th:qDS_and_K} then says $\ds{x_1^2}_3^q=1$.

On the other hand, if we pick $f=x_1x_3$, then $\projmorph_n(f)=-u_1u_2$. This time, the coefficient that needs to be scaled by $\qfact{2}$ is $-1$, and therefore $\ds{x_1x_3}_3^q=-(1+q)$.
\end{example}

\begin{remark}
\label{rem:q1_proof} For $q=1$, Theorem~\ref{th:qDS_and_K} follows from a geometric viewpoint, and the statement was exploited by the authors in \cite{NT20}. The reasoning is as follows: we pick $f$ to be the Schubert polynomial $\schub{w}$ for a permutation $w\in\sgrp_{n}$ of length $n-1$. It can be shown that it suffices to check the equality in Theorem~\ref{th:qDS_and_K} for such $f$.\footnote{The reason is that these Schubert polynomials span a vector space inside all polynomials in $\mathbf{k}[\alpx_n]$  with $\deg(f)=n-1$, which possesses a complementary subspace on which both sides of~\eqref{eq:qDS_and_K} vanish.}
 Now on the one hand, $\ds{\schub{w}}^1_{n}$ computes certain intersection numbers $a_w$ for the permutahedral variety, as explained in ~\cite{NT20}, following~\cite{And10,Kim20}. On the other hand, work of Klyachko \cite{Kly85,Kly95} mentioned in the introduction shows that these same numbers $a_w$ can be computed in the ring $\kly_n^1\simeq H^* (\Perm_n,\bQ)^{\sgrp_n}$ by the formula $(n-1)! \times \Top_n(\schub{w})$. This completes the argument. In \S\ref{sub:Kim}, we will see how these calculations turn out to also have a geometric content when $q$ is a prime power.
\end{remark}

\subsection{Extension to all coefficients in $\mcbk$}
\label{ssub:more general expansions}

Theorem~\ref{th:qDS_and_K} shows that a certain coefficient in the basis $\mcbk$ of squarefree monomials can be interpreted in terms of qDS. It is therefore natural to inquire about the coefficients of other squarefree monomials. In this subsection we show that they can also be expressed in terms of qDS. 

Let $I$ be a finite subset of $\bZ$, and write it uniquely as a union of maximal intervals: \[I=I_1\sqcup I_2\sqcup \ldots \sqcup I_k\] with $I_j=[a_j,b_j]$ and $a_{j+1}-b_j\geq 2$ for $1\leq j<k$.
Given $f$ in $\mathbf{k}[\alpz]$, define
\[
f^I(\alpx)=f(\alpz)_{\left|{\begin{cases}z_i=0 \text{ if }i\notin I;\\ z_i=x_{a_j}+\cdots +x_{i}\text{ if }i\in I_j.\end{cases}}\right.}.
\]

Note that $f\mapsto f^I$ is an algebra morphism from $\mathbf{k}[\alpz]$ to $\mathbf{k}[x_i,~i\in I]$. Define
\begin{equation}
\label{eq:QI}
 Q_I(\alpx)\coloneqq \prod_{j=1}^{k-1}(qx_{b_{j}+1}-x_{b_{j}+2})\times\cdots\times (qx_{a_{j+1}-1}-x_{a_{j+1}}).
\end{equation}

\begin{example}
Consider $I=\{1,3,4,5,9,10\}$. Then we have $I_1=\{1\}$, $I_2=\{3,4,5\}$, and $I_3=\{9,10\}$.
For $f$ any polynomial in the $\alpz$ variables, we obtain $f^{I}(\alpx)$ by setting $z_1=x_1$, $z_3=x_3$, $z_4=x_3+x_4$, $z_5=x_3+x_4+x_5$, $z_9=x_9$, and $z_{10}=x_9+x_{10}$.
Furthermore we have
\[
Q_I(\alpx)=(qx_2-x_3)\cdot (qx_6-x_7)(qx_7-x_8)(qx_8-x_9).
\]
\end{example}

If $I=\{1,\dots,r\}$ and $f=z_1^{c_1}\cdots z_r^{c_r}$, then
\[
f^{I}(\alpx)=x_1^{c_1}(x_1+x_2)^{c_2}\dots (x_1+\cdots+x_r)^{c_r}=y_{(c_1,\dots,c_r)}.
\]
In this case $Q_I(\alpx)$ is the empty product, and so equals $1$.

We are now ready to state the generalization of Theorem~\ref{th:qDS_and_K}. For this we define $\ds{\cdot}_{I}^q$ to mean the qDS of the argument over the variables $x_{i}$ for $i\in I$.

\begin{theorem}
\label{th:expansion_f}
Let $f$ be a homogeneous degree $d$ polynomial in $\alpz=(z_i)_{i\in\bZ}$.
Then the coefficient $u^*_I(f)$ of $u_{I}$ in the expansion
\[\proj(f)=\sum_{I\subset \bZ, |I|=d}u_I^*(f)u_I\] is given by
\begin{equation}
\label{eq:expansion_f}
u_I^*(f)=\frac{1}{\qfact{b_k-a_1+1}}\ds{Q_I(\alpx)\times f^I(\alpx)}_{\{a_1,a_1+1,\ldots,b_k+1\}}^q.
\end{equation}
\end{theorem}
 We recover qDS in the sense of \S\ref{sec:qDS} when $a_1=1$ and $b_k=n-1$. We postpone the proof  to the Appendix~\ref{app:proof_technical_theorem}, and discuss an important special case.\smallskip

 Consider $f=z_1^{c_1}\cdots z_k^{c_k}$ where $c_i>0$ for $i\in [k]$. Suppose $|c|=r$.
 We can naturally think of $c$ as an $\bN$-vector in $\wc{r}{r}$.
 From equation~\eqref{eq:expansion_f}, one infers that $u_I^{*}(f)= 0$ for  all $I\subset \bZ$ with cardinality $r$ unless $I$ is an interval containing $[k]$.
 So to express $\proj(f)$ in the basis $\mcbk$ we only need to compute $u_{I_j}^{*}(f)$ for the $r-k+1$ intervals $I_j\coloneqq [j,j+r-1]$ for $k+1-r\leq j\leq 1$. For such  intervals, $Q_{I_j}$ equals $1$, and
 \[
 f^{I_j}(\alpx)=\prod_{1\leq i\leq k}(x_j+x_{j+1}+\cdots+x_i)^{c_i}.
 \]
Now, $u_{I_j}^*(f)$ is obtained by applying qDS to this expression, and we can reindex to get
 \[
 u_{I_j}^*(f)=\frac{1}{\qfact{r}}\ds{\prod_{0\leq j\leq r-k}(x_1+\dots+x_{j+1})^{c_1}\cdots (x_1+\dots+x_{j+k})^{c_k}}_{[1,n]},
 \]
and the right-hand side is readily recognized as $A_{(0^j,c,0^{r-j-k})}(q)\times \frac{1}{\qfact{r}}$. 

To summarize this discussion, we have that if $c$ is a strong composition of $r$, there holds in $\kly$
\begin{align}
\label{eq:important for connected remixed}
u_1^{c_1}\cdots u_k^{c_k}=\frac{1}{\qfact{r}}\sum_{0\leq j\leq r-k}A_{(0^j,c,0^{r-j-k})}(q)u_{[1-j,r-j]}.
\end{align}

\begin{remark}
We have stated this result here to illustrate~\ref{th:expansion_f}, but it can be proved without it. It is indeed a consequence of the interpretation of $A_c(q)$ via Theorem~\ref{th:qDS_and_K} together with the shift-invariance of the relations of $\kly$.
\end{remark}

\section{Remixed Eulerian numbers}
\label{sec:remixed}

In this section we give certain properties of positivity of the numbers $A_c^q$. First we show that when $q$ is a positive real number, these give certain probabilities in a certain diffusion process on $\bZ$. We then show that as polynomials in $q$, $A_c^q$ has nonnegative integer coefficients.

A much more in-depth study of these polynomials is given by the authors in \cite{NT21_remixed}.

\subsection{A probabilistic process after Diaconis-Fulton}

Consider the integer line $\bZ$ as a set of \emph{sites}.
A \emph{configuration} is an $\bN$-vector $c=(c_i)_{i\in \bZ}$ with $\sum_ic_i<\infty$ as before, which we now visualize as a finite set of particles, where there are $c_i$ particles stacked at site $i$. The configuration $c=(\dots,0,3,0,c_0=1,1,2,0,2,0,\dots)$ is represented below. \emph{Stable} configurations are those for which $c_i\leq 1$ for all $i$, and are naturally identified with finite subsets of $\bZ$ via their support.

\begin{center}
\includegraphics{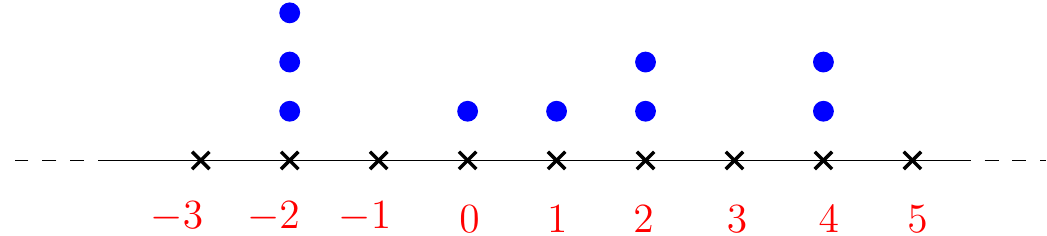}
\end{center}

Given \emph{jump probabilities} $q_L,q_R\geq 0$ satisfying $q_L+q_R=1$, we consider the following process with state space the set of configurations: Suppose we are in a configuration $c$. If $c$ is stable, the process stops. Otherwise, pick any $i$ such that $c_i>1$ and move the top particle at site $i$ to the top of site $i-1$  with probability $q_L$, and to the top of site $i-1$   (resp. $i+1$)(resp. $q_R$). The resulting configurations are $L_i(c)$ and $R_i(c)$ respectively, as seen already in this work.

Let us assume here that we perform such a jump every unit of time while the process hasn't stabilized. To specify the process completely, one needs to give a rule (deterministic or probabilistic) on how to pick the site $i$ in case there are several sites with more than one particle. It turns out that for our purposes it will not matter, as we now explain.

 Note first that whichever rule is chosen, it is easy to show that the process will end up in a stable configuration with probability $1$. We can then define $\prob_c(I)$ to be the probability that the process, according to the chosen rules and  starting a configuration $c$, ends up in the stable configuration given by the finite subset $I$ of $\bZ$. Diaconis and Fulton~\cite{DiaFul91} showed a strong \emph{abelian property} of this model, namely that \emph{$\prob_c(I)$ is independent of the chosen rules for picking the jumping particle}. 

Assume $q_R>0$ and let $q\coloneqq q_L/q_R$, so that $q_L=\frac{q}{1+q}$ and $q_R=\frac{1}{1+q}$. We have then the following result:
\begin{proposition}
\label{prop:probabilities as coefficients}
Fix an $\bN$-vector $c$. Then we have the following expansion in $\kly$ in the basis $\mcbk$:
\[
u^c=\sum_{\text{finite }I\subset\bZ} \prob_c(I) u_I.
\]
In particular, if $c\in\wc{r}{r}$, then
\[
\prob_c([r])=\frac{A^q_c}{\qfact{r}}.
\]
\end{proposition}

In this probabilistic interpretation, we recognize Lemma \ref{lem:mult_by_ui} as the {\em gambler's ruin}. Namely, the coefficients on the right hand side compute the left and right exit probabilities of an asymmetric random walker on the integer segment $[a,b]$, starting from $i$.

\begin{proof}[Proof of Lemma~\ref{prop:probabilities as coefficients}]
If $c$ is already stable, there is nothing to show as $u^c=u_{\supp(c)}\in \mcbk$, and $\prob_c(J)=\delta_{J,\supp(c)}$ for any finite $J\subset \bZ$.
Otherwise suppose $c_i>2$ for some $i\in \bZ$. We are allowed to pick $i$ to hop because of the abelian property, and then obtain immediately the following transition equation
\begin{align}
\prob_c(I)=q_L \prob_{L_i(c)}(I)+q_R\prob_{R_i(c)}(I).
\end{align}
Now the  we see that the coefficient $A^I_c$ of $u_I$ in the $\mcbk$-expansion of $u^c$ satisfies the same relation. For a fixed $I$, these relations characterize $A^I_c$ following the proof for remixed Eulerian numbers in the case $I=[r]$, and the proof is complete.
\end{proof}

The obvious parallel between the transition moves in this particle jumping process and the relations defining $\kly$ eliminates the element of surprise from our preceding result. Nevertheless, on viewing the result in isolation, it is intriguing that this simple probabilistic process computes $q$-analogues of a discrete-geometric notion \textemdash{} mixed Eulerian numbers in this case.  Our next remark frames this particle process in the appropriate context which is the work of Diaconis-Fulton \cite{DiaFul91}.

\begin{remark}
The above process is an instance of Diaconis-Fulton's \emph{growth game} \cite{DiaFul91}. For the case where finitely many particles are dropped on the origin beginning with the origin being the only inhabited site, one obtains $q$-analogues of classical Eulerian numbers (suitably normalized), and this case is also discussed in \cite{Mi20}.
\end{remark}

A  property of remixed Eulerian numbers that will come handy in the sequel is a cyclic sum rule satisfied by remixed Eulerian numbers. In the case $q=1$, this was established by Postnikov \cite{Pos09} by a geometric argument. It was later reproven by Petrov~\cite{Pet18} by a probabilistic argument which we essentially follow here.

To state the rule, we need some notation.
For $c=(c_1,\dots,c_{r})\in \wc{r}{r}$, set $\tilde{c}=(c_1,\dots,c_{r},0)$. Write $c\sim c'$ if $\tilde{c}'$ can be  obtained as a cyclic rotation of $\tilde{c}$. This is an equivalence relation on $\wc{r}{r}$, and we define
 $\mathsf{Cyc}(c)$ to be the equivalence class of $c$.
For $c=(2,0,1)$, $\mathsf{Cyc}(c)$ contains $c$ and $c'=(1,0,2)$ since $\tilde{c}=(2,0,1,0)\sim (1,0,2,0)=\tilde{c}'.$

\begin{proposition}
\label{prop:cyclic sum remixed}
For $c\in \wc{r}{r}$, we have
\[
\sum_{c'\in \mathsf{Cyc}(c)}A_{c'}(q)=\qfact{r}.
\]
\end{proposition}
\begin{proof}
As this is a polynomial relation, it is enough to prove it when $q$ is a nonnegative real. Consider the same dynamic as above, but on a circular set of sites labeled $1$ through $r+1$ clockwise. Fix $c\in \wc{r}{r}$ and start the process with the configuration of $c_i$ particles at site $i$ for $1\leq i\leq r$, and $0$ at site $r+1$. The process stabilizes with probability one and has the abelian property, as it is still an instance of the processes considered in \cite{DiaFul91}.

In a stable configuration, exactly one of the $r+1$ sites will be unoccupied. By rotational symmetry, the probability $P_j$ of ending at the stable configuration supported on $[r+1]\setminus \{j\}$ is equal to the probability of reaching the stable configuration supported on $[r]$ but starting from the configuration $c$ rotated counterclockwise $j$ times. Clearly $P_j$ is $0$ if $c_{j}>0$. If not, the rotated configuration has support in $[r]$ and so corresponds naturally to an element $c'$ in $\mathsf{Cyc}(c)$. The probability $P_j$ in this case is given by $\prob_{c'}([r])$ introduced above, since the process behaves exactly as on $\bZ$.
 
The preceding argument together with Proposition~\ref{prop:probabilities as coefficients} gives 
\begin{align}
\sum_{c'\in \mathsf{Cyc}(c)}\frac{A_{c'}(q)}{\qfact{r}}=1,
\end{align}
which is what we set out to prove.
\end{proof}

Continuing our example above, we get the following relation which can be checked from the data in Example~\ref{ex:remixed_Eulerians}
\[A_{(2,0,1)}(q)+A_{(1,0,2)}(q)=\qfact{3}.\]

A particularly interesting case of Proposition~\ref{prop:cyclic sum remixed} is when $\supp(c)$ is an interval. Indeed, the cyclic rotations do not matter and the sum in question only runs over linear shifts of $c$, and only those sequences contribute to the sum whose support lies in $[r]$. We have already encountered this in the discussion right after Theorem~\ref{th:expansion_f}, and we address this further in \S\S\ref{sub:connected} particularly as it has ramifications when we discuss Schubert polynomials.

\subsection{Positivity of coefficients}
\label{sub:positivity rmen}
In the rest of this section we consider $A_{c}(q)$ as a polynomial.

Definition~\ref{defi:remixed_Eulerian} implies immediately that $A_{c}(q)\in \mathbb{Z}[q]$. We know also that for $q=1$, we have $A_c\coloneqq A_c(1)$ is the mixed Eulerian number \cite{Pos09}. It is positive, for instance by its original interpretation as a mixed volume. More generally, for $q>0$ we now know that $A_c(q)>0$ because of the probabilistic interpretation from Proposition~\ref{prop:probabilities as coefficients}.

  What is not obvious from its definition, or even from the relations~\eqref{eq:recurrence_A}, is that the coefficients of $A_{c}(q)\in \mathbb{Z}[q]$ are actually nonnegative; indeed the presence of the factor $\qint{b-a+2}$ prevents us from concluding so immediately.

\begin{proposition}
\label{prop:mixed_positive_coeffs}
For any $c=(c_1,\ldots,c_r)$ such that $|c|=r$, we have $A_c(q)\in \mathbb{N}[q]$.
\end{proposition}

\begin{proof}
Let $I_1\sqcup I_2\sqcup \ldots \sqcup I_k$ denote the maximal interval decomposition of $\supp(c)$, and suppose that $m_j\coloneqq |I_j|$ for $1\leq j\leq k$.
Define
\[\tilde{A}_{c}(q)\coloneqq A_{c}(q)/\prod_i \qfact{m_i}.\]
We will prove the stronger claim that $\tilde{A}_{c}(q)\in \mathbb{N}[q]$, which of course implies the  proposition since the $q$-factorials have nonnnegative coefficients. 
Suppose first that $c_i\leq 1$ for all $i$. The only possibility is $c=(1^r)$ in which case $\tilde{A}_{c}(q)=A_c(q)/\qfact{r}=1$ by Theorem~\ref{thm:uniqueness_via_recurrence_petrov}.

Now consider $i$ such that $c_i\geq 2$, and suppose it belongs to the maximal interval $I_j=[a,b]$.  Note that $m_j+1=b-a+2$.
Now note that ~\eqref{eq:recurrence_A} says
\begin{align}
\qint{m_j+1}A_{c}(q)=q^{i-a+1}\qint{b-i+1}A_{L_i(c)}(q)+\qint{i-a+1}A_{R_i(c)}(q),
\end{align}
where $L_i(c)$ and $R_i(c)$ are defined as in the proof of Theorem~\ref{thm:uniqueness_via_recurrence_petrov}.
By dividing both sides of this equation  by $\qint{m_{j}+1}\prod_{1\leq p\leq k} \qfact{m_p}$, we obtain
\begin{equation}
\label{eq:recurrence_Areduced}
\tilde{A}_{c}(q)=b_L\;q^{i-a+1}\qint{b-i+1}\;\tilde{A}_{L_i(c)}(q)+b_R\;\qint{i-a+1}\;\tilde{A}_{R_i(c)}(q).
\end{equation}
Here $b_L$ and $b_R$ are determined by how $\supp(L_i(c))$, $\supp(R_i(c))$ and $\supp(c)$ are related, and we perform this analysis next.

If $a=1$, then $L_i(c)$ is not defined and the corresponding term is omitted.
If $a=2$ or $c_{a-2}=0$, then the maximal interval decomposition of $\supp(L_i(c))$ is obtained by replacing $I_j$ by ${\{a-1\}}\cup I_j$.
Note that the cardinality of the latter is $m_j+1$, and therefore we infer that $b_L=1$.
If we are not in either of the preceding two scenarios, then we have $I_{j-1}=[f,a-2]$ for some $f\geq 1$, so that $[f,b]$ is a maximal interval in the support of $L_i(c)$. In this case, we have $b_L=\qbin{b-f+1}{b-a+2}$.

 Symmetrically, if $b=1$, then $R_i(c)$ is not defined and the corresponding term is omitted.
 If  $b=r-1$ or $c_{b+2}=0$, then $b_R=1$.
 Otherwise we have $I_{j+1}=[b+2,g]$ for some $g\leq r$, so that $[a,g]$ is a maximal interval in  $\supp(R_i(c))$. In this case, we have $b_R=\qbin{g-a+1}{b-a+2}$.

Thus equation~\eqref{eq:recurrence_Areduced} gives a recurrence for $\tilde{A}_{c}(q)$ which clearly shows that it has nonnegative coefficients.
\end{proof}
The cyclic sum rule for remixed Eulerians descends to one for the reduced analogues.
\begin{corollary}
\label{cor:cyclic sum reduced}
Fix $c\in \wc{r}{r}$. Suppose $\supp(c)$ decomposes as a disjoint union of \textup{(}inclusion\textup{)} maximal intervals as $I_1 \sqcup \cdots \sqcup I_k$. Set $m_j=|I_j|$.
\[
\sum_{c'\in \mathsf{Cyc}(c)}\tilde{A}_{c'}(q)=\frac{\qfact{n-1}}{\prod_{1\leq j\leq k}\qfact{m_j}}.
\]
\end{corollary}

\subsection{Connected remixed Eulerian numbers}
\label{sub:connected}
We now return to the case where $c$ is an $\bN$-vector such that $\supp(c)$ is an interval in $\bZ$.
Our result here is an alternative encoding of the coefficients in the expansion in equation~\eqref{eq:important for connected remixed}.
In fact, what we obtain is a $q$-analogue of a recent result of Berget-Spink-Tseng \cite[Theorem A.1]{Ber20} in their investigation of log-concavity of matroid $h$-vectors.

\begin{proposition}
\label{prop:connected_men_gf}
Let $c=(c_1,\ldots,c_k)\vDash r$. We have
\begin{align}
  \label{eq:u^c principal}
\sum_{m\geq0}t^m\prod_{i=1}^k\qint{m+i}^{c_i}=\frac{\displaystyle\sum_{i=0}^{r-k}A_{0^ic0^{r-k-i}}(q) t^i}{(t;q)_{r+1}}.
\end{align}
\end{proposition}
\begin{proof}

Fix $m\geq 0$. Using the shift invariance of the relations defining $\kly$, we can expand $u_{m+1}^{c_1}\dots u_{m+k}^{c_k}=
\tau^m(u_1^{c_1}\cdots u_k^{c_k})$ in the basis $\mcbk$ of $\kly$ by shifting indices on the right-hand side in equation~\eqref{eq:important for connected remixed}.
We thus obtain
\begin{align}
\qfact{r}u_{m+1}^{c_1}\dots u_{m+k}^{c_k}=
\sum_{j=0}^{r-k}A_{0^{j}c0^{r-j-k}}(q)u_{[1,r]+(m-j)}.
\end{align}
Via the specialization $\spec:u_i\mapsto \qint{i}$ for $i>0$, see \ref{prop:specialization}, we obtain the polynomial identity
\[\qint{m+1}^{c_1}\dots \qint{m+k}^{c_k}=\sum_{j=0}^{\min(m,r-k)}A_{0^jc0^{r-k-j}}(q)\qbin{r+m-j}{m-j}.\]
By Cauchy's $q$-binomial theorem \eqref{eq:cauchy_2}, the right-hand side is the coefficient of $t^m$ in $\frac{\sum_{i=0}^{r-k}A_{0^ic0^{r-k-i}}(q) t^i}{(t;q)_n}$, which completes the proof.
\end{proof}

We briefly touch upon various instances of connected remixed Eulerian numbers in disguise in literature. As elaborated in \cite[\S3]{Gessel_Survey}, given positive integers $p_1$ through $p_m$, MacMahon considered the $\mathrm{maj}$-$\dsc$ distribution over the set of  permutations $M_{(p_1,\dots,p_m)}$ of the multiset $\{1^{p_1},\dots,m^{p_m}\}$. Suppose $r=p_1+\cdots+p_m$. MacMahon showed \cite[Eq. (3.3)]{Gessel_Survey} that
\begin{align}
\sum_{n\geq 0}t^n\prod_{1\leq i\leq m}\qbin{n+p_i}{p_i}=\frac{\sum_{\pi\in M_{(p_1,\dots,p_m)}}t^{\dsc(\pi)}q^{\mathrm{maj}(\pi)}}{(t;q)_{r+1}}.
\end{align}
By setting $c_i\coloneqq |\{j\suchthat p_j\geq i\}|$ and $k\coloneqq \max\{p_1,\dots,p_m\}$, one can rewrite this as
\begin{align}
\sum_{n\geq 0}t^n\prod_{1\leq i\leq k}\qint{n+i}^{c_i}=\frac{\prod_{1\leq i\leq m}\qfact{p_i}\left(\sum_{\pi\in M_{(p_1,\dots,p_m)}}t^{\dsc(\pi)}q^{\mathrm{maj}(\pi)}\right)}{(t;q)_{r+1}}.
\end{align}
Note that we must have $c_1\geq c_2\geq\cdots\geq c_k>0$. Also given any such partition $(c_1,\dots,c_k)$, one can always associate a corresponding $(p_1,\dots,p_m)$. Thus we see that for partitions $c=(c_1,\dots,c_k)$, the polynomial $A_{0^ic0^{r-k-i}}(q)$ tracks the distribution of major index over $\pi \in M_{(p_1,\dots,p_m)}$ where $\dsc(\pi)=i$, up to scaling by $\prod_{1\leq i\leq m}\qfact{p_i}$. The case where $p_i=1$ for all $1\leq i\leq m$ is particularly well known.
In this case $c=(r)$ and one obtains
\begin{align}
A_{(0^i,r,0^{r-k-i})}(q)=\sum_{\substack{\pi\in \sgrp_r\\ \dsc(\pi)=i}} q^{\mathrm{maj}(\pi)},
\end{align}
and the right-hand side gives the $q$-Eulerian number of Carlitz \cite{Car75}.

More generally, the connected remixed Eulerian numbers appear in work of Garsia-Remmel \cite{GaRe86} as \emph{$q$-hit polynomials}.
Indeed, Equations I.11 and I.12 in loc. cit. together imply that the $A_{0^ic0^{r-k-i}}(q)$ are $q$-hit polynomials (up to a power of $q$) of certain Ferrers boards determined by $c$. Informally, this is how the correspondence goes. Consider any Dyck path from $(0,0)$ to $(r,r)$ such that its \emph{area sequence} has $c_i$ instances of $i$ for $i\in [r]$. Then it must be the case that $\supp(c)$ is an interval in $[r]$.  Such a Dyck path in turn determines a Ferrers diagram $\lambda(c)$, and computing the $q$-hit polynomial as a weighted sum over all nonattacking placements in the $r\times r$ board where $i$ rooks land inside $\lambda(c)$ gives $A_{0^ic0^{r-k-i}}(q)$.
We address this connection in more detail in \cite{NT21_remixed}.

\section{Specializations of quasisymmetric polynomials}
\label{sec:quasisymmetric}
We now aim to give the expansion of $\projmorph_+(P)=P(u_1,u_2-u_1,u_3-u_2,\dots)$ in the basis $\mcbkp$ for any {quasisymmetric} polynomial $P(x_1,\dots,x_m)$. Amongst other things, this will shed a new light on the main results of the authors in~\cite{DS} for $q=1$.  It helps to consider  the case of monomials first.

\subsection{Expansion of monomials in $u_i-u_{i-1}$}
\label{sub:monomials_x}
Consider the expansion of $\projmorph(\alpx^c)=\prod_i(u_i-u_{i-1})^{c_i}$ for $c$ any $\bN$-vector:
\begin{align}
\label{eq:expansion_monomial}
\prod_i(u_i-u_{i-1})^{c_i}=\sum_{I}\psi^c_I u_I.
\end{align}
We use Theorem~\ref{th:expansion_f} to find a formula for the coefficients $\psi^c_I$. Assume $|I|=|c|$, otherwise $\psi^c_I=0$ for degree reasons.
Let $I_1\sqcup I_2\sqcup \ldots \sqcup I_k$ be the maximal interval decomposition of $I$, with $I_j=[a_j,b_j]$. We then introduce $\widetilde{I}\coloneqq\widetilde{I}_1\sqcup \widetilde{I}_2\sqcup \ldots \sqcup \widetilde{I}_k$
where $\widetilde{I}_j=[a_j,b_j+1]$ for all $j$.

Arguing as in the proof of Theorem~\ref{th:expansion_f}, we have that $\psi^c_I =0$ unless $\supp(c)\subseteq \widetilde{I}$.

\begin{proposition}
Let $c,I$ satisfy $\supp(c)\subseteq \widetilde{I}$ and ${\sum_{i\in \widetilde{I}_j}c_i=|I_j|}$ for $j=1,\dots,k$. Then
\begin{align}
\label{eq:ds_mon_as_product}
\psi^c_I=\prod_{1\leq j\leq k}\frac{1}{\qfact{|I_j|}}\ds{x_{a_j}^{c_{a_j}}x_{a_j+1}^{c_{a_j+1}}\cdots x_{b_j}^{c_{b_j}}x_{b_j+1}^{c_{b_j+1}}}_{\{a_j,\dots,b_j+1\}}^q.
\end{align}
\end{proposition}
Note that all terms in the product can be calculated by Proposition~\ref{prop:qDS for monomials}. If $\sum_{i\in \widetilde{I}_j}c_i=|I_j|$ fails for a certain $j$, then $\psi^c_I =0$ .\smallskip

{\em We now restrict to the case where $\supp(c)$ is an interval}. The reason for this will become clear soon. Without loss of generality we may assume that $\supp(c)=[k]$ for some positive integer $k$, and suppose $|c|=r$.

By the analysis above, the only nonzero coefficients $\psi^c_I$  occur when $\widetilde{I}$ is an interval containing $[k]$. A direct analysis reveals the following two scenarios are the only possible ones:
\begin{enumerate}
\item $\widetilde{I}=\widetilde{I}_1$ is an interval of length $r+1$ containing $[k]$. Write $I=[1,r]-a$ for some $0\leq a\leq r-k+1$. Then we immediately get from \eqref{eq:ds_mon_as_product} that
\begin{align}
\label{eq:psi_c_single}
\qfact{r}\psi_I^c=\ds{\alpx^c}_{[1,1+r]-a}^q.
\end{align}

\item $\widetilde{I}=\widetilde{I}_1\sqcup\widetilde{I}_2$ with $\widetilde{I}_1=[i-L_i,i]$ and $\widetilde{I}_2=[i+1,i+1+R_i]$ for a certain $i\in [1,k-1]$, where $L_i=\sum_{j\leq i}c_j$ and $R_{i}=\sum_{j>i}c_j$. Now \eqref{eq:ds_mon_as_product}
 says
 \begin{align}
 \label{eq:psi_c_double}
 \qfact{r}\psi_I^c=\qbin{r+1}{L_i+1}\ds{x_1^{c_1}\cdots x_{i}^{c_i}}_{[i-L_i,i]}^q \ds{x_{i+1}^{c_{i+1}}\cdots x_k^{c_k}}_{[i+1,i+1+R_i]}^q
 \end{align}
\end{enumerate}
which can be easily evaluated.

\subsection{qDS of fundamental quasisymmetric polynomials}

In this section we will determine $\ds{F_{\alpha}(x_1,\dots,x_m)}_{n}^q$ for $\alpha$ a strong composition of $r$ and $m\leq n$. The case $q=1$ was one of the main results of \cite{DS}; see Remark~\ref{rem:proof_comparison} below.

Recall that $\n(\alpha) =\sum_{1\leq i\leq \ell(\alpha)}(i-1)\alpha_i$. 

\begin{theorem}\label{thm:rho_fundamental}
Let $\alpha\vDash r$, and let $m\geq \ell(\alpha)$.
We have that in $\kly_+$,
\begin{equation}
\label{eq:rho_fundamental}
\projmorph_+(F_{\alpha}(x_1,\dots,x_m))=\frac{q^{\n(\alpha)}}{\qfact{r}}u_{[1,r]+m-\ell(\alpha)}.
\end{equation}
\end{theorem}

In words, the expansion of $F_\alpha(u_1,u_2-u_1,\dots,u_m-u_1)$ in the basis $\mcbkp$ of $\kly_+$ consists of the single term $\frac{q^{\n(\alpha)}}{\qfact{r}}u_{[1,r]+m-\ell(\alpha)}$. Let us point out that the complete expansion in $\kly$ has many more terms in general. It can be described explicitly, but we will not need it here.

\begin{proof}
Let $k\coloneqq \ell(\alpha)$. We will prove the claim by induction on $m-k$.

 If $m=k$, then $F_{\alpha}(x_1,\dots,x_k)$ is the monomial $x^\alpha$.
As $\alpha$ is a strong composition, $\supp(\alpha)$ is connected and so the expansion of $\projmorph(x^\alpha)=(u_1-u_0)^{\alpha_1}\dots (u_r-u_{r-1})^{\alpha_{k}}$ in the $\mcbk$ basis is given by \eqref{eq:psi_c_single} and \eqref{eq:psi_c_double}.
To get the expansion for $\projmorph_+(x^\alpha)$, we retain the basis elements in $\mcbkp$. Only $u_{[r]}$ survives, corresponding to the case $a=0$ in \eqref{eq:psi_c_single}, so that
 it comes with the coefficient $\psi_{[r]}^\alpha=\ds{\alpx^c}_{r+1}^q/{\qfact{r}}$. Now $\ds{\alpx^c}_{r+1}^q$ equals $q^{\n(\alpha)}$ by Proposition~\ref{prop:qDS for monomials} as wanted. The claim is true for $m=k$.\smallskip

So we turn our attention to the case $m>k$.
Observe the following well-known recursion:
	\begin{align*}
	F_{\alpha}(x_1,\dots,x_m)=\left \lbrace \begin{array}{ll}F_{(\alpha_1,\dots,\alpha_k)}(x_1,\dots,x_{m-1})+x_mF_{(\alpha_1,\dots,\alpha_{k-1},\alpha_k-1)}(x_1,\dots,x_m) & \alpha_k>1,\\
	F_{(\alpha_1,\dots,\alpha_k)}(x_1,\dots,x_{m-1})+x_mF_{(\alpha_1,\dots,\alpha_{k-1})}(x_1,\dots,x_{m-1}) & \alpha_k=1.\end{array}\right.
	\end{align*}
We induct on $m+|\alpha|$.
Let $\alpha'$ denote the composition $(\alpha_1,\dots,\alpha_k-1)$ when $\alpha_k>1$ and $(\alpha_1,\dots,\alpha_{k-1})$ if $\alpha_k=1$. In either case, applying $\projmorph_+$ to the recursion above, and using the inductive hypothesis gives
\begin{align}
\projmorph_+(\qfact{r}F_{\alpha}(x_1,\dots,x_m))=q^{\n(\alpha)}u_{[1,r]+m-k-1}+q^{\n(\alpha')}\qint{r}x_mu_{[1,r-1]+m-k},
\end{align}
which by Corollary \ref{cor:mult_by_xi} becomes
\begin{align}
\projmorph_+(\qfact{r}F_{\alpha}(x_1,\dots,x_m))=q^{\n(\alpha)}u_{[1,r]+m-k-1}+q^{\n(\alpha')}q^{k-1}\left(u_{[1,r]+m-k}-u_{[1,r]+m-k-1}\right).
\end{align}
The inductive step is completed by noting that $q^{k-1}q^{\n(\alpha')}=q^{\n(\alpha)}$.
\end{proof}
We obtain  the following consequence generalizing \cite[Theorem 1.1]{DS} which is the case $q=1$.

\begin{corollary}
\label{cor:fundamental}
Let $\alpha\vDash n-1$, and let $\ell(\alpha)\leq m\leq n$.
Then we have
\begin{align*}
\ds{F_{\alpha}(x_1,\dots,x_m)}_{n}^q=\left\lbrace \begin{array}{ll}q^{\n(\alpha)}
& \text{if }m=\ell(\alpha);\\
0 & \text{otherwise.}\end{array}\right.
\end{align*}
In particular, for $f$ any quasisymmetric function of degree $n-1$, we have
\begin{align}
  \label{eq:fundamental ps}
\sum_{j\geq 1}f(1,q,\dots,q^{j-1})t^j=\frac{\sum_{m}t^m\ds{f(x_1,\dots,x_m)}_{n}^q}{(t;q)_{n}}.
\end{align}
\end{corollary}
\begin{proof}
	By Theorem~\ref{th:qDS_and_K}, we know that $\ds{F_{\alpha}(x_1,\dots,x_m)}_{n}^q$ equals $\qfact{r}$ times the coefficient of $u_{[r]}$ in $\projmorph_+(F_{\alpha}(x_1,\dots,x_m))$. The first claim of Corollary~\ref{cor:fundamental} then follows  immediately from Theorem~\ref{thm:rho_fundamental}. Now let $f=F_{\alpha}$ and pick $j\geq \ell(\alpha)$. 
The specialization $\spec:i\mapsto \qint{i}$ in Theorem~\ref{thm:rho_fundamental} gives immediately
\begin{align}
F_{\alpha}(1,q,\dots,q^{j-1})=q^{\n(\alpha)}\qbin{j-\ell(\alpha)+r}{r}.
\end{align}
Multiplying by $t^j$ and summing over all $j$, and then appealing to \eqref{eq:cauchy_binomial_theorem}, we get
\begin{align}
\sum_{j\geq 1}F_{\alpha}(1,q,\dots,q^{j-1})t^j=\frac{t^{\ell(\alpha)}q^{\n(\alpha)}}{(t;q)_n}.
\end{align}
This last formula is well-known in the theory of  quasisymmetric functions and $(P,\omega)$-partitions, cf.~\cite[Eq. (2.1)]{Gessel_Survey} of which it is a special case.
The right-hand side is clearly $\frac{t^{\ell(\alpha)}\ds{F_{\alpha}(x_1,\dots,x_{\ell(\alpha)})}_n^q}{(t;q)_n}$.  Linearity implies the claim for general $f$.
 \end{proof}

\begin{remark}
\label{rem:proof_comparison}
The reader should compare our proof of Corollary~\ref{cor:fundamental} with the proof in the case $q=1$ \cite[Proposition 4.1]{DS}. The latter relies upon the (slightly involved) exact computation of the divided symmetrization of monomial quasisymmetric polynomials.
It turns out that the qDS of monomial quasisymmetric polynomials is not as nice or easy.
We do not currently know of ways to arrive at Corollary~\ref{cor:fundamental} outside of the one that we employ \textemdash{} identify (and subsequently leverage) the link between qDS and coefficient extraction in the Klyachko algebra.
\end{remark}

\begin{remark}
For a large class of quasisymmetric functions arising `naturally' as quasisymmetric functions of posets in the theory of P-partitions, our previous result allows us to assign a combinatorial meaning to their qDS generalizing that in \cite[\S 4.3]{DS}.
\end{remark}

\subsection{The work of Aval--Bergeron--Bergeron}
\label{sub:ABB}
Recall that $\qsym_n^{+}\subset\mathbf{k}[\alpx_n]$ is the ideal generated by quasisymmetric polynomials without constant term.
We claim that $\ds{f}_n^q$ is equal to $0$ if $f\in \qsym_n^{+}$ has degree $n-1$.
In fact we have the following stronger statement.

\begin{corollary}
\label{cor:rho vanishes on qsym_+}
For any polynomial $f$ in $\qsym_n^{+}$, we have
\[
\projmorph_{n}(f)=0.
\]
In particular, if $f$ has degree $n-1$ then $\ds{f}_n^q=0$.
\end{corollary}
\begin{proof}
Since $\projmorph_n$ is an algebra morphism, it suffices to verify this for a quasisymmetric polynomial $f$, and even for $f=F_{\alpha}(x_1,\dots,x_n)$ for some nonempty strong composition $\alpha$. Setting $m=n$ in the statement of Theorem \ref{thm:rho_fundamental} immediately implies the first claim. The second claim then follows immediately from Theorem~\ref{th:qDS_and_K}.
\end{proof}

Since $\projmorph_n$ vanishes on $\qsym_n^{+}$, one is naturally lead to consider the quotient $\bQ[\alpx_n]/\qsym_n^+$.
A monomial basis for this quotient was computed in Aval--Bergeron--Bergeron \cite[Theorem 4.1]{ABB04} and fortunately it fits nicely in our context.

Consider the set  
  \[
    \abb_n=\{(c_1,\dots,c_n)\suchthat \sum_{i\leq j\leq n }c_{j}\leq n-i \text{ for all } 1\leq i\leq n\}.
  \]
    Let $L_n=\mathrm{Vect}\left(\alpx^{c}\suchthat {c}\in \abb_n \right)$ be the linear span of the corresponding monomials.
 For instance, if $n=3$, then $\abb_3=\{(0,0,0),(1,0,0),(0,1,0),(1,1,0), (2,0,0)\}$ and therefore $\{1,x_1,x_2,x_1x_2,x_1^2\}$ is a basis for $L_3$.
The results in \cite{ABB04} imply the vector space decomposition
\begin{align}
\label{eq:poly_decomposition}
\bQ[\alpx_n]=L_n\oplus \qsym_n^+.
\end{align}

\begin{theorem}\label{thm:abb_qds}
If $f\in\mathbf{k}[\alpx_n]$, homogeneous of degree $n-1$, is written $f=g+h$ with $g\in L_n$ and $h\in \qsym_n^+$, then
\[
\ds{f}_{n}^q=g(1,q,q^2,\dots,q^{n-1}).
\]
\end{theorem}
\begin{proof}
Let $g$ and $h$ be as in the statement of the theorem.
We have that $\ds{h}_n^q=0$ by Corollary~\ref{cor:rho vanishes on qsym_+}. So we consider $\ds{g}_n^q$. By linearity it suffices to consider $\ds{\alpx^c}_n^q$ for $c\in \abb_n$ satisfying $|c|=n-1$. We use Proposition~\ref{prop:qDS for monomials}.
Note that  $\sum_{1\leq j\leq i}c_i\geq i$ for $i\in [n]$, which means that the {\L}ukasiewicz path $P_c$ drops below the $x$-axis only in its final step, and thus the descent set $S_c$ is empty. It follows that
\begin{align}
\ds{\alpx^c}_n^q=q^{N(c)}=q^{\sum_{1\leq i\leq n}(i-1)c_i}.
\end{align}
The claim now follows.
\end{proof}
In words, qDS on the degree $n-1$ component of $\mathbf{k}[x_1,\dots,x_n]$ may be obtained by first projecting to the first factor in~\eqref{eq:poly_decomposition} and then computing the principal specialization (sending $x_i$ to $q^{i-1}$) of the resulting polynomial.

\begin{example}\label{ex:show_off_structural}
    From Example~\ref{ex:demo_qds_monomials}, we know $\ds{x_1x_3}_{3}^q=-(1+q)$. We can arrive at the same result differently by noting that
  \[
    x_1x_3=x_1(x_1+x_2+x_3)-(x_1^2+x_1x_2).
  \]
 Since $x_1x_2+x_1^2 \in L_3$ and $x_1(x_1+x_2+x_3)\in \qsym_3^{+}$ Theorem~\ref{thm:abb_qds} tells us that $\ds{x_1x_3}_{3}^q=g(1,q,q^2)$ where $g=-x_1x_2-x_1^2$.
\end{example}

We close this subsection by recording a corollary that naturally leads us to looking at the coinvariant algebra closely.
\begin{corollary}
The map $\projmorph$ induces a morphism from the coinvariant algebra $\mathbf{k}[\alpx_n]/\sym_n^+$ to $\kly_n$. 
\end{corollary}
\begin{proof}
Since $\sym_n^+\subset \qsym_n^+$, and $\projmorph_n$ vanishes on $\qsym_n^+$ by Corollary~\ref{cor:rho vanishes on qsym_+}, the claim follows.
\end{proof}

This morphism for $q=1$ coincides with the pullback homomorphism $\iota^*$ between the cohomology ring of the flag variety and the $\sgrp_n$-invariant part of the cohomology ring of the permutahedral variety, as detailed in the introduction. A geometric interpretation for generic $q$, or any other value of $q$, is not known to the authors.

\section{Klyachko--Macdonald identity and Chow class of a Deligne--Lusztig variety}
\label{sec:qKM}

We are now ready to return to our primary motivation mentioned in the introduction, which was to understand the relationship between the two following formulas for $w\in \sgrp_n$ of length $\ell$, and $\schub{w}$ is the corresponding Schubert polynomial (cf. \S\S\ref{sub:schubert}).
\begin{align}
\schub{w}(u_1,u_2-u_1,\dots,)&=\frac{1}{\ell}\sum_{\mathbf{a}\in \reduced(w)} u_{a_1}u_{a_2}\cdots u_{a_\ell}\quad\textup{ in } \kly^{1}_n\simeq H^*(\Perm_n,\bQ)^{\sgrp_n}.
\label{eq:kly1}\\
\schub{w}(1,\dots,1))&=\frac{1}{\ell}\sum_{\mathbf{a}\in \reduced(w)} a_1a_2\cdots a_\ell.
\label{eq:macd1}
\end{align}

The first identity is Klyachko's identity (for type $A$) first stated in~\cite{Kly85} and proved in~\cite{Kly95}; since the latter is in Russian, the authors included a slightly simplified version of the proof in~\cite[\S8]{NT20}. The second identity  is Macdonald's reduced word identity \cite{Macdonald}.\smallskip

Before we deal with the $q$-version and prove Theorem~\ref{th:main_4} which encompasses both identities, let us explain that it is possible to deduce \eqref{eq:macd1} from \eqref{eq:kly1}. It is not as straightforward as specializing $u_i\mapsto 1$ in \eqref{eq:kly1}, since the identity holds in $\kly_n^1$ and not in $\kly^1$ or $\kly_+^1$, cf. \S\ref{sub:specialization}. In fact both sides in \eqref{eq:kly1} vanish for $\ell\geq n$ since the top graded dimension of $\kly_n^1$ is $n-1$. 
It should not come as a surprise since Klyachko's identity actually holds for any representative of the Schubert classes. Now Schubert polynomials have the important stability property under the inclusions $\sgrp_n\subseteq\sgrp_{n+1}$, and are thus defined for $w\in\sgrp_\infty$. So pick $w\in\sgrp_\infty$, and let $n_0$ be such that $w\in\sgrp_{n_0}$. Then \eqref{eq:kly1} is valid for all $n\geq n_0$, and thus holds in $\kly^1_+$. The specialization $u_i\mapsto 1$ is now possible and thus we obtain~\eqref{eq:macd1}.\smallskip

The methodology of proof of $\eqref{eq:kly1}$ cannot be simply extended to the $q$-case: it is based on the theory of root systems, and it is not clear how the parameter $q$ fits into this story. Also the inductive nature of this proof does not mix well with the $\comaj$ statistic. We will instead prove Theorem~\ref{th:main_4} using the Fomin-Stanley approach to the $q$-Macdonald identity. The subsequent remainder of this section involves applications to certain coefficients $a_w(q)$ related to a particular Deligne--Lusztig variety.

\subsection{A $q$-Klyachko--Macdonald identity}
\label{sub:qKM}

We define the \emph{NilCoxeter algebra} $NC_n$ as the free associative algebra over $\bC$ generated by $v_1,\dots,v_{n-1}$ subject to $v_i^2=0$, $v_iv_j=v_jv_i$ for $|i-j|\geq 2$ and $v_{i}v_{i+1}v_{i}=v_{i+1}v_iv_{i+1}$ for $1\leq i\leq n-2$. Given these relations, it follows that $NC_n$ has a linear basis comprising of elements $v_{w}$ indexed by $w\in \sgrp_n$. The NilCoxeter algebra plays a crucial role in Fomin and Stanley's proof of the $q$-analogue of Macdonald's reduced word identity \cite{FS94}. As we demonstrate next, we are able to adapt their proof\footnote{
As a warning to the reader, Fomin and Stanley use $u$-variables to denote the generators for the NilCoxeter algebra. We use $v$ instead since the $u$-variables are the generators for $\kly$.} to incorporate the Klyachko algebra in the setup.

Let $\schub{w}(x_1,\ldots,x_{n-1})$ be the Schubert polynomial associated with $w\in\sgrp_n$. Let $\schub{n}$ be defined by
\[\schub{n}(x_1,\ldots,x_{n-1})=\sum_{w\in\sgrp_n}\schub{w}v_w \quad\in \bQ[x_1,\dots,x_{n-1}]\otimes NC_n.\]
 For $i=1,\dots,n-1$, define $h_i(x)=1+xv_i$  where $x$ is any parameter. Following~\cite[\S 2]{FS94}, we have the equality
\begin{align}
\label{eq:fs_bjs}
\schub{n}(x_1,\ldots,x_{n-1})=\prod_{i=1}^{n-1}\prod_{j=n-1}^ih_j(x_i),
\end{align}
which is in fact equivalent to the Billey-Jockusch-Stanley formula \cite{Bil93}. Since the $v_i$ do not commute, we emphasize how to interpret the product on the right hand side. The inner product is multiplied from left to right, whereas the order is immaterial as far as the outer product is concerned.
For instance, here is $\schub{3}(x_1,x_2)$.
\[
\schub{3}(x_1,x_2)=(h_2(x_1)h_1(x_1) )\cdot h_2(x_2)=(1+x_1v_2)(1+x_1v_1) \cdot (1+x_2v_2)
\]

Using the fact that $h_i(x)h_j(y)=h_j(y)h_i(x)$ for $j-i\geq 2$ \cite[Lemma 3.1 (i)]{FS94}, we may rewrite \eqref{eq:fs_bjs} as
\begin{equation}\label{eq:rec_schubert}
\schub{n}(x_1,\ldots,x_{n-1})=
\schub{n-1}(x_1,\ldots,x_{n-2})_{|v_j\mapsto v_{j+1}}\prod_{i=1}^{n-1}h_i(x_i),
\end{equation}
where the terms in the product on the right-hand side are multiplied from left to right. So, for instance
\[
\schub{4}(x_1,x_2,x_3)= \schub{3}(x_1,x_{2})_{|v_j\mapsto v_{j+1}} h_1(x_1)h_2(x_2)h_3(x_3).
\]

The result in Theorem~\ref{th:qKM} is central to this section. The reader should compare it  to \cite[Lemma 5.3]{FS94} as it foreshadows our application of principal specialization to the statement below.
To establish this result, we will need the following lemma:

\begin{lemma}[`Yang-Baxter relation']
\label{lemma:Yang_Baxter}
Let $\mathcal{A}$ be a commutative algebra and $a,b,c$ three elements in $\mathcal{A}$ that satisfy  $c(a+c-b)=0$. Then for any $i=1,\ldots, n-2$, we have in $\mathcal{A}\otimes NC_n$
\begin{equation}
\label{eq:Yang_Baxter}
h_{i}(a)h_{i+1}(b)h_{i}(c)=h_{i+1}(c)h_{i}(a+c)h_{i+1}(b-c).
\end{equation}
\end{lemma}
\begin{proof}
We can pick $i$=1 without loss of generality. By expanding the expressions on both sides, we must compare the coefficients of $v_w$ for all $6$ elements of $\sgrp_3$. These are easily checked to coincide in all cases, using the assumed relation $c(a+c-b)=0$. 
\end{proof}

If $\mathcal{A}$ is an integral domain, then either $c=0$ in which case \eqref{eq:Yang_Baxter} is trivial, or $c=a+b$ and then this becomes precisely the case used in~\cite{FS94}.

\begin{theorem}
\label{th:qKM}
In $\kly_+\otimes NC_n$, there holds:
\begin{equation}
\label{eq:schubert_nilcoxeter_factorization}
\schub{n}(u_1,u_2-u_1,\ldots,u_{n-1}-u_{n-2})=\prod_{k=\infty}^1\prod_{j=n-1}^1h_j\left(q^{k-1}(1-q)u_j\right).
\end{equation}
For any $w\in \sgrp_n$ one has the identity in $\kly_+$
\begin{align}
\label{eq:qKM}
\schub{w}(u_1,u_2-u_1,u_3-u_2,\dots)=\frac{1}{\qfact{\ell}}\sum_{\mathbf{a}\in \reduced(w)} q^{\comaj(\mathbf{a})}u_{a_1}u_{a_2}\cdots u_{a_\ell},
\end{align}
where $\ell$ is the length of $w$.
\end{theorem}

These identities do not hold in $\kly$. Consider $\schub{231}=x_1x_2$. The left hand side has a term $-u_{\{0,2\}}$ while the right hand side does not.

Applying the specialization $\spec$ to \eqref{eq:qKM} gives the $q$-Macdonald reduced word identity:

\begin{corollary}\label{cor:specialization_qKM}
\begin{align}
\label{eq:qM}
\schub{w}(1,q,q^2,\dots)&=\frac{1}{\qfact{\ell}}\sum_{\mathbf{a}\in \reduced(w)} q^{\comaj(\mathbf{a})}\qint{a_1}\qint{a_2}\cdots \qint{a_\ell}.
\end{align}
\end{corollary}

\subsection{Chow classes of Deligne--Lusztig varieties}
\label{sub:Kim}

We just saw that Theorem~\ref{th:qKM} lifts the $q$-Macdonald reduced word identity to the algebra
$\kly$.

Here we will see another application of Theorem~\ref{th:qKM}: when $q$ is a prime power, it gives a positive formula for the coefficients in  Schubert cycle expansion of the  class of a certain Deligne--Lusztig variety.
This is a positive characteristic analogue of one of the key observations of \cite{NT20}.

Deligne--Lusztig varieties were introduced by the eponymous authors in the seminal article~\cite{Deligne_Lusztig}. These are fundamental subvarieties of the generalized flag varieties over a finite field, introduced to describe the representation theory of finite groups of Lie type; see \cite{Dudas_Lectures}. In~\cite{Kim20}, Kim studies classes of Deligne--Lusztig varieties in the Chow ring of the flag variety, namely the coinvariant ring $R_n=\bQ[\alpx_n]/\sym_n^+$. 

Here we only focus on type $A$, in the split case. Let us briefly explain Deligne--Lusztig varieties in this case for completeness. 
Let $q$ be a prime power, $\mathbb{F}_q$ denotes the finite field with $q$ elements and $\overline{\mathbb{F}_q}$ an algebraic closure of it. 
Denote by $\mathrm{Frob}:x\mapsto x^q$ the Frobenius map on $\overline{\mathbb{F}_q}$, whose fixed points form $\mathbb{F}_q$. 
Consider the general linear group $G=GL_n(\overline{\mathbb{F}_q})$. Now $\mathrm{Frob}$ acts on $G$ by $(a_{ij})\mapsto (a_{ij}^q)$ so that it has as fixed points $G^F=GL_n(\mathbb{F}_q)$. Let $B\subset G$ be the subgroup of upper triangular matrices and $G/B$ be the flag variety, identified as the set of complete flags $(F_i)_{i=0}^n$ in $\mathbb{F}_q^n$. Let $\sigma\in\sgrp_n$. The \emph{Deligne--Lusztig variety} $X(\sigma)$ is defined as the set of flags $\mc{F}=(F_i)_i$ such that $\mc{F}$ and $\mathrm{Frob}(\mc{F})$ are in {\em relative position $\sigma$}. One way to express this last condition is that there exists a basis $(e_i)_i$ of $\mathbb{F}_q^n$ adapted to the flag $\mc{F}$ (that is, $(e_1,\dots,e_k)$ is a basis of $F_k$ for any $k$) such that $(e_{\sigma(i)})_i$ is adapted to $\mathrm{Frob}(\mc{F})$. Note that $X(\sigma)$ has a natural structure of smooth projective variety.

In the article \cite{Kim20}, D. Kim gives an expression for the Chow class of any Deligne--Lusztig variety (in any type) in terms of Schubert intersection numbers. Let us state the result of interest here, which concerns the standard Deligne--Lusztig variety $\mathrm{Cox}_n=X(\sigma_n)$ with $\sigma_n=s_1s_2\dots s_{n-1}$ the long cycle $(1,2,\dots,n)$. The variety $\mathrm{Cox}_n$ is sometimes called the \emph{Coxeter variety}~\cite[\S 6]{Dudas_Lectures}. We decompose its Chow class $[\mathrm{Cox}_n]$ as an element of $R_n$:
\[[\mathrm{Cox}_n]=\sum_w a_w(q)\schub{w_ow} \mod \sym_n^+,\]
where $w$ runs through all permutations of $\sgrp_n$ of length $n-1$.

 Then \cite[Proposition 6.2]{Kim20}, in the special case of $\sigma_n$, can be rewritten in the form 
 \begin{equation} \label{eq:kim_formula}
 a_w(q)=\ds{\schub{w}}_{r+1}^q,
 \end{equation}
 for any $w\in\sgrpp_{r+1}$, the set of permutations in $\sgrp_{r+1}$ with length $r$.\smallskip

By Theorem~\ref{th:qDS_and_K}, we know that $a_w(q)$ is $\qfact{r}$ times the coefficient of $u_{[1,r]}$ in the squarefree basis expansion of $\schub{w}(u_1,u_2-u_1,u_3-u_2,\dots)$. Together with the identity~\eqref{eq:qKM}, this gives us the formula
\begin{align}
\label{eq:awq formula}
 a_w(q)=\frac{1}{\qfact{r}}\sum_{\mathbf{a}\in \reduced(w)} q^{\comaj(\mathbf{a})}A_{\cont({\bf a})}(q),
 \end{align}
where $\cont(\mathbf{a})$ is the sequence of nonnegative integers whose $i$th entry counts the number of $i$s in $\mathbf{a}$ for all $i\geq 1$.
Setting $q=1$ in ~\eqref{eq:awq formula} gives us the quantity $a_w\coloneqq a_w(1)$ which is exactly what is computed in \cite[Theorem 1.1]{NT20}.

\begin{example}
Consider $w=32415\in \sgrpp_5$.
It has three reduced words: $2123$, $1213$, and $1231$. The corresponding contents are: $(1,2,1,0)$, $(2,1,1,0)$, and $(2,1,1,0)$.
We have $A_{(1,2,1,0)}=(1+q)^2(1+q+q^2)$ and $A_{(2,1,1,0)}=(1+q)(1+q+q^2)$. By ~\eqref{eq:awq formula} we get
\begin{align*}
a_w(q)=\frac{1}{\qfact{4}}(q^5(1+q)^2(1+q+q^2)+(q^3+q^4)(1+q)(1+q+q^2))
=q^3.
\end{align*}
\end{example}

The statement of \cite[Theorem 1.1]{NT20} mentions certain symmetries afforded by the $a_w$. It is natural to inquire whether they lift in the presence of the $q$ parameter.
Consider conjugation by the longest word.
From the definition of qDS we have
\[
\ds{f}_{r+1}^q=(-1)^{r}q^{\binom{r}{2}}\ds{\omega f}_{r+1}^{q^{-1}},
\]
where $\omega$ is the involution on $\bQ[\alpx_n]$ mapping $x_i\to x_{n+1-i}$ for $i\in [n]$.
Since $\schub{w}(x_1,\dots,x_{n})=\schub{w_0ww_0}(-x_n,\dots,-x_1)$ modulo $\sym_n^+$, we infer the following
\begin{proposition} 
For any $w\in\sgrpp_{r+1}$,
\[
a_{w_0ww_0}(q)=q^{\binom{r}{2}}a_w(q^{-1}).
\]
\end{proposition}
\noindent This relation can also be easily derived from~\eqref{eq:awq formula}.
The symmetry $a_w=a_{w^{-1}}$ no longer lifts to something obvious in the presence of the $q$. For instance, check that $a_{153264}(q)=q^7 + 2q^6 + 3q^5 + q^4 + q^3$ and $a_{143625}(q)=q^8 + 2q^7 + 2q^6 + 2q^5 + q^4$.

In the two subsections that follow we give $q$-analogues for certain results that hold for $a_w$; see  \cite[\S 5.2 and \S 6]{NT20}. The techniques employed in loc. cit. are robust enough to allow for incorporating the $q$ parameter  easily, and so our exposition will be brief and we refer the reader to \cite{NT20} for further details and any undefined terminology.

\subsection{Properties of $a_w(q)$}
\label{sub:properties of a_w}

In this section we establish two summatory properties of the numbers $a_w$, based on the notion of factorization of a permutation. We borrow notation from \cite[\S]{NT20}, and refer the reader to loc. cit. for details in the interest of brevity.

We can factor $w$ in $\sgrp_n$ uniquely into indecomposable permutations as
\begin{equation}
\label{eq:indecomposables}
w=w_1\times w_2\times\cdots \times w_k,
\end{equation}
where each $w_i$ is an indecomposable permutation in $\sgrp_{m_i}$ for some $m_i>0$.
We say that $w$ is {\em quasiindecomposable} if exactly one $w_i$ is different from $1$, i.e. $w=1^{i}\times u\times1^{j}$ for $u$ indecomposable $\neq 1$ and integers $i,j\geq 0$.
Given this decomposition, the \emph{cyclic shifts} $w^{(1)},\ldots,w^{(k)}$ of $w$ are given by
\begin{equation}
\label{eq:cyclic_shifts}
w^{(i)}=(w_i\times w_{i+1} \cdots \times w_k) \times (w_1\times \cdots \times w_{i-1}).
\end{equation}

These notions are very natural in terms of reduced words: Let the \emph{support} of $w\in \sgrp_n$ be the set of letters in $[n-1]$ that occur in any reduced word for $w$. Then $w$ is indecomposable if and only if it has full support $[n-1]$. It  is quasiindecomposable if its support is an interval in $\mathbb{Z}_{>0}$.

\begin{theorem}[Cyclic Sum Rule]
\label{thm:cyclic_sum}
Let $w\in \sgrpp_n$, and consider its cyclic shifts $w^{(1)},\ldots,w^{(k)}$ defined by way of ~\eqref{eq:indecomposables} and~\eqref{eq:cyclic_shifts}. We have that
\begin{equation}
\label{eq:cyclic_sum}
\sum_{i=1}^k a_{w^{(i)}}(q)= \sum_{{\bf a}\in \reduced(w)}q^{\comaj({\bf a})}.
\end{equation}
\end{theorem}

\begin{proof}
The proof presented in \cite[Theorem 5.6]{NT20} works virtually unchanged. Instead of appealing to the cyclic sum rule for (ordinary) mixed Eulerian numbers, we invoke Proposition~\ref{prop:cyclic sum remixed}.
\end{proof}

\begin{remark}
In his seminal paper on  reduced word enumeration, Stanley \cite{St84} introduced a family of symmetric functions $({\sf F}_w)_{w\in \sgrp_{\infty}}$ that have come to be called \emph{Stanley symmetric functions}. By work of Edelman-Greene \cite{EG87} and Lascoux-Sch\"utzenberger \cite{LS85} we have the Schur-positive expansion ${\sf F}_w=\sum_{\lambda}a_{w\lambda}s_{\lambda}$,
where the $a_{w\lambda}$ may be computed using the \emph{Lascoux-Sch\"utzenberger tree}.
From the Edelman-Greene correspondence one can then show
\[
\sum_{{\bf a}\in \reduced(w)}q^{\comaj({\bf a})}=\sum_{\lambda}a_{w\lambda}\sum_{T\in \mathrm{SYT}(\lambda)} q^{\maj(T)},
\]
which implies the right-hand side in Theorem~\ref{thm:cyclic_sum} may be computed using the major index statistic on tableaux. In the case of $w$ vexillary, we have exactly one nonzero $a_{w\lambda}$ so the right-hand side simplifies enormously.
\end{remark}

\begin{example}
Let $w=53124768\in \sgrpp_8$.
We have $w=53124 \times 21\times 1$, and hence its three cyclic shifts are
$w^{(1)}=53124 \times 21\times 1= 53124768$, $w^{(2)}= 21 \times 1\times 53124=21386457$, and $w^{(3)}= 1\times 53125\times 21=16423587$.
One computes that
\begin{align*}
a_{w^{(1)}}(q)&=q^7 + q^6 + q^5 + q^4 + q^3 + q^2,\nonumber\\
a_{w^{(2)}}(q)&=q^{14} + 2q^{13} + 3q^{12} + 3q^{11} + 3q^{10} + 3q^9 + 3q^8 + 2q^7 + q^6,\nonumber\\
a_{w^{(3)}}(q)&=q^{12} + 2 q^{11} + 4 q^{10} + 5 q^{9} + 6 q^{8} + 5 q^{7} + 5 q^{6} + 4 q^{5} + 3 q^{4} + q^{3}.
\end{align*}
One may now check that $a_{w^{(1)}}(q)+a_{w^{(2)}}(q)+a_{w^{(3)}}(q)$ indeed tracks $\mathrm{comaj}$ over  $\reduced(w)$. Finally we also note that the Stanley symmetric function ${\sf F}_w$ in this case expands as
${\sf F}_w=s_{421}+s_{43}+s_{52}$.
One can sum the major index generating function over tableaux of shape $421$, $43$, and $52$ to obtain $a_{w^{(1)}}(q)+a_{w^{(2)}}(q)+a_{w^{(3)}}(q)$ again.
\end{example}

Our next result considers $a_w(q)$ when $w$ is quasiindecomposable and provides a simple way to compute them in terms of specializations of Schubert polynomials.
For a permutation $v$ of length $\ell$ and $m\geq 0$, consider
\begin{equation}
\label{eq:nu_m}
\nu_v(m)\coloneqq \nu_{1^m\times v}= \schub{1^m\times v}(1,q,q^2,\ldots).
\end{equation}
\begin{theorem}
\label{thm:quasiindecomposable}
Assume that $v\in\sgrp_{p+1}$ is \emph{indecomposable} of length $n-1$.
Define quasiindecomposable permutations $v^{[m]}\in\sgrpp_n$ for $m=0,\ldots,n-p-1$ by $v^{[m]}\coloneqq 1^{m}\times v\times1^{n-p-1-m}$.
Then
one has
\begin{equation}
\label{eq:summation_indecomposable}
\sum_{j\geq 0}\nu_v(j)t^j=\frac{\sum_{m=0}^{n-p-1} a_{v^{[m]}}(q)t^m}{(t;q)_n}.
\end{equation}
\end{theorem}
\begin{proof}
One only needs to follow the script of \cite[Theorem 5.8]{NT20} making necessary changes along the way,  taking into account the fact that the map  sending $i_1\cdots i_{n-1}\mapsto (i_1+m)\cdots (i_{n-1}+m)$ is a comaj-preserving bijection between $\reduced(v)$ and $\reduced(v^{[m]})$ for $m=0,\ldots,n-p-1$.
\end{proof}

\begin{example}
	Consider $n=7$ and $v=4321\in \sgrp_4$ an indecomposable permutation of length $6$.
	We have that $v^{[0]}=4321567$, $v^{[1]}=1543267$, $v^{[2]}=1265437$, and $v^{[3]}=1237654$.
	One may check that
	\[
		\sum_{j\geq 0}\nu_u(j)t^j=\frac{q^4+(q^{8} + 2 q^{7} + 2 q^{6} + 2 q^{5})t+(2 q^{10} + 2 q^{9} + 2 q^{8} + q^{7})t^2+q^{11}t^3}{(t;q)_7}.
	\]
	Take particular note of the fact that coefficients in the numerator on the right hand side are all positive, which is a priori not immediate.
	Theorem~\ref{thm:quasiindecomposable} then tells us that $a_{v^{[0]}}=q^4$, $a_{v^{[1]}}=q^{8} + 2 q^{7} + 2 q^{6} + 2 q^{5}$, $a_{v^{[2]}}=2 q^{10} + 2 q^{9} + 2 q^{8} + q^{7}$, and $a_{v^{[3]}}=q^{11}$.
\end{example}

\subsection{Combinatorial interpretation for $a_w(q)$ in special cases}
\label{sub:special aw}
We discuss two nice classes of permutations\textemdash{} {\L}ukasiewicz permutations and Grassmannian permutations \textemdash{} beginning with the former.

\subsubsection{{\L}ukasiewicz permutations}
\label{subsub:Luk}

In \cite[Definition 6.1]{NT20} we define a {\L}ukasiewicz composition to be a vector $(c_1,\dots,c_{n})$ with weight $n-1$ such that $c_1+\cdots+ c_k \geq k$ for all $1\leq k\leq n-1$. A {\L}ukasiewicz permutation in $\sgrp_n$ is simply one whose code is a {\L}ukasiewicz composition. We recall that the code of $w\in\sgrp_n$ is the vector $(c_1,\dots,c_n)$ where $c_i$ is the number of $j>i$ such that $w_i>w_j$.

\begin{proposition}
\label{prop:aw Lukasiewicz}
For $w$ a {\L}ukasiewicz permutation in $\sgrpp_n$, we have
\[
a_w(q)=\mathfrak{S}_{w}(1,q,\dots,q^{n-1}).
\]
\end{proposition}
\begin{proof}
The proof is similar to that presented in \cite[Theorem 6.5]{NT20}, except that one eventually employs the equality $\ds{x_1^{c_1}\cdots x_n^{c_n}}=q^{\sum_{1\leq i\leq n}(i-1)c_i}$ for $c$ a {\L}ukasiewicz composition to conclude.
\end{proof}

\begin{remark}
\label{rem:dominant cox}
Two classes of permutations in $\sgrpp_n$ that are {\L}ukasiewicz are \emph{dominant} permutations and \emph{Coxeter elements}.
A permutation is dominant if it is $132$-avoiding, which in turn implies that its code can naturally be interpreted as a partition of $n-1$. For dominant $w$, the Schubert polynomial $\mathfrak{S}_w(x_1,\dots,x_n)$ is a single monomial. If we let $\mathrm{code}(w)=(\lambda,0^{n-\ell(\lambda)})$, then $a_w(q)=q^{N(\lambda)}$. Equation~\eqref{eq:awq formula} then implies the equality that $\sum_{\mathbf{a}\in \reduced(w)}q^{\comaj({\bf a})}A_{\cont({\bf a})}=q^{N(\lambda)}\qfact{n-1}$, which is reminiscent of the cyclic sum rule in Proposition~\ref{prop:cyclic sum remixed}.

A Coxeter element in $\sgrpp_{n}$ is a permutation $w$ such that any reduced word for $w$ contains one instance each of $s_1$ through $s_{n-1}$. As described in \S 6.3, we associate $I_w\subset [n-2]$ with any Coxeter element $w$ including $i\in [n-2]$ in $I_w$ if $s_i$ appears before $s_{i+1}$ in any reduced word for $w$. We then have the analogue of \cite[Proposition 6.13]{NT20}:
\[
a_w(q)=\sum_{\substack{\sigma\in \sgrp_{n-1}\\ \Dsc(\sigma)=I_w}}q^{\maj(\sigma^{-1})}
\]
An alternative description involves the major index distribution over all Young tableaux of the ribbon shape associated with $I_w$. 
\end{remark}

\subsubsection{Grassmannian permutations}
A permutation is said to be \emph{Grassmannian} if it has a unique descent. If this descent is in position $m$, we say that the permutation is \emph{$m$-Grassmannian}. The partition $\lambda(w)$ obtained by sorting the entries of the code of $w$ in nonincreasing order (and omitting 0s) is called the \emph{shape} of $w$. For instance, $w=146235$ is $3$-Grassmannian. Since the code $w$ is $(0,2,3,0,0,0)$, we have $\lambda(w)=(3,2)$.
Note that together with the position of the first descent, the shape allows us to reconstruct a unique Grassmannian permutation interpreted as an element of $\sgrp_{\infty}$.

Given a partition $\lambda$, we let $\mathrm{SYT}(\lambda)$ denote the set of standard Young tableaux of shape $\lambda$. Recall that a positive integer $i$ is a descent of $T\in \mathrm{SYT}(\lambda)$ if $i+1$ occupies some row strictly below that occupied by $i$. Let $\mathrm{SYT}(\lambda,d)$ be the subset of $\mathrm{SYT}(\lambda)$ comprising tableaux with exactly $d$ descents.

\begin{proposition}
\label{prop:aw grassmannian}
Let $w\in \sgrpp_n$ be a $m$-Grassmannian permutation.
Then we have
\[
a_w(q)=\sum_{T\in \mathrm{SYT}(\lambda(w),m-1)}q^{\maj(T)}.
\]
\end{proposition}
\begin{proof}
We follow the approach in \cite[Proposition]{NT20} using our results from \S\ref{sec:quasisymmetric}.
Set $\lambda\coloneqq \lambda(w)$.
We infer that
\begin{align}
\label{eq:awq grassmannian nearly}
a_w(q)=\sum_{T\in \mathrm{SYT}(\lambda,m-1)}q^{(m-1)(n-1)-\maj(T)}.
\end{align}
Recall now that Sch\"utzenberger's evacuation map applied to $T$ results in a tableau $\mathrm{evac}(T)$ of the same shape and with the same number of descents. More importantly, if $i$ is a descent in $T$, then $n-1-i$ is a descent in $\mathrm{evac}(T)$. It follows that
$\maj(T)+\maj(\mathrm{evac}(T))=(m-1)(n-1)$
for all $T\in \mathrm{SYT}(\lambda,m-1)$.
As evacuation is an involution on $\mathrm{SYT}(\lambda,m-1)$, the claim follows.
\end{proof}
Incidentally the proof just presented tells us that $a_w(q)$ is a `palindromic' polynomial in $q$ because of the relation $q^{(m-1)(n-1)}a_w(q^{-1})=a_w(q)$.
\begin{example}
\label{ex:3-grassmannian aw}
Consider the $3$-Grassmannian permutation $w=146235$ with $\lambda(w)=(3,2)$.
Then
\[
a_w(q)=q^4+q^5+q^6
\]
from the following three tableaux:
\begin{align*}
\ytableausetup{mathmode,boxsize=1em}
\begin{ytableau}
2 & 4\\
*(blue!40)1 & *(blue!40)3 & 5
\end{ytableau}\hspace{5mm}
\begin{ytableau}
2 & 5\\
*(blue!40)1 & 3 & *(blue!40)4
\end{ytableau}
\hspace{5mm}
\begin{ytableau}
3 & 5\\
1 & *(blue!40)2 & *(blue!40)4
\end{ytableau}.
\end{align*}
\end{example}

We close by using the alternative expression for $a_w(q)$ in \eqref{eq:awq formula} and the preceding proposition to arrive at an evaluation for a class of remixed Eulerian numbers.

 \begin{corollary}
   \label{cor:special_Ac}
 Let $w\in \sgrpp_{r+1}$ be $m$-Grassmannian  of shape $\lambda$. For $1\leq i\leq r$, let $c_i$ be the number of cells in $\lambda $ that have content $i-m$. Then we have
 \[
 A_{(c_1,\dots,c_{r})}(q)=
 q^{-N(\lambda)}\left(\sum_{T\in \mathrm{SYT}(\lambda,m-1)}q^{\maj(T)}\right) \prod_{(i,j)\in \lambda}\qint{h(i,j)},
 \]
 where $h(i, j)= \lambda_i + \lambda_j'
-i - j + 1$ is the hook-length of the cell $(i, j)$ in $\lambda$.
 \end{corollary}
\begin{proof}
The proof follows from the approach in \cite[Corollary 6.18]{NT20}.
\end{proof}

\begin{example}
We consider the same $3$-Grassmannian permutation as in Example~\ref{ex:3-grassmannian aw}. We get that $(c_1,
\dots,c_5)=(0,1,2,1,1)$, and Corollary~\ref{cor:special_Ac} says
\begin{align*}
A_{012110}(q)=q^{-2}(q^4+q^5+q^6)\qint{4}\qint{3}\qint{2}
=q^2(1+q)^2(1+q+q^2)^2(1+q^2).
\end{align*}
\end{example}
\begin{remark}
\label{rem:won't do vexillary}
In \cite[\S 7]{NT20}, the authors consider the case of $w$ a {\em vexillary} permutation. We gave a combinatorial interpretation for $a_w(1)$ in terms of certain special descents in tableaux, and one is led to believe that there should be a way to modify the arguments so that the $q$ can be incorporated. Given the intricacy of the original argument, we do not pursue this matter further in this article as that would lead us astray.
\end{remark}

\section{The interval property}
\label{sec:interval}
We now turn our attention to gaining a better understanding of a curious phenomenon encountered by the authors in \cite{DS,NT20}, which continues to manifest itself in the presence of the additional $q$ parameter.
The careful reader may have noticed a common underlying trait of  equations~\eqref{eq:summation_indecomposable}, \eqref{eq:fundamental ps} and \eqref{eq:u^c principal}.
Broadly speaking, each equation has a generating function tracking principal specializations on the left-hand side, and the right-hand side rewrites this as a rational function wherein the numerator records the qDS of `shifts' of the appropriate family of polynomials.
This resemblance acquires further mystery given that the objects involved \textemdash{} monomials whose exponent vector is a $\bN$-vector with interval support, quasisymmetric polynomials, and Schubert polynomials indexed by quasiindecomposable permutations\textemdash{} do not appear to possess obvious commonalities.
As we demonstrate below, this `coincidence' is explained by a characteristic possessed by the expansions in the basis $\mcbk$ (or $\mcbkp$) when one takes the images of these polynomials under $\proj$.
In doing so, we  also give a satisfying answer to a problem raised by the authors in \cite[\S 8.4]{NT20}.

\begin{definition} \textup{(}Interval property\textup{)}
 A polynomial $f\in\mathbf{k}[\alpz]$ (resp. $\mathbf{k}[\alpz_+]$) is said to possess the \emph{strong (resp. weak) interval property} if the expansion $\displaystyle\sum_{I\subset \bZ}c_Iu_I$ of $\proj(f)$ (resp. $\proj_+(f)$) in the basis $\mcbk$ (resp. $\mcbkp$) satisfies
 \[
 c_I\neq 0 \Longrightarrow \text{$I$ is an interval in $\bZ$}.
 \]
\end{definition}

Note that this is really a property of elements of $\kly$, resp. $\kly_+$; we will only use it for polynomials though, since in this section we use the Klyachko algebras as intermediate steps towards understanding certain formulas involving  principal specializations of polynomials.

 For instance, $z_1^2z_3$ does not possess the strong interval property as $\proj(f)=u_1^3u_3$ expands as $\frac{1}{q+1}(u_{\{1,2,3\}}+qu_{\{0,1,3\}})$ and $\{0,1,3\}$ is not an interval.
On the other hand if we work in $\kly_+$, we see that  $u_1^2u_3=\frac{1}{q+1}u_1u_2u_3$. Therefore, $z_1^2z_3$ possesses the weak interval property.

For brevity, we say that a polynomial is SIP (resp. WIP) if it possesses the  strong (resp. weak) interval property.
Polynomials that are SIP (resp. WIP)  form a vector space which we denote by $\mc{C}$ (resp. $\mc{C}_+$): in fact, using the defining ideals of $\kly$ and $\kly_+$, we have 
\[\mc{C}=\bigoplus_{I\text{ interval }\subset\bZ}\mathbf{k}z_I \oplus \ideal,\quad\mc{C}_+=\bigoplus_{I\text{ interval }\subset\bP}\mathbf{k}z_I \oplus \ideal_+.
\]

We now discuss the families of polynomials that concern us.
For starters note that the discussion after Theorem~\ref{th:expansion_f} already gives us an important family of polynomials in $\mathbf{k}[\alpz]$, whose merit can be measured by the fact that it explains the three motivating examples at the beginning of this section.

\begin{corollary}
\label{cor:interval monomials are SIP}
Let $c$ be an $\bN$-vector with interval support. Then $\alpz^c$ is $\mathrm{SIP}$.
\end{corollary}

There are families of polynomials in $\mathbf{k}[\alpz_+]$ that are WIP which are also relevant as we aim to demonstrate next. It will help us to introduce a property of $\bN$-vectors that generalizes slightly what was referred to as the {\L}ukasiewicz property in \cite{NT20}.

\begin{definition}
\label{defi:quasiLuk}
Let $c=(c_1,c_2,\dots)$ be a nonzero $\bN$-vector with support in $\bP$. Let
\begin{align*}
  m_c\coloneqq \min\{j\geq 1~|~c_j>0\text{ and }c_{j+1}=0\}\text{ and } M_c\coloneqq \max\{j\geq 1~|~c_j>0\}.
\end{align*}
We call $c$ \emph{quasi {\L}ukasiewicz} if
\begin{align}
\label{eq:quasiLuk}
\sum_{i=1}^jc_i\geq j\text{ for }j=m_c,\ldots,M_c.
\end{align}
\end{definition}

\begin{remark} \label{rem:quasiLuk}
Write $\supp(c)$ as a union of maximal intervals $I_j=[a_j,b_j]$, $j=1,\dots,k$ with $b_j$ increasing. Note that $m_c=b_1$ while $M_c=b_k$. Then to verify that $c$ is {\L}ukasiewicz it is easily checked that it suffices to verify Property~\eqref{eq:quasiLuk} for $j=b_1,a_2-1,b_2,a_3-1,b_3,\cdots,a_k-1,b_k$. It follows that, for a given support, this property only depends on the sum of $c$ restricted to each maximal interval in this support.
\end{remark}

This family is a quite natural common generalization of two important special cases:
\begin{itemize}
\item If $c$ is quasi {\L}ukasiewicz with $c_1>0$, then $(c_1,\dots,c_{|c|+1})$ is a {\L}ukasiewicz composition in the sense of \cite{NT20}, that is the inequality~\ref{eq:quasiLuk} holds for $j=1,\dots,|c|$.
\item Any $c$ with interval support that satisfies $c_i=0$ for $i>|c|$ is quasi {\L}ukasiewicz.
\end{itemize}

We transport this notion to permutations:

\begin{definition}
A permutation $w$ is \emph{(quasi) {\L}ukasiewicz} if for all its reduced words $\mathbf{a}$ (equivalently, any of them\footnote{By Matsumoto's property, any two reduced words are related by a sequence of braid/commutation moves. It is easy to see, using for instance Remark~\ref{rem:quasiLuk}, that if $\cont(\mathbf{a})$ is quasi {\L}ukasiewicz, then it remains so after applying a braid/commutation move.}), the content $\cont(\mathbf{a})$ is (quasi) {\L}ukasiewicz.
\end{definition}
For instance, one can check that the permutation $32154\in \sgrp_{5}$ is {\L}ukasiewicz, say by picking the reduced word $\mathbf{a}=4121$ in which case $\cont(\mathbf{a})=(2,1,0,1)$. On the other hand, note that the permutation $1 \times 32154=143265$ is not quasi {\L}ukasiewicz, as the reduced word $\mathbf{a}=5232$ has $\cont(\mathbf{a})=(0,2,1,0,1)$, which does not satisfy the criteria of being quasi {\L}ukasiewicz.

\begin{remark}
In \cite{NT20} and in \S\ref{sec:qKM}, we called a permutation {\L}ukasiewicz if its code is {\L}ukasiewicz. This is not different than the current definition. Indeed, for $w\in \sgrp_n$, consider the reduced word read from its \emph{bottom pipe dream}. From \cite[Appendix A]{NT20} we have that $\cont(\mathbf{a})$ being {\L}ukasiewicz is equivalent to the code of $w$ being {\L}ukasiewicz.
\end{remark}

\begin{theorem}
  \label{th:important interval families}
 The following families of polynomials in $\mathbf{k}[\alpz_+]$ are WIP:
 \begin{enumerate}
 \item \label{it:connected_monomials} $\alpz^c$ for  $c$ quasi {\L}ukasiewicz.
  \item \label{it:connected_quasisymmetric} $\morph_+(P)$ for $P(x_1,\ldots,x_m)$ a quasisymmetric polynomial in $x_1,\ldots,x_m$.
  \item \label{it:connected_schub} $\morph_+(\schub{w})$ for $w$ a quasi {\L}ukasiewicz permutation.
\end{enumerate}
\end{theorem}

\begin{proof}
For \eqref{it:connected_monomials}, one possibility is to use the recurrence~\eqref{eq:recurrence_uc}. Let us give a proof using the probabilistic interpretation from Proposition~\ref{prop:probabilities as coefficients} instead. Let $q>0$. Then one must show that when $I\subset \bZ_+$ is not an interval, the probability $\prob_c(I)$ is zero; equivalently, let us prove that if $I$ is not an interval and $\prob_c(I)>0$, then $I$ contains nonpositive integers. In this case, there exists $j>0$ with $j\notin I$ such that $I$ has both elements smaller and larger than $j$. Because $\prob_c(I)>0$, which means that it is possible for $c$ to stabilize at $I$, one has $c_j=0$ and $|I|=|c|$. Now by \eqref{eq:quasiLuk}, $\sum_{i=1}^{j-1}c_i=\sum_{i=1}^{j}c_i\geq j$. 
So there are at least $j$ particles in the sites $1$ to $j-1$, and by the hypothesis it is possible for all of them to stabilize strictly left of $j$. 
Therefore some must end up at nonpositive sites, which is what we wanted to prove.

For \eqref{it:connected_quasisymmetric}, we appeal to Theorem~ \ref{thm:rho_fundamental} from which it follows immediately. As to \eqref{it:connected_schub}, it follows from the identity~\eqref{eq:qKM}: if $w$ is quasi {\L}ukasiewicz, then each summand on the right-hand side is WIP, as implied by \eqref{it:connected_monomials}, and therefore so is the entire sum.
\end{proof}

\begin{proposition}
\label{prop:weakly_interval_identities}
Consider a homogeneous degree $d$ polynomial $f\in\mathbf{k}[\alpz_+]$ that is WIP, so that
\[\proj_+(f)=\sum_{a\geq 1}f_a u_{[a,a+d-1]}.\]
for some $f_a$ in $\mathbf{k}$. Then we have
\begin{align*}f_a=\ds{f(\underbrace{0,\ldots,0}_{a-1},x_{1},x_{1}+x_{2},\ldots,x_{1}+\cdots+x_{d},0,\ldots)}_{d+1}^q
\end{align*}
 and
\begin{align*}
f(\qint{1},\qint{2},\ldots)=\sum_{a\geq 1}f_a\qbin{a+d-1}{d}.
\end{align*}
Moreover, if $f=\morph_+(g)$ for $g\in \mathbf{k}[\alpx_+]$, then
\begin{align*}
f_a=\ds{g(\underbrace{0,\ldots,0}_{a-1},x_{1},x_{2},\ldots,x_{d+1},0,\ldots)}_{d+1}^q
\end{align*}
and
\begin{align*}
g(1,q,q^2,\ldots)
=\sum_{a\geq 1} f_a\qbin{a+d-1}{d}.
\end{align*}
\end{proposition}
\begin{proof}
The first equality follows from restricting Theorem~\ref{th:expansion_f} to $\kly_+$.
 Indeed we simply have to compute $u_{I}^*(f)$ for $I=[a,a+d-1]$. The details are very similar to those in the discussion immediately after Theorem~\ref{th:expansion_f}, and we omit them.
 The second equality then follows from applying the specialization $\spec$ (cf. \S\S\ref{sub:specialization}) sending $u_i \to\qint{i}$.

For the second expression, it follows
\begin{align*}
g(u_1,u_2-u_1,\ldots,)&=\sum_{a=1}^\infty \ds{g(\underbrace{0,\ldots,0}_{a-1},x_{1},x_{2},\ldots,x_d,-(x_1+\cdots +x_d),0,0,\ldots}_{d+1}\frac{u_{[a,a+d-1]}}{\qfact{d}},
\end{align*}

Here we can replace $-(x_1+\cdots +x_d)$ by $x_{d+1}$, since $x_1+\ldots +x_{d+1}$ is symmetric in $x_1,\ldots,x_{d+1}$.
The second equality follows by the same specialization $\spec$.
\end{proof}

Proposition~\ref{prop:weakly_interval_identities} can thus be applied to any of three examples of Theorem~\ref{th:important interval families}. Note however that the identities that we have obtained for each of these families turned out to be summation formulas over specializations of \emph{multiple} `compatible' members of these families, as mentioned at the beginning of this section: cf. \eqref{eq:u^c principal}, \eqref{eq:fundamental ps}, \eqref{eq:summation_indecomposable}. We now proceed to give a common setting to explain such formulas.

Let $f$ be a formal power series in $(x_i)_{i\in \bZ}$, homogeneous of degree $d\geq 1$, and such that the support of $f$ is bounded above, i.e. there exists an $M$ such that the variables $x_i$ for $i>M$ do not appear in $f$.
Let $\psi_+(f)$ be the element in $\mathbf{k}[\alpx_+]$ obtained by setting $x_i=0$ for all $i\leq 0$.\footnote{This map is denoted $\pi_+$ in \cite{Lam18}, which unfortunately conflicts with our meaning of $\proj_+$.}
We continue to denote the shift map sending $x_i\to x_{i+1}$ by $\tau$. For $i\geq 0$, consider the sequence of polynomials in $\alpx_+$:
\begin{align}
  f^{(i)}&\coloneqq\psi_+(\tau^i(f))=f(\dots,0,x_1,x_2,\dots,x_i;x_{i+1},x_{i+2},\dots),
\end{align}
where the semicolon separates positive and nonpositive indices in $f=f((x_i)_{i\in \bZ})$. Note that the $f^{(i)}$ have automatically the following compatibility property for any $i\geq 0$:
\begin{equation}
\label{eq:compatibility_fi}
f^{(i)}(x_1,x_2,\dots)=f^{(i+1)}(0,x_1,x_2,\dots).
\end{equation} 
We now consider three properties such that such a series can possess:
\begin{enumerate}
  \item \label{item-init} $f^{(0)}\neq 0$ and $x_1$ divides $f^{(0)}$.
  \item \label{item-wip}$\morph_+(f^{(i)})$ is WIP for all $i\geq 0$.
  \item \label{item-backstable}
     $f^{(i)}(x_1,\dots,x_{d+1},0,\dots)\in \qsym^+_{d+1}$  for all $i\geq d+1$.
\end{enumerate}

Note these are really properties of the polynomials $f^{(i)}$. In fact we could have started directly with such a family of polynomials satisfying the compatibility relations \eqref{eq:compatibility_fi}. The series $f$ serves as a nice device to encode such a family, inspired by the back stable Schubert polynomials $\overleftarrow{\mathfrak{S}}_w$ ~\cite{Lam18}, see below. 

Assume $f$ satisfies all three properties. Condition~\eqref{item-wip} says that $\morph_+(f^{(i)})$ is WIP for all $i\geq 0$ and so by Proposition~\ref{prop:weakly_interval_identities} we get
\begin{align}
\label{eq:deets_1}
f^{(i)}(1,q,q^2,\dots)&=\sum_{a\geq 1}\ds{f^{(i)}(\underbrace{0,\ldots,0}_{a-1},x_{1},x_{2},\ldots,x_{d+1},0,\ldots)}_{d+1}^q\qbinom{a+d-1}{d}.
\end{align}
Condition~\eqref{item-init} immediately implies $f^{(0)}(0,x_1,x_2,\dots)=0$, and more generally
\begin{align}
  f^{(i)}(\underbrace{0,\ldots,0}_{i+1},x_{1},x_{2},\ldots)=0.
  \end{align}
  Thus~\eqref{eq:deets_1} can be rewritten as
  \begin{align}
f^{(i)}(1,q,q^2,\dots)&=\sum_{1\leq a\leq i+1}\ds{f^{(i-a+1)}(x_{1},x_{2},\ldots,x_{d+1},0,\ldots)}_{d+1}^q\qbinom{a+d-1}{d}.
\end{align}
By condition~\eqref{item-backstable} we have $\ds{f^{(i-a+1)}(x_{1},x_{2},\ldots,x_{d+1},0,\ldots)}_{d+1}^q=0$ for $i-a+1\geq d+1$ as the argument lands inside $\qsym_{d+1}^+$, using Corollary~\ref{cor:rho vanishes on qsym_+}. 
Reindexing by setting $j=i-a+1$ we get
\begin{align}
\label{eq:deets_2}
  f^{(i)}(1,q,q^2,\dots)=\sum_{j=0}^{\min(i,d)}\ds{f^{(j)}}_{d+1}^q\qbinom{i+d-j}{d}.
\end{align}
If we now multiply both sides of by $t^{i}$ and sum over all $i\geq 0$, by appealing to Cauchy's $q$-binomial theorem \eqref{eq:cauchy_2} we obtain:

\begin{proposition}
  \label{prop:compatible_specializations}
  Suppose $f$ is a power series in $(x_i)_{i\in \bZ}$ with support bounded above and homogeneous of degree $d\geq 1$. Assume further that $f$ meets conditions~\eqref{item-init} through \eqref{item-backstable}. Then we have
  \begin{align*}
    \sum_{i\geq 0}f(1,q,q^2,\dots)t^{i}=\frac{\displaystyle\sum_{0\leq j\leq d}\ds{f^{(j)}}_{d+1}^qt^j}{(t;q)_{d+1}}.
  \end{align*}
\end{proposition}

In light of Proposition~\ref{prop:compatible_specializations} we revisit three important families considered earlier.
\subsubsection*{The connected $y_c$}
Given a strong composition $c=(c_1,\dots,c_k)\vDash d$, consider
\begin{align}
\overleftarrow{y}_c=\prod_{1\leq i\leq k}(\cdots+x_{-1}+x_0+x_1+\cdots+x_i)^{c_i}.
\end{align}
Since $\overleftarrow{y}_c^{(0)}=x_1^{c_1}(x_1+x_2)^{c_1}\cdots(x_1+\cdots+x_k)^{c_k}=y_c$,  condition~\eqref{item-init} is met since $c_1>0$.
Condition~\eqref{item-backstable} follows because for $i\geq d+1$, we have $\overleftarrow{y}_c^{(i)}(x_1,\dots,x_{d+1},0,\dots)=(x_1+\cdots+x_{d+1})^d$ which is symmetric in $x_1,\dots,x_{d+1}$. Finally, condition~\eqref{item-wip} follows from Corollary~\ref{cor:interval monomials are SIP}.

Proposition~\ref{prop:connected_men_gf} now follows immediately from Proposition~\ref{prop:compatible_specializations}.

\subsubsection*{Fundamental quasisymmetric functions}
Let $\tilde{f}$ be a quasisymmetric function in the variables $(x_i)_{i\leq 0}$. Define $f=\tau^{j}(\tilde{f})$ where $j$ is the smallest positive integer so that the resulting $f$ meets condition~\eqref{item-init}.\footnote{It is not hard to check that such a $j$ always exists. Indeed say $\tilde{f}=\sum_{\alpha\vDash d}c_{\alpha}F_{\alpha}(\alpx_{-})$, then $j=\ell(\alpha)$ for $\alpha$ a composition of shortest length such that $c_{\alpha}>0$.} Thus $f$ is a quasisymmetric function in $(x_i)_{i\leq j}$ satisfying Condition~\eqref{item-init}.  Condition~\eqref{item-wip} follows from Theorem~\ref{th:important interval families}.
Finally condition~\eqref{item-backstable} is clear since the $f^{(i)}(x_1,\dots,x_{d+1},0,\dots)$ belong to $\qsym_{d+1}$.
We may now infer Equation~\eqref{eq:fundamental ps} as an consequence of Proposition~\ref{prop:compatible_specializations}.

As an explicit instance, say $\tilde{f}=F_{\alpha}(\alpx_{-})$ for $\alpha=(\alpha_1,\dots,\alpha_k)\vDash d$.
Then $f=\tau^{k}(F_{\alpha}(\alpx_{-}))$.
In particular $f^{(0)}=x_1^{\alpha_1}\cdots x_k^{\alpha_k}$, and more generally $f^{(i)}=F_{\alpha}(x_1,\dots,x_{k+i})$.
The reader can now check that the equality in Proposition~\ref{prop:compatible_specializations} is that in equation~\eqref{eq:fundamental ps} after both sides are divided by $t^{k}$.

\subsubsection*{$\schub{w}$ for $w$ quasiindecomposable}
We now discuss Theorem~\ref{thm:quasiindecomposable}.
To this end, we rely on beautiful recent work of Lam-Lee-Shimozono \cite{Lam18}. To keep the exposition brief, we borrow terminology from loc. cit. without definition.
Assume $w\in \sgrp_{p+1}$ is indecomposable of length $d$. We interpret $w$ as a permutation in $\sgrp_{\bZ}$ by declaring $w(i)=i$ for all $i\in \bZ\setminus[p+1]$.
In \cite[Theorem 3.2]{Lam18}, Lam et al introduce the \emph{back stable Schubert polynomial} $\overleftarrow{\mathfrak{S}}_w$, which is a formal power series analogue of the ordinary Schubert polynomial $\schub{w}$. One of its properties is that $\psi_+(\overleftarrow{\mathfrak{S}}_w)=\schub{w}$.

We take $f=\overleftarrow{\mathfrak{S}}_w$.
By \cite[Proposition 3.7]{Lam18} and after taking images under $\psi_+$, we get
\begin{align}
  f^{(i)}=\schub{1^i\times w}.
\end{align}
We now verify conditions~\eqref{item-init} through \eqref{item-backstable} under the assumption that $w$ is indecomposable.

We have $f^{(0)}=\schub{w}\neq 0$; $w$ is indecomposable, so in particular $w(1)>1$, which in turn implies the $x_1$ divides $f^{(0)}$; say by the Billey-Jockusch-Stanley formula \cite{Bil93}. This verifies condition~\eqref{item-init}. Condition~\eqref{item-wip} follows from Theorem~\ref{th:important interval families}. As to~\eqref{item-backstable}, $\schub{1^i\times w}$ is well known to be symmetric in the variables $x_1,\dots,x_i$ \cite{Macdonald}. Thus for $i\geq d+1$, $\schub{1^i\times w}(x_1,\dots,x_{d+1},0,\dots)$ is a symmetric polynomial in its variables, and thus naturally lands in $\qsym_{d+1}^+$.  
We have thus verified all criteria needed for Proposition~\ref{prop:compatible_specializations}, and  Theorem~\ref{thm:quasiindecomposable} follows.\medskip

It would be interesting to find more families and identities that fit within this framework.

\section{A $q$-analogue of the volume polynomial}
\label{sec:volumes}
We now return to what was Postnikov's original motivation to discuss divided symmetrization and discuss how the parameter $q$ fits into this context.
Given $\lambda=(\lambda_1,\dots,\lambda_n)\in \bR^n$, the \emph{permutahedron} $P_{\lambda}$ is defined to be the convex hull of all points of the form $w\cdot \lambda$ as $w$ ranges over  $\sgrp_n$.
Postnikov \cite[Section 3]{Pos09} establishes that if $\lambda_1\geq \cdots\geq \lambda_n$, the volume $V(\lambda_1,\dots,\lambda_n)$ of $P_{\lambda}$ is given by $\frac{1}{(n-1)!}\ds{(\lambda_1x_1+\cdots+\lambda_nx_n)^{n-1}}_{n}^{q=1}$.
It follows in particular that $V(\lambda)$ is a polynomial in the $\lambda_i$'s.

In light of this, it is reasonable to define the \emph{$q$-volume} of $P_{\lambda}$ as
\begin{align}
  \label{eq:definition q volume}
\mathrm{V}^q(\lambda_1,\dots,\lambda_n)\coloneqq \frac{1}{(n-1)!}\ds{(\lambda_1x_1+\dots+\lambda_nx_n)^{n-1}}_n^q.
\end{align}
We refer to $\mathrm{V}^q(\lambda)$ as the \emph{$q$-volume polynomial} of $P_{\lambda}$ given that it recovers Postnikov's volume polynomial when $q=1$.
The choice of nomenclature is justified further by the expansions that we consider next, which lift Postnikov's results nicely and also serve to unify some of the qDS evaluations from this work. 

By appealing to Proposition~\ref{prop:qDS for monomials}, we can expand  in monomials of the form $\lambda_1^{c_1}\cdots\lambda_n^{c_n}$ as follows:
\begin{align}
\mathrm{V}^q(\lambda_1,\dots,\lambda_n)=\sum_{\substack{c=(c_1,\dots,c_n)\\ |c|=n-1}}\ds{\alpx^c}_n^q \frac{\lambda_1^{c_1}\cdots\lambda_n^{c_n}}{c_1!\cdots c_n!}.
\end{align}
Thus we obtain a generalization to \cite[Theorem 3.2]{Pos09}.

We can also generalize Postnikov's expression \cite[Section 16]{Pos09} for the volume polynomial in terms of mixed Eulerian numbers in a straightforward manner. Set $\mu_i=\lambda_{i}-\lambda_{i+1}$ for $1\leq i\leq n-1$.
In view of Theorem~\ref{th:qDS_and_K}, we may then express $\mathrm{V}^q(\lambda)$  alternatively as:
\begin{align}
\label{eq:vol_q_1}
\mathrm{V}^q(\lambda_1,\dots,\lambda_n)=\widehat{\mathrm{V}}^q(\mu_1,\dots,\mu_{n-1})\coloneqq \frac{\qfact{n-1}}{(n-1)!}\times \Top_n((\mu_1u_1+\dots+\mu_{n-1}u_{n-1})^{n-1}).
\end{align}
Here the coefficient extraction on the right-hand side takes place in $\kly_n$ with the ground field $\mathbf{k}$ containing $\bQ(q)[\mu_1,\ldots,\mu_{n-1}]$.
Once again we may expand $(\mu_1u_1+\dots+\mu_{n-1}u_{n-1})^{n-1}$ in terms of monomials of the form $\mu_1^{c_1}\cdots \mu_{n-1}^{c_{n-1}}$.
From equation~\eqref{eq:vol_q_1} it then follows that
\begin{align}
\label{eq:vol_q_2}
\mathrm{V}^q(\lambda_1,\dots,\lambda_n)&=
 \sum_{\substack{c=(c_1,\dots,c_{n-1})\\ |c|=n-1}}A_{c}(q)\;\frac{\mu_1^{c_1}\cdots \mu_{n-1}^{c_{n-1}}}{c_1!\cdots c_{n-1}!}.
\end{align}

We now proceed to describe how  qDS of Schubert polynomials appears in the context of $q$-volumes. We note that the analogous expression when $q=1$ is not present in \cite{Pos09}.
We will need the notion of (type $A$) degree polynomials introduced in \cite{PS09} and which we refer to for their combinatorial definition in terms of chains in the Bruhat order. Given permutations $v$ and $w$ in $\sgrp_n$, the degree polynomial $\mc{D}_{v,w} (y_1,\dots,y_n)\in \bQ[y_1,\dots,y_n]$ can be defined 
 by the fact that its evaluation at $(\Lambda_1>\dots>\Lambda_n)\in\bZ^n$ is the degree of the Schubert variety $X^w$ in its projective embedding given by the line bundle indexed by $\Lambda$; see \cite[Section 4]{PS09}.

By using \cite[Equation 4.2]{PS09} (essentially due to Chevalley) we have the following equality\footnote{An equality of this sort holds for any pair of dual bases where the duality pairing is as described in \cite[Section 5]{PS09}. Indeed, degree polynomials are duals to Schubert polynomials under that pairing.}
\begin{align}
\label{eq:degree times Schub}
\frac{(\lambda_1x_1+\dots+\lambda_nx_n)^{n-1}}{(n-1)!}=\sum_{w\in\sgrpp_n}\schub{w}(x_1,\dots,x_n)\mc{D}_w(\lambda_1,\dots,\lambda_n) \mod \sym_n^+.
\end{align}
Since qDS vanishes on degree $n-1$ homogeneous polynomials belonging to $
\sym_n^+$, we have
\begin{align}
\mathrm{V}^q(\lambda_1,\dots,\lambda_n)=\sum_{w\in\sgrpp_n}a_w(q)\mc{D}_w(\lambda_1,\dots,\lambda_n).
\end{align}
Thus the $a_w(q)$ arise as coefficients when one expresses the $q$-volume in terms of degree polynomials.

Our results in this section have been largely obtained by algebraic manipulations, and it makes sense to inquire about an interpretation for $q$-volumes that is of a geometric flavor. We give one in \cite{NT21_remixed}, interpreting the exponents of $q$ as statistics on an interesting polyhedral decomposition of the permutahedron.

 \appendix
\section{Proof of Proposition~\ref{prop:qDS for monomials}}
\label{sub:proof_qDS_monomials}
Given $c\in\wc{n}{n-1}$, we need to show
\begin{equation}
\label{eq:qDS for monomials bis}
\ds{x_1^{c_1}\dots x_n^{c_n}}_n^q=(-1)^{|S|} \sum_{\sigma\in X_S}q^{\binom{n-1}{2}-\mathrm{inv}(\sigma)-\sht{c}}
\end{equation}
where $h_i(c)=(\sum_{k\leq i}c_k)-i$ for $i=1,\dots,n$, $\sht{c}=\sum_{i=1}^{n-1} h_i(c)$, $S$ is the set of $i\in\{1,\dots,n-1\}$ such that $h_i<0$, and finally $X_S$ is the set of permutations of $\sgrp_n$ with descent set $S$.\medskip

 Assume first that there exists an index $k\in [n-1]$ such that $h_k<-1$, and pick $k$ to be the largest such index. Since $h_n({c})=-1$, this implies that $h_{k+1}\geq h_k+1$, that is $c_{k+1}\geq 2$.
Consider the move sending ${c}$ to ${d}\coloneqq (c_1,\dots,c_k+1,c_{k+1}-1,\dots, c_n) \in\wc{n}{n-1}$. We have $h_i(d)=h_i$ for $i\neq k$ and $h_k(d)=h_k+1$. Thus we  have $S_{c}=S_{d}$ and $\sht{d}=\sht{c}+1$.

We show that $\ds{\alpx^{{c}}}_{n}=q\ds{\alpx^{{\bf d}}}_{n}$. Indeed
\begin{align}
  \label{eq:equality_under_moves}
  \ds{q\alpx^{{\bf d}}-\alpx^{{c}}}_{n}=\ds{x_1^{c_1}\cdots x_{k}^{c_k}(qx_k-x_{k+1})x_{k+1}^{c_{k+1}-1}\cdots x_n^{c_n}}_{n}.
\end{align}
By our hypothesis that $h_k<-1$ , we know that $\deg(x_1^{c_1}\cdots x_{k}^{c_k})< k-1$. Lemma~\ref{lem:petrov_recursion} implies that the right hand side of \eqref{eq:equality_under_moves} equals $0$, which in turn implies that
$\ds{\alpx^{{c}}}_{n}=q\ds{\alpx^{{\bf d}}}_{n}$.\smallskip

Now assume that there exists an index $k\in [n-1]$ such that $h_k({c'})>0$, and pick $k$ to be the smallest such index. Since $h_0({c'})=0$, this implies that $h_{k-1}(c')+1\leq h_k(c')$, that is $c'_{k}\geq 2$.
Consider the move sending ${c'}$ to a  sequence ${d'}\coloneqq (c'_1,\dots,c'_{k}-1,c'_{k+1}+1,\dots, c'_n)\in\wc{n}{n-1}$. We have $h_i(d)=h_i(c')$ for $i\neq k$ and $h_{k}(d')=h_k(c')-1\geq 0$. Thus we have $S_{c'}=S_{d'}$ and $\sht{d'}=\sht{c'}-1$. Arguing like before allows us to conclude that$\ds{\alpx^{{c'}}}_{n}=q^{-1}\ds{\alpx^{{d'}}}_{n}$.
Repeating this transformation turns ${ c'}$ to a composition ${ c''}$ which satisfies $h_i({c''})\in \{0,-1\}$ for all $i\in [n]$, and we have
\begin{align}
	\ds{\alpx^{{c}}}_{n}=q^{b-a}\ds{\alpx^{{c''}}}_{n}
\end{align}

The upshot of the preceding discussion is that to compute $\ds{\alpx^{c}}_{n}$ it suffices to consider vectors ${c}$ such that  $h_{i}({c})\in \{0,-1\}$ for all $i\in [n]$. The map ${c}\mapsto S_{c}$ restricted to these sequences
is a 1--1 correspondence with subsets $S\subseteq [n-1]$. We write $S\mapsto c(S)$ for the inverse map, and set $\alpx(S)\coloneqq \alpx^{c(S)}$.

For $S\subseteq [n-1]$, consider the family of polynomials
\begin{align}
\beta_n(S)=q^{|S|+1-n}\sum_{\substack{\sigma\in \sgrp_n\\ \des(\sigma)=S}}q^{\binom{n}{2}-\mathrm{inv}(\sigma)}
\end{align}
Our goal now is to prove that
\begin{align}
  \label{eq:monomial_indexed_by_S}
  \ds{\alpx(S)}_{n}=(-1)^{|S|}\beta_n(S)
\end{align}
To this end, we proceed by induction on $n$.
When $n=1$, we have $S=\emptyset$. In this case, we have $\alpx(S)=1$, and both sides of the equality in \eqref{eq:monomial_indexed_by_S} equal $1$.
Let $n\geq 2$ henceforth.
 Assume further that $S\neq [n-1]$, and let 
 $i\notin S$. Define $S_i=S\cap [i]$ and $S^i=\{j\in [n-i]\suchthat j+i\in S\}$. Then Lemma~\ref{lem:petrov_recursion} gives
\begin{align}
\label{eq:recurrence xS}
q\ds{\alpx(S)}_{n}-\ds{\alpx(S\cup \{i\})}_{n}=\qbin{n}{i} \ds{\alpx(S_i)}_{i}
\ds{\alpx(S^i)}_{n-i}.
\end{align}
The numbers $(-1)^{|S|}\beta_n(S)$ also satisfy this identity, because:
\begin{align}
  q\beta_n(S)+\beta_n(S\cup\{i\})=\qbin{n}{i} \beta_i(S_i)
  \beta_{n-i}\left(S^i\right).
\end{align}
This has a simple combinatorial proof: given a permutation corresponding to the left hand side, split its $1$-line notation after position $i$, and standardize both halves so that they become permutations on $[1,i]$ and $[1,n-i]$ respectively. The factor of $\qbin{n}{i}$ comes from inversions that have their left endpoint in $[1,i]$ and the right end point in $[i+1,n]$. Counting these inversions is the same as counting inversions in binary words in $\{0,1\}$ with $i$ 0s and $n-i$ 1s. The factor $q$ comes from the factor $q^{|S|+1-n}$ ultimately.

To conclude, note that \eqref{eq:recurrence xS} determines all values $\ds{\alpx(S)}_{n}$ by induction in terms of the single value $\ds{\alpx([n-1])}_{n}$. Now $\alpx([n-1])=x_2\cdots x_n$, so by the following lemma $\ds{\alpx([n-1])}_{n}=(-1)^{n-1}$. This coincides with $(-1)^{n-1}\beta_n(S)$ so \eqref{eq:monomial_indexed_by_S} holds in this case also and the proof is complete.

\begin{lemma}For $1\leq i\leq n$, we have
\begin{equation}
\label{eq:evalXi}
 \text{If  } X_i\coloneqq \prod_{\substack{1\leq j\leq n\\ j\neq i}}x_j,\text{  then  }\ds{X_i}_n=(-1)^{n-i}q^{\binom{i-1}{2}}\qbin{n-1}{i-1}.
\end{equation}
\end{lemma}
\begin{proof}[Proof of the lemma]
We use induction on $n$. The claim is clearly true when $n=1$, where the empty product is to be interpreted as $1$.
Assume $n\geq 2$ henceforth. For $i=1,\ldots,n-1$, we have
\begin{align*}
\ds{qX_{i+1}-X_i}_{n}&=\ds{(x_1\ldots x_{i-1})(qx_i-x_{i+1})(x_{i+2}\cdots x_{n})}_{n}\\
&=\qbin{n}{i}\ds{x_1\ldots x_{i-1}}_{i}\ds{x_{2}\ldots x_{n-i}}_{n-i}.
\end{align*}
By the inductive hypothesis, we have $\ds{x_1\ldots x_{i-1}}_{i}=q^{\binom{i-1}{2}}$ and $\ds{x_{2}\ldots x_{n-i}}_{n-i}=(-1)^{n-i-1}$.
Therefore
\begin{align}\label{eq:n-1_equations}
  \ds{X_{i+1}}_{n}-\ds{X_i}_n=(-1)^{n-i-1}q^{\binom{i-1}{2}}\qbin{n}{i}.
\end{align}
By summation, this gives~\eqref{eq:evalXi} up to a common additive constant. The proof of the lemma is then complete using $\sum_{i=1}^{n}\ds{X_i}_n=\ds{\sum_{i=1}^{n}X_i}_{n}=0$ since $\sum_{i=1}^{n}X_i$ is symmetric in $x_1,\dots,x_n$, and
$\sum_{i=1}^{n}(-1)^{n-i}q^{\binom{i-1}{2}}\qbin{n-1}{i-1}=0$ by the $q$-binomial theorem in \eqref{eq:cauchy_binomial_theorem}.
\end{proof}
\section{Proof of Theorem~\ref{th:expansion_f}}
\label{app:proof_technical_theorem}

Notice that the expression on the RHS of \eqref{eq:expansion_f} is linear in $f$,  and thus it is enough to consider $f= \alpz^c$ for $c$ an $\bN$-vector with $|c|=d$.

Given $I$ with $|I|=d$, consider $p_c^I\coloneqq u_I^*(\alpz^c)$.
Decompose $I$ as $I_1\sqcup \cdots \sqcup I_k$ where $I_j=[a_j,b_j]$ where $b_{j}\leq a_{j+1}-2$.
It is easily seen that $p_c^I=0$ unless $\supp(c)\subset I$ and $|I_j|=\sum_{i\in I_j}c_i$ for all $j$.
Therefore, assume these conditions hold.

Let $c^{I_j}$ be the $\bN$-vector which agrees with $c$ on the indices determined by the interval $I_j$ and is $0$ everywhere else.
From the relations defining $\kly$ it follows that
\begin{align}\label{eq:kly_arbitrary_coefficient}
p_c^I=\prod_{1\leq j\leq k} p^{I_j}_{c^{I_j}}=\prod_{1\leq j\leq k} \frac{A_{c^{I_j}}}{\qfact{|I_j|}},
\end{align}
where, in arriving at the second equality, we have used (a shifted version of) Theorem~\ref{th:qDS_and_K} to conclude that the coefficient of $u_{I_j}$ in $u^{c^{I_j}}$ is given by $\frac{A_{c^{I_j}}}{\qfact{|I_j|}}$.

To establish the claim we need to show that
\begin{align}\label{eq:qds_arbitrary_coefficient}
p_c^I=\ds{Q_I(\alpx)\prod_{j=1}^kx_{a_j}^{c_{a_j}}(x_{a_j}+x_{a_j+1})^{c_{a_{j+1}}}\cdots (x_{a_j}+\cdots +x_{b_j})^{c_{b_{j}}}}_{\{a_1,\ldots,b_k+1\}}^q\times\frac{1}{\qfact{b_k-a_1+1}}.
\end{align}
It is immediate that the RHS vanishes if $\supp(c)\not\subset I$ as the degree of the argument is too small. So assume that $\supp(c)\subset I$.
By applying Lemma~\ref{lem:petrov_recursion} repeatedly, one may rewrite the RHS as
\begin{align}
\label{eq:to_be_compared}
\prod_{1\leq j\leq k}\left(\ds{x_{a_j}^{c_{a_j}}\cdots (x_{a_j}+\cdots+x_{b_j})^{c_{b_j}}}_{\{a_j,\dots,b_j+1\}}^q \times \frac{1}{\qfact{|I_j|}}\right).
\end{align}
Note that $(c_{a_j},\dots,c_{b_j})$ is nothing but $c^{I_j}$.
Again we see that the right hand sides vanishes if there exists $1\leq j\leq k$ such that $|I_j|\neq |c^{I_j}|$ for degree reasons.
Therefore, assuming that this is not the case, and recalling the definition of $A_{c^{I_j}}$, we see that the right hand in \eqref{eq:qds_arbitrary_coefficient} coincides with the expression for $p_c^I$ from \eqref{eq:kly_arbitrary_coefficient}.
This concludes our proof.
\section{Proof of Theorem \ref{th:qKM}}

Let us deal with the second part first. The identity ~\eqref{eq:qKM} follows from extracting coefficients to the right-hand side of \eqref{eq:schubert_nilcoxeter_factorization}, using  \cite[Lemma 5.4]{FS94}.  More specifically we need to set $t_j=v_j$ and $z_j=(1-q)u_j$ in \cite[Lemma 5.4]{FS94}.\smallskip

So we just need to prove \eqref{eq:schubert_nilcoxeter_factorization}, and for this we follow the steps of \cite{FS94}. The key difference is that we will need to appeal to the new Lemma~\ref{lemma:Yang_Baxter}. 

It is clearly enough to prove that for any $k\geq 0$, we have
\begin{align}
\label{eq:qKM_recursion}
\schub{n}(q^{k}u_1,&q^{k}(u_2-u_1),\ldots,q^{k}(u_{n-1}-u_{n-2}))=\\
&\schub{n}(q^{k+1}u_1,q^{k+1}(u_2-u_1),\ldots,q^{k+1}(u_{n-1}-u_{n-2}))\prod_{j=n-1}^1h_j(q^k(1-q)u_j)\nonumber
\end{align}
So we focus on proving~\eqref{eq:qKM_recursion}.
In fact, it is enough to do the case $k=0$ because the relations among the $u_i$ are homogeneous and are thus verified by the rescaled $q^ku_i$.

The case $n=2$ amounts to checking that $\schub{2}(u_1)=1+u_1v_1$ equals $\schub{2}(qu_1)h_1((1-q)u_1)=(1+qu_1v_1)(1+(1-q)u_1v_1)$. This is true as $v_1^2=0$.

The case $n=3$ can also be checked directly, and this time we need the relations in $\kly_+$.
The left-hand side, by using the relation $(1+q)u_1^2=u_1u_2$,  expands as
\begin{align}
\label{eq:n_3_tedium}
\schub{3}(u_1,u_2-u_1)=1+u_1v_1+u_2v_2+u_1^2v_2v_1+qu_1^2v_1v_2+qu_1^3v_2v_1v_2.
\end{align}
We need to that this equals
$\schub{3}(qu_1,q(u_2-u_1))h_2((1-q)u_2)h_1((1-q)u_1)$, which in turn may be rewritten as
\begin{align*}
h_2(qu_1)h_1(qu_1)\underline{h_2(qu_2-qu_1)h_2(u_2-qu_2)}h_1((1-q)u_1).
\end{align*}
The  underlined terms give $h_2(u_2-qu_1)$, say by \cite[Lemma 3.1 (ii)]{FS94}.
A routine but tedious computation repeatedly employing the relation $u_1^2=u_1(u_2-qu_1)$ in $\kly_+$ shows that the resulting expression collapses to the right-hand side in~\eqref{eq:n_3_tedium}.

\smallskip

Let $n\geq 4$. Let $R_n$ denote the right-hand side of~\eqref{eq:qKM_recursion} for $k=0$. We want to show that $R_n$ equals $\schub{n}(u_1,u_2-u_1,\ldots,u_{n-1}-u_{n-2})$. Declare $u_0=0$.
By using \eqref{eq:rec_schubert}, we have
\begin{align}%
R_n&=
\schub{n-1}(qu_1,\ldots,q(u_{n-2}-u_{n-3}))_{|v_i\mapsto v_{i+1}}\underbrace{\prod_{i=1}^{n-1}h_i(q(u_i-u_{i-1}))\prod_{j=n-1}^1h_j((1-q)u_j)}_{S\coloneqq S(u_1,\dots,u_{n-1})}
\label{eq:Rn brb}
\end{align}
For $i\geq 1$, we furthermore declare
\begin{align*}
\tilde{h}_i\coloneqq h_i(qu_{i}-qu_{i-1}); \quad \quad \tilde{h}'_i\coloneqq h_i((1-q)u_{i}).
\end{align*}
We use \cite[Lemma 3.1 (i)]{FS94} again to rewrite $S$ as follows:
\begin{align}
S&=\tilde{h}_1\cdots \tilde{h}_{n-2} \cdot {\tilde{h}_{n-1}\tilde{h}'_{n-1}}\cdot \tilde{h}'_{n-2}\cdots \tilde{h}'_1\nonumber\\
&=\tilde{h}_1\cdots \tilde{h}_{n-2} \cdot {h_{n-1}(u_{n-1}-qu_{n-2})}\cdot \tilde{h}'_{n-2}\cdots \tilde{h}'_1\nonumber\\
&=\tilde{h}_1\cdots \tilde{h}_{n-3} \cdot \underline{\tilde{h}_{n-2}h_{n-1}(u_{n-1}-qu_{n-2})\tilde{h}'_{n-2}}\cdot \tilde{h}'_{n-3}\cdots \tilde{h}'_1
\label{eq:ready_for_ybe}
\end{align}

For $j\geq 1$ one can verify that $a=qu_j-qu_{j-1}$, $b=u_{j+1}-qu_j$ and $c=(1-q)u_j$ satisfy the assumption of Lemma~\ref{lemma:Yang_Baxter}. Using this Yang--Baxter relation in the underlined factor in the right-hand side of ~\eqref{eq:ready_for_ybe}, with $j=n-2$ and these values of $a,b,c$, gives
\begin{align}
\label{eq:S_factors_via_YBE}
S&= \tilde{h}_1\cdots \tilde{h}_{n-3} \cdot {h_{n-1}((1-q)u_{n-2})h_{n-2}(u_{n-2}-qu_{n-3})h_{n-1}(u_{n-1}-u_{n-2})}\cdot \tilde{h}'_{n-3}\cdots \tilde{h}'_1\nonumber\\
&={h_{n-1}((1-q)u_{n-2})}\tilde{h}_1\cdots \tilde{h}_{n-3}\cdot \underline{h_{n-2}(u_{n-2}-qu_{n-3})}\cdot \tilde{h}'_{n-3}\cdots \tilde{h}'_1{h_{n-1}(u_{n-1}-u_{n-2})}.
\end{align}
Now note that $h_{n-2}(u_{n-2}-qu_{n-3})=\tilde{h}_{n-2}\tilde{h}'_{n-2}$, and so excluding the rightmost and leftmost factors in \eqref{eq:S_factors_via_YBE}, there remains $S(u_1,\dots,u_{n-2})$. Applying this argument repeatedly allows us to conclude by induction that
\begin{align}
\label{eq:S rewritten}
S(u_1,\ldots,u_{n-1})&=\prod_{j=n-1}^{2}h_j((1-q)u_j)\prod_{i=1}^{n-1} h_i(u_i-u_{i-1})
\end{align}

We return to the expression for $R_n$ in \eqref{eq:Rn brb}.
On substituting the expression for $S(u_1,\dots,u_{n-1})$ obtained in \eqref{eq:S rewritten}, we get
\begin{align}
R_n&=
{\schub{n-1}(qu_1,\ldots,q(u_{n-2}-u_{n-3}))_{|v_i\mapsto v_{i+1}}\prod_{j=n-1}^{2}h_j((1-q)u_j)}\prod_{i=1}^{n-1} h_i(u_i-u_{i-1})
\end{align}
Applying the inductive hypothesis to the first two factors, we may rewrite this equality as
\begin{align}
R_n&=
{\schub{n-1}(u_1,u_2-u_1,\ldots,u_{n-2}-u_{n-3})_{|v_i\mapsto v_{i+1}}}\prod_{i=1}^{n-1} h_i(u_i-u_{i-1})
\end{align}
A comparison with \eqref{eq:rec_schubert} implies the claim, and we are done.

\bibliographystyle{hplain}
\bibliography{Biblio_qDS}

\end{document}